\documentclass[reqno,10pt]{amsart}
\usepackage{amscd,amssymb,amsmath,misha,color,url}
\usepackage{stmaryrd}
\usepackage{latexsym}
\usepackage[all]{xy}
\usepackage[colorlinks=true]{hyperref}
\usepackage[dvipsnames]{xcolor}
\usepackage{tikz-cd}
\usepackage{cleveref}

\def\nor{{\rm nor}}
\def\.{,\dots,}

\def\QQ{{\bbQ}}
\def\ZZ{{\bbZ}}

\def\FF{{\bbF}}
\def\PP{{\bbP}}
\def\AA{{\bbA}}
\def\NN{{\bbN}}

\def\RR{{\bbR}}
\def\II{{\bbI}}
\def\DD{{\bbD}}
\def\kcirccirc{{k^{\circcirc}}}
\def\Rad{{\rm Rad}}
\def\wky{{\wt{k(y)}}}
\def\wkx{{\wt{k(x)}}}
\def\wks{{\wt{k(s)}}}
\def\cO{{\calO}}
\def\cR{{\calR}}
\def\cL{{\calL}}

\def\cP{{\calP}}
\def\mon{{\rm mon}}

\def\gp{{\rm gp}}

\begin{document}

\author{Katharina H\"{u}bner and Michael Temkin}
\title{Adic curves: stable reduction, skeletons and metric structure}

\date{\today}


\thanks{The first author acknowledges support by Deutsche Forschungsgemeinschaft (DFG) through the Collaborative Research Centre TRR 326 "Geometry and Arithmetic of Uniformized Structures", project number 444845124, and the second author acknowledges support by ISF Grant 1203/22 and ERC Consolidator Grant 770922 - BirNonArchGeom. The paper was revised when M.T. was visiting MPIM-Bonn.}

\begin{abstract}
We study the structure of adic curves over an affinoid field of arbitrary rank. In particular, quite analogously to Berkovich geometry we classify points on curves, prove a semistable reduction theorem in the version of Ducros' triangulations, define associated curve skeletons and prove that they are deformational retracts in a suitable sense. An important new technical tool is an appropriate compactification of ordered groups that we call the ranger compactification. Intervals of rangers are then used to define metric structures and construct deformational retractions.
\end{abstract}

\maketitle

\section{Introduction}
The goal of this paper is to study basic properties of adic curves, including the topological type, metric structure and relation to semistable reduction. Some of the ideas are not new, because their analogs are well known in the framework of Berkovich geometry, but it seems that this task was never done in the literature. Similarly to the case of Berkovich curves, the fastest route is to explicitly describe $\PP^1_k$ and classify its points, and then use an appropriate version of the stable reduction theorem. However, the theory of formal models is less developed over complete valuation rings of higher height, so it is not so easy to transfer algebraic semistable reduction into adic geometry. We decided to use a local to global approach instead. In particular, because this method is of interest of its own. In Berkovich geometry such a proof of the analytic semistable reduction theorem was established in the second author's MS thesis but never published, though an algebraic version of a similar idea was later worked out in \cite{Temkin-stable}. The starting point is a local uniformization of one-dimensional valued fields, which was also proved in that thesis and was later published in \cite{Temkin-stable} (in fact, similar results were proved independently quite a few times, including works of Matignon and Kuhlmann). The result in the analytic (i.e. completed) setting immediately follows from its algebraic version. Then we deduce a local uniformization of adic curves in terms of elementary neighborhoods analogous to elementary curves in Berkovich geometry, and then a global triangulation theorem follows quite easily. The main technical novelty in these proofs is that we have to work with elementary domains which are only strictly \'etale isomorphic to discs and annuli.

A main new technical tool used in our study is a certain {\em ranger compactification} $\cR_{|k^\times|}$ of the ordered group $|k^\times|$. Intervals in $\cR_{|k^\times|}$ replace real intervals used in Berkovich geometry -- they are edges of skeletons of adic curves and they are used to contract adic curves onto their skeletons. When a first version of this paper was finished we discovered that (probably) for the first time rangers were described by Kuhlmann and Nart in \cite{cuts}. Their motivation was somewhat more algebraic than geometric, but certainly it stems from the same source -- an explicit description of extensions of valuations for $k(t)/k$, that is, the structure of $\PP^1_k$. It might be the case that a similar notion was known to model-theorists as types on ordered groups, but we could not find this in the literature. It is also worth mentioning that in the beginning of this project we used the terminology of cuts analogously to Kuhlmann and Nart. However, it is slightly inconsistent with the classical terminology (e.g. in the case of $\QQ$), and after trying various mutations, like semicuts or cut types, we decided to introduce a new notion -- a ranger.

\subsection{Motivation}
The structure of a Berkovich curve $X$ over an algebraically closed real-valued field $k$ is best reflected by a skeleton $\Delta_\calX\subset X$ associated to a semistable $\kcirc$-model $\calX$ of $X$. It is the dual incidence graph of the closed fiber of $\calX$ provided with a metric that takes into account ``depth'' of the nodes, and sometimes is also decorated with genera (or even residual $\tilk$-curves) at the vertices. Its very useful properties are as follows: $\Delta_\calX$ is a deformational retract of $X$ and its complement is a disjoint union of open discs. In addition, the metric of $\Delta_\calX$ is determined by the sheaf $\Gamma_X$ of logarithmic absolute values of analytic functions restricted to $\Delta_\calX$. In fact, $\Gamma_X|_{\Delta_\calX}$ is the sheaf of pl functions on $\Delta_\calX$ satisfying natural $|k^\times|$-rationality conditions. Finally, the structure of Berkovich curves over non-algebraically closed fields is less understood, but at least one can use descent to define the geometric metric, topological skeletons, and deformational retractions onto skeletons. Concerning the references, curves were studied already in \cite[Chpater~4]{berbook} and \cite[\S3.6]{berihes}, while skeletons in higher dimensions and their pl structure were studied in \cite{bercontr2}.

In the same vein there should exist a notion of skeletons of adic spaces over an affinoid field $k$ with a natural pl structure over the ordered group $|k^\times|$, which can be arbitrary. It seems that such a structure was not defined in the literature, though a somewhat analogous direction was studied in model theory, see \cite[Section~3]{ducros2023tropical} and the references given there, and a toric geometry over such groups was recently developed in \cite{amini2022geometry}. In the current paper we lay forth general definitions that should be used in such a theory and fully explore them in the one-dimensional case. In particular, we classify points and construct geometric metric, skeletons and deformational retractions of adic curves. These results will be used in our subsequent works on wild loci of covers. Skeletons of higher-dimensional adic spaces are certainly worth a detailed study but this will be done elsewhere.

\subsection{Overview of the paper}
We suggest a somewhat new approach to pl geometry which is slightly novel already in the classical case -- it saturates the usual polyhedra with infinitesimal points according to imposed rationality conditions, and these points conveniently bookkeep the information about slopes. One of the advantages of this approach is that it extends on the nose to the case of arbitrary ordered groups and even to families, see~\ref{genrem}. However, the latter will be developed elsewhere and in this paper we only study in detail the one-dimensional pl geometry over ordered groups, see Section \ref{orderedsec}. To a large extent this reduces to describing a certain completion/compactification $\cR_\Gamma$ of an ordered group $\Gamma$, which we call {\em ranger completion}. Its elements contain cuts and infinitesimal elements, see \S\ref{rangersubsec} for a rather ad hoc description. Then we provide $T=\cR_\Gamma$ with the structure sheaf $\Gamma^+_T$ of non-negative $\Gamma$-pl functions and explain that in the one-dimensional case this sheaf induces a natural $\Gamma$-metric. Using intervals in $T$ we introduce $\Gamma$-graphs, which are natural generalizations of usual metric graphs when $\RR$ is replaced by $\Gamma$. Finally, in the end of \S\ref{orderedsec} we provide a general spectral construction, which is an analogue of Huber's spectrum in the world of abelian groups. We observe that $\Spa(\Gamma\oplus \ZZ^n,\Gamma^+)$ should be viewed as the $\Gamma$-pl affine space $\AA^n_\Gamma$, and show that for $n=1$ one obtains the space $(T,\Gamma_T^+)$, as expected.

In Section \ref{curvessec} we start our study of adic curves. We classify points via valuation-theoretic invariants, study specialization and generization relations on adic curves and describe the generic fiber in terms of usual algebraic curves (in the case when the topology on $k$ is discrete) or Berkovich curves (in the case when the topology is given by a microbial valuation).

The main results are obtained in Section~\ref{metricsec}. First, when $k$ is algebraically closed we explicitly describe points on $\PP^1_k$ in Theorem~\ref{classth} pretty similarly to the description of points on the Berkovich affine line in \cite[\S1.4.4]{berbook}. In particular, the set of generalized Gauss valuations in the closed fiber $X=\PP^1_s$ is canonically homeomorphic to $\cR_{|k^\times|}$ and $X$ provided with the sheaf of monoids $\Gamma^+_X=(\calO_X^+\cap\calO_X^\times)/(\calO_X^+)^\times$ is a $|k^\times|$-tree. This yields a natural $\Gamma_k$-metric on $X$ and we extend these results to arbitrary ground fields by providing $X$ with the geometric sheaf of values $\oGamma^+_X$ obtained by descending $\Gamma^+_\oX$ from $\oX=X\wtimes_k k^a$ to $X$, see \S\ref{geomsec}. Next we study the local degree $n_f$ of a morphism of curves. In Berkovich geometry one simply sets $n_f(y)=[\calH(y):\calH(x)]$, where $x=f(y)$, but in adic geometry the residue fields do not have to be henselian and one obtains a meaningful definition only setting $n_f(y)=[k(y)^h:k(x)^h]$. In particular, we prove in Theorem~\ref{semicontth} that this function is semicontinuous. Morphisms of local degree one can be viewed as partial henselizations, we call them (local) $h$-isomorphisms (from henselian) and use them to define h-discs and h-annuli as domains possessing morphisms of local degree one onto discs and annuli, and these are the domains which appear in the semistable reduction theorem for adic curves. We study their basic properties in \S\ref{henselsec}. Then, using local uniformization of valued fields, we provide in Theorem~\ref{locunif} a local description of adic curves over algebraically closed affinoid fields, and globalize this to triangulations of arbitrary curves in Theorems~\ref{triangth}, \ref{curvesekeletonth} (arbitrary ground field) and \ref{simulth} (simultaneous triangulation).

\begin{rem}
Independently of this work a tightly related (if not equivalent in an appropriate setting) notion of Nash isomorphisms was recently introduced by in \cite{Ducros-Loeser} while studying curves over valuation rings via model-theoretic techniques. The importance of this notion is not restricted to curves, but even the immediate motivation of Ducros-Loeser was similar to ours -- such a notion necessarily shows up when one defines triangulations of valuative curves in presence of points with non-henselian residue fields. In particular, it is vacuous in Berkovich geometry, but is necessary in adic geometry when dealing with points of higher height.
\end{rem}

Also, we use the triangulation to describe the pl structure of morphisms in Theorem~\ref{morth}.

\begin{rem}
The simplest way to metrize Berkovich or adic curves over arbitrary ground fields is to descent the metric from the geometric case. This is usually used in Berkovich geometry and this is our method in the paper. Another metric can be defined via the sheaf $\Gamma_X^+$, but it seems that it was not studied so far. At the very least it is clear that the obtained metric can have singularities when the ground field is not defectless, see \cite[Remark~7.6]{ducros2023tropical}.
\end{rem}

\subsection{Conventiones}
Given an abelian group $G$ written multiplicatively by $\Gamma=\log G$ we denote an isomorphic group written additively. The isomorphism is suggestively denoted by $\log\:G\toisom\Gamma$.

We refer to \cite{huber-adic} and \cite{Huber-book} for the definition and basic properties of adic spaces. The reader can also consult the lecture notes of S. Morel, see \cite{Morel}, or of J. Weinstein in \cite{Weinstein-notes}. We will mainly work with adic spaces $X$ over an affinoid field $k$. In such a case we set $S=\Spa(k)$, $\oS=\Spa(\whka)$ and $\oX=X\times_S\oS$. Given a point $x$ on an adic space we will denote its {\em group of absolute values} by $|k(x)^\times|$. Often we will also consider the {\em group of values} $\Gamma_k=\log|k^\times|$.

\section{One-dimensional geometry of ordered groups}\label{orderedsec}
In this section we discuss one-dimensional geometry over ordered groups because such spaces serve as skeletons of adic curves. Our task essentially reduces to describing a certain completion of intervals in ordered groups. We choose a brief but slightly ad hoc exposition in the main paper and just sketch simple proofs, while details and various additional material can be found in the appendix.

\subsection{Ranger completion}\label{rangersec}
In this subsection we provide a direct construction of $\Gamma$-intervals. In a sense, this is a bit of an illustrative warming up section, as after that we suggest a much more general and conceptual construction. The results of this subsection have a large intersection with \cite{cuts} but we decided to keep the exposition self-contained.

\subsubsection{Ordered groups}
In this paper, we will work with ordered groups both in additive and multiplicative notation. The latter will mainly happen when the ordered group shows up as the group of values of a valuation. For an ordered group $\Gamma$, by $\Gamma^+$ we denote the monoid of non-negative elements in the additive notation and the monoid of elements bounded by $1$ in the multiplicative notation.

\subsubsection{Convex subgroups and spectrum}
Let $\Gamma$ be an ordered group. By $S=\Spec(\Gamma^+)$ we denote the set of prime ideals $p\subset\Gamma^+$ (including the empty ideal) provided with the usual Zariski topology. Note, that an ideal $p\subsetneq\Gamma^+$ is prime if and only if it is saturated (i.e. $n\gamma\in p$ implies that $\gamma\in p$). Sending $p$ to $H_p=\Gamma\setminus(p\cup -p)$ one obtains a natural bijection of $S$ onto the set of convex subgroups of $\Gamma$. Alternatively, $H_p$ can be described as the group of units of the localization monoid $\Gamma^+_p$.

\subsubsection{Scales}
We say that non-zero elements $\gamma,\gamma'\in\Gamma^+$ are {\em of the same scale} and write $\gamma\sim\gamma'$ if $\gamma<n\gamma'$ and $\gamma'<n\gamma$ for a large enough $n$. Thus, a {\em scale} of $\Gamma$ is the equivalence class of elements with respect to the above equivalence relation. For example, $\Gamma$ has finitely many scales if and only if it is of finite height $h$, and then $h$ is the number of scales. Also we write $\gamma\ll\gamma'$ (resp. $\gamma\preceq\gamma'$) if the scale of $\gamma$ is strictly smaller than (resp. does not exceed) that of $\gamma'$.

The points of $S$ can be naturally divided into three types: the generic point corresponds to $p=\emptyset$ correspond to scales and their ideals are radicals $\sqrt{(\gamma)}$ of principal ideals, {\em limit points} are all other points and their ideals cannot be represented as radicals of a finitely generated (and hence principal) ideal. The latter can be distinguished further by cofinality (i.e. minimal cardinality of a generating set), but this is not important for any algebraic application. The following result is very simple, so we omit the proof.

\begin{lem}
Let $\Gamma$ be an ordered group and $S=\Spec(\Gamma^+)$. Then a point $s\in S$ is a scale point if and only if it is an immediate specialization of a point $t\in S$. In particular, scales $\gamma$ of $\Gamma$ can be identified with immediate specializations $s\prec t$. Furthermore, $p_s=\{x\in\Gamma^+|\ \gamma\preceq x\}$, $p_t=\{x\in\Gamma^+|\ \gamma\ll x\}$, $H_s=\{\pm x\in\Gamma|\ 0\le x\ll \gamma\}$ and $H_t=\{\pm x\in\Gamma|\ 0\le x\preceq \gamma\}$.
\end{lem}

\subsubsection{Rangers}\label{rangersubsec}
Given a totally ordered group (resp. set) $\Gamma$, by a {\em ranger} on $\Gamma$ we mean an embedding of ordered groups (resp. sets) $\Gamma\into\Gamma'$ and an element $\gamma'\in\Gamma'$. Two rangers are {\em equivalent} if they both contain a third one. Thus, up to an isomorphism each equivalence class contains a unique minimal representative, which we denote $\Gamma'=\Gamma[\gamma']$ because $\Gamma'$ is generated by $\Gamma$ and $\gamma'$ (resp. $\Gamma'=\Gamma\cup\{\gamma'\}$). All rangers of $\Gamma$ form a totally ordered set that will be denoted $\calR(\Gamma)$ or $\cR_\Gamma$. The set of rangers with $\gamma'\in(\Gamma')^+$ will be denoted $\cR_\Gamma^+$. The set of {\em bounded} rangers, obtained by removing the maximal and minimal rangers (see below) will be denoted $\cR^b_\Gamma$.

\begin{rem}
In model theory of ordered groups (resp. sets) a ranger $\gamma$ is nothing else but a $\Gamma$-type, i.e. a type of an element in a theory which contains $\Gamma$ as constants.
\end{rem}

\subsubsection{Classification}\label{classrenagers}
In the obvious way, rangers on a totally ordered set $\Gamma$ can be divided into the following classes, where the numeration is analogous to the numeration of Berkovich and adic points:

\begin{itemize}
\item[(1)] {\em Unbounded rangers}: $-\infty$ and $\infty$ satisfying $-\infty<\gamma<\infty$ for any $\gamma\in\Gamma$.
\item[(2)] {\em Principal rangers}: $\gamma'\in\Gamma$.
\item[(3)] {\em Cut rangers}: a usual cut $\gamma'$ on $\Gamma$ given by a splitting $\Gamma=A^-\coprod A^+$, where $A^-$ (resp. $A^+$) is non-empty and does not contain a maximum (resp. minimum), and $a^-<a^+$ for any $a^-\in A^-$ and $a^+\in A^+$. Then $a^-<\gamma'<a^+$ for any $a^-\in A^-$ and $a^+\in A^+$.
\item[(5)] {\em Infinitesimal rangers}: a successor $\gamma'=\gamma^+$ (resp. a predecessor $\gamma'=\gamma^-$) of a principal ranger $\gamma\in\Gamma$. In other words, $\gamma<\gamma'<\gamma_1$ for any $\gamma_1\in\Gamma$ with $\gamma_1>\gamma$ (resp. $\gamma>\gamma'>\gamma_1$ for any $\gamma_1\in\Gamma$ with $\gamma_1<\gamma$).
\end{itemize}

Note that any non-principal ranger can be described by a splitting $\Gamma=A^-\coprod A^+$, where one of these sets is empty in the unbounded case and has a maximum or minimum in the infinitesimal case. This is all we have to say about rangers on sets, and until the end of \S\ref{rangersec} we will study rangers for ordered groups.

\begin{rem}
In the case of multiplicatively written ordered groups the unbounded rangers will be denoted 0 and $\infty$.
\end{rem}

\subsubsection{The case of groups}
A description of rangers for ordered groups is a bit more complicated because a ranger also determines inequalities like $n\gamma'<\gamma$. But we will see that this is the only novelty and the issue is solved just by switching to the divisible hull $\Gamma_\QQ$ of $\Gamma$.

\begin{lem}\label{ranglem}
For any ordered group $\Gamma$ there are natural bijections of ordered sets $\calR(\Gamma)=\calR(\Gamma_\QQ)=\calR(|\Gamma_\QQ|)$ between the sets of rangers of ordered groups $\Gamma$ and $\Gamma_\QQ$ and the underlying ordered set $|\Gamma_\QQ|$.
\end{lem}
\begin{proof}
Any ranger $\gamma'\in\Gamma'$ on $\Gamma$ induces a ranger $\gamma'\in\Gamma'_\QQ$ on $\Gamma_\QQ$ and hence also a ranger on the underlying ordered set $|\Gamma_\QQ|$, and it is easy to see that these maps preserve the equivalence classes. We will only construct maps in the opposite directions, omitting the simple check that they are the inverses of the above maps.

Any ranger on $\Gamma_\QQ$ tautologically restricts onto $\Gamma$, and any principal ranger on $|\Gamma_\QQ|$ corresponds to a principal ranger on $\Gamma_\QQ$. Finally, given a non-principal ranger $\Gamma'=\Gamma_\QQ\cup\{\gamma'\}$ on $|\Gamma_\QQ|$ one obtains a ranger on $\Gamma_\QQ$ by defining the following total order on the group $\Gamma_\QQ\oplus\ZZ\gamma'$: the inequality $\gamma_1+n\gamma'\le\gamma_2+m\gamma'$ holds if and only if either $m=n$ and $\gamma_1\le\gamma_2$ or $m\ne n$ and $\gamma'\le\frac{1}{n-m}(\gamma_2-\gamma_1)$ in $\Gamma'$.
\end{proof}

\subsubsection{Classification for groups}
A ranger $\gamma'\in\cR_\Gamma$ is called {\em divisorial} if it is a principal ranger of $\Gamma_\QQ$, and then we will freely identify it with an element of $\Gamma_\QQ$. Non-divisorial rangers are divided into three types accordingly to their type as elements of $\cR(\Gamma_\QQ)$ -- infinitesimal, unbounded or cut.

\subsubsection{Rangers when the convex rank is finite}
For illustration we consider the explicit example when the convex rank of $\Gamma$ is finite. It is easily obtained using the classification of rangers, so we omit the justification. By Lemma~\ref{ranglem}, it suffices to consider only divisible groups.

\begin{exam}\label{hexam}
If $\Gamma$ is a divisible ordered group of convex rank $h<\infty$, then it is easy to see by induction that $\Gamma$ (non-canonically) splits as $\Gamma_1\oplus\dots\oplus\Gamma_h$ where each summand is a divisible ordered group of rank one and the total order is the lexicographical one. Fix embeddings of ordered groups $\Gamma_i\into\bbR$, then $\calR_\Gamma$ is naturally identified with the set of sequences $r=(r_1\.r_n)$ with $r_i\in[-\infty,\infty]$ such that the following conditions hold:
\begin{itemize}
\item[(0)] $n\le h+1$.
\item[(1)] If $n=h+1$, then $r_{h+1}\in\{-\infty,\infty\}$.
\item[(2)] If $i<n$, then $r_i\in\Gamma_i$.
\item[(3)] If $r_n\in\Gamma_n$, then $n=h$.
\end{itemize}
In particular, any non-principal ranger ends with an element $r_n$ not contained in $\Gamma_n$ (where we set $\Gamma_{h+1}=1$). Clearly, $(-\infty)$ and $(\infty)$ are the unbounded rangers, a ranger is infinitesimal if and only if $n=h+1$, a ranger is principal if and only if $n=h$ and $r_h\in\Gamma_h$, and a ranger is a cut otherwise.
\end{exam}

\subsubsection{Composition}
One way to justify the above example is by using induction on the convex rank and the following lemma about composition of ordered groups.

\begin{lem}\label{composlem}
Assume that $\Gamma$ contains a convex subgroup $\Gamma_2$ and $\Gamma_1=\Gamma/\Gamma_2$ is the quotient with the induced order, then

(i) There is a natural imbedding $\cR_{\Gamma_2}\into\cR_\Gamma$ with convex image.

(ii) There is a natural map $\cR_\Gamma\to\cR_{\Gamma_1}$. Its fibers over non-divisorial rangers are singletons and its fibers over divisorial ranges are translations of $\cR^b_{\Gamma_2}$ by divisorial rangers. In particular, two rangers $r,r'\in\cR_\Gamma$ are mapped to the same ranger in $\cR_{\Gamma_1}$ if and only if they are contained in an interval $[\gamma,\gamma+\gamma_2]$ with $\gamma\in\Gamma$ and $\gamma_2\in\Gamma_2$.
\end{lem}
\begin{proof}
Passing to the divisible hull of $\Gamma$ we do not change $\cR_\Gamma$, so we can assume that $\Gamma$ and hence also $\Gamma_1,\Gamma_2$ are divisible. The first claim obviously follows from Lemma~\ref{ranglem}, so let us prove the second one. Given a $\Gamma$-ranger $\gamma'\in\Gamma'$ let $\Gamma'_2$ be the convex closure of $\Gamma_2$ in $\Gamma'$. It consists of elements sandwiched between elements of $\Gamma_2$, hence $\Gamma'_2\cap\Gamma=\Gamma_2$ and $\Gamma'_1:=\Gamma'/\Gamma'_2=\Gamma'_1[\gamma']$ is a ranger on $\Gamma_1$. Lemma~\ref{ranglem} implies that any ranger on $\Gamma_2$ can be lifted to a ranger on $\Gamma$, so it remains to study fibers, which are not singletons.

If rangers $r<r'$ on $\Gamma$ are mapped to the same ranger $r_1\in\cR_{\Gamma_1}$, then any divisorial ranger $\gamma\in[r,r']$ is mapped to $r_1$, in particular, $r_1=\gamma_1$ is divisorial. The preimage in $\cR_\Gamma$ of the set $\{\gamma_1^-,\gamma_1,\gamma_1^+\}\subset\cR_{\Gamma_1}$ consists of all rangers $x\in\cR_\Gamma$ such that all divisorial rangers in the interval $[\gamma,x]$ (or $[x,\gamma]$) lie in $\gamma+\Gamma_2$. Hence by claim (i) this preimage is nothing else but the set $\gamma+\cR_{\Gamma_2}$. It remains to note that the $\gamma$-translations of the unbounded rangers $-\infty,\infty$ are mapped to the infinitesimal rangers $\gamma_1^-,\gamma_1^+$, while the whole $\gamma+\cR^b_{\Gamma_2}$ is mapped to $\gamma_2$.
\end{proof}

\begin{rem}
(i) One way to interpret Lemma \ref{composlem} is to say that $\calR_\Gamma$ is obtained from $\cR_{\Gamma_1}$ by inserting $\cR^b_{\Gamma_2}$ instead of each divisorial ranger. In particular, this explains why the set of rangers on an ordered group of convex rank $h$ has a fractal structure with $h$ scales.

(ii) Another way to interpret the lemma is to say that $\cR_{\Gamma_1}$ is obtained from $\cR_\Gamma$ by the following coarsening operation: one forgets the scales contained in $\Gamma_2$ and identifies rangers that become indistinguishable.
\end{rem}

\subsubsection{Ranger completion}
By the {\em ranger completion} or the {\em constructible completion} of $\Gamma$ we mean the set $\cR=\cR_\Gamma$ provided with the topology whose base is formed by the intervals of the form $[a,b]_\cR$ with $a,b\in\Gamma_\QQ$ being divisorial or unbounded rangers but excluding intervals of the form $[a,a]_\cR$ for an unbounded ranger~$a$. In particular, a point $x\in\cR$ is open if and only if $x\in\Gamma_\QQ$.

\begin{theor}\label{segmenttheor}
Let $\Gamma$ be an ordered group, then

(i) The topological space $\cR=\cR_\Gamma$ is quasi-compact, connected and separates into a disjoint union of two connected components by removing any bounded ranger.

(ii) $\Gamma_\QQ$ is contained in~$\cR_\Gamma$ as a dense subset.

(iii) If $x$ is a divisorial ranger, then its closure is $\ox=\{x,x^+,x^-\}$ and the intervals $[-\infty,x^-]=[-\infty,x)$ and $[x^+,\infty]=(x,\infty]$ are closed. If $y$ is not divisorial, then the intervals $[-\infty,y)$ and $(y,\infty]$ are open. In particular, $y$ is closed.

(iv) The maximal Hausdorff quotient $\cR^h$ of $\cR$ is obtained by identifying points inside each closure $\ox$. The induced topology on $\cR^h$ is the usual order topology, and $\cR^h$ is compact and connected.
\end{theor}
\begin{proof}
Assertion (ii) follows directly from the definition of the topology of~$\cR_\Gamma$.
Claim (iii) easily follows from the classification of rangers in \S\ref{classrenagers}, and it implies the last claim of (i) and the description of the Hausdorff quotient $\cR^h$ in (iv). Furthermore, by the same classification $\cR^h$ has no cuts, hence it is connected. So, it is an ordered set with minimal and maximal elements (given by the unbounded rangers) and without cuts, and the usual proof of compactness of a closed interval extends to a proof that $\cR^h$ is compact. Since the fibers of the quotient map $\cR\to\cR^h$ are either singletons or of the form $\{x^-,x,x^+\}$, they are connected and quasi-compact. Therefore $\cR$ is connected and quasi-compact.
\end{proof}

\begin{rem}
The set $\cR^h$ can be identified with the set of all rangers which are not infinitesimal. For instance, in Example~\ref{hexam} one just removes rangers with $n=h+1$. However the quotient topology on $\cR^h$ is generated by the intervals of the form $(a,b)$, $[-\infty,b)$, $(a,\infty]$, and so the embedding $\cR^h\into\cR$ is not continuous.
\end{rem}

\subsubsection{$\Gamma$-segments}
Assume that $\Gamma$ is divisible. By the {\em $\Gamma$-metric} on $\cR$ we mean the distance function $\mu(x,y)=\max(x-y,y-x)\in\Gamma^+$ defined on the set of pairs of open points. By a {\em $\Gamma$-segment} $I$ we mean any interval $[x,y]$, $[x,y)$, $(x,y]$, or $(x,y)$ in $\cR$ provided with the induced topology and $\Gamma$-metric. The segment is {\em constructible} if~$x$ and~$y$ are divisorial.

\begin{rem}
In fact, the $\Gamma$-metric just fixes a homeomorphism of $I$ with an interval in $\cR$ up to a reflection and translation by an element of $\Gamma$. One can use this observation to refine the definition to non-divisible groups, so that non-integrality of open (or rational) points is also taken into account.
\end{rem}

\subsubsection{Scales of a cut}\label{cutscale}
The metric provides additional information about cuts. Without restriction of generality we can assume that $\Gamma$ is divisible. Given a cut $r\in\cR$ on $\Gamma$ consider the ideal $p_r\subset\Gamma^+$ generated by the lengths of constructible segments $I$ containing $r$. If $I$ contains $r$ and is of length $\gamma$, then for any $n\ge 1$ we can divide $I$ into $n$ segments of equal length, and one of them contains $r$. Hence $\frac{1}{n}\gamma\in p_r$ and we obtain that $p_r$ is prime. Each element $\gamma$ induces a translation on $\Gamma$ and hence also on $\cR$, and it is easy to see that $\gamma\in p_r$ if and only if $\gamma+r\neq r$. If $p_r=\sqrt{(\gamma)}$, then we say that the cut is of scale $\gamma$.

\subsubsection{Symmetric and asymmetric cuts}
Let $r$ be a cut. For any constructible segment $I=[\gamma_1,\gamma_2]$ with $r\in I$ either $I_1=[\gamma_1,r)$ or $I_2=(r,\gamma_2]$ contains a constructible interval whose length lies in $p_r$, and hence this interval contains a non-trivial translation of $r$. If this is always the case for both $I_1$ and $I_2$, we say that the cut $r$ is {\em symmetric}.

Otherwise, we call the cut $r$ {\em asymmetric} and say that it {\em points down} (resp. {\em up}) if there exists $\gamma\in\Gamma$ such that $I=[\gamma,r]$ (resp. $I=[r,\gamma]$) does not contain constructible subintervals of length lying in $p_r$, and hence also does not contain translates of $r$. Naturally, we call $I$ the {\em short side} of the cut, and by our convention the cut points towards the short side. For the sake of completeness we note that although $I$ is short it is still longer than any element outside of $p_r$: for any $\delta\in\Gamma^+\setminus p_r$ there exists a constructible subinterval of $I$ of length exceeding $\delta$.

\begin{rem}\label{cutrem}
It is easy to see that a cut $r$ is asymmetric if and only if it adds a new scale, that is, the scale of $r$ in $\Gamma[r]$ is not equal to the scale of any element of $\Gamma$. The latter happens if and only if the map $\Spec(\Gamma[r]^+)\to\Spec(\Gamma^+)$ is not bijective, and then it is bijective on the complement of the scale of $r$.
\end{rem}

\begin{exam}\label{cutexam}
(i) Let $\Gamma$ be of convex rank $h$ as in Example~\ref{hexam} and let $r=(r_1\.r_n)$ be a cut. Then $r$ is of $n$-th scale and it is symmetric if and only if $r_n\in\RR$. It points down if $r_n=\infty$, and it points up if $r_n=-\infty$.

(ii) For any non-generic point $s\in S$ there exist asymmetric cuts $r$ with $p_r=p_s$, for example, $r=\sup(\gamma\in\Gamma^+\setminus p_s)$. On the other hand, it depends on $\Gamma$ whether symmetric cuts $r$ with $p_r=p_s$ exist -- both for scale and limit point. For example, $\prod_{n\in\NN}\bbR$ with scale order $s_0\gg s_1\gg\dots$ has no symmetric cuts both at scales $s_n$ and at the limit point $s_\infty$, its subgroup $\oplus_{n\in\NN}\bbR$ has symmetric cuts only at the limit point, $\prod_{n\in\NN}\bbQ$ has symmetric cuts only at scales $s_n$ and $\oplus_{n\in\NN}\bbQ$ has symmetric cuts at all levels.
\end{exam}

\subsection{$\Gamma$-graphs}

\subsubsection{$\Gamma$-trees}
In appendix A we introduce a notion of a quasi-tree and study its basic notions. The main condition is that a quasi-tree $T$ is a topological space such that any pair of points $a,b$ is connected by a unique quasi-segment $[a,b]$ such that each point $c\in(a,b)$ separates $a$ and $b$ in $T$ in the sense that they lie in different connected components of $T\setminus\{c\}$. By the very definition, a subset $S\subseteq T$ is connected if and only if it is {\em pathwise connected} by quasi-segments.

A {\em $\Gamma$-tree} is a locally connected quasi-tree $T$ provided with a {\em $\Gamma$-metric structure}: each quasi-segment $[a,b]$ is provided with a structure of a $\Gamma$-segment and these $\Gamma$-structures are compatible on intersections of quasi-segments (which are necessarily quasi-segments themselves).

\subsubsection{Classification of points}
There is an obvious topological classification of points on quasi-trees due to the {\em valence} $v(x)$ -- the number of branches (or germs of rays) at a point $x$. Namely, {\em leaf} points, {\em edge} points and {\em vertices} have $v(x)=1$, $v(x)=2$ and $v(x)\ge 3$, respectively. On the other hand, the metric structure induces a classification of its own -- a point is called {\em divisorial} if it is open in one, and hence any, segment containing it. Note that branches at a divisorial point $x$ are parameterized by the infinitesimal points in the closure of $x$. The $\Gamma$-metric structure is completely determined by the distance $\mu(x,y)\in\Gamma^+$ between any two divisorial points. Divisorial points do not have to be (and usually are not) open in $T$ when their valence is infinite. In the case of adic curves, all divisorial points are not open in $T$.

Other points will be named according to their ranger type on one, and hence any, $\Gamma$-interval through them -- cut, infinitesimal and unbounded. Note that an infinitesimal point does not have to have a generization if it is a leaf. Now let us match topological and $\Gamma$-metric types.

\begin{lem}\label{graphtypes}
Let $T$ be a $\Gamma$-tree, then any unbounded point is a leaf and any vertex is a divisorial point. These are the only restrictions: cut and infinitesimal points can be both leaves and edge points, and divisorial points can be of an arbitrary topological type.
\end{lem}
\begin{proof}
Clearly, an unbounded point cannot sit in the middle of a $\Gamma$-segment, hence its valence is one. Assume that $x$ is a vertex, so there exist constructible intervals $[x,a]$, $[x,b]$, $[x,c]$ and pairwise intersections equal to $\{x\}$. In particular, $[a,b]\cap[a,c]=[a,x]$. Shrinking $[x,b]$ or $[x,c]$ we can assume that $\mu(a,b)=\mu(a,c)$ and then switching $b$ and $c$ induces an isometry of $[a,b]\cup[a,c]$ onto itself, which restricts to an autoisometry of $[b,c]$ with $x$ being the fixed point. Thus $x$ is the middle point of $[b,c]$ and hence it is divisorial too.

Finally, these are the only restrictions as can be seen already from the cases of segments in $\cR_\Gamma$ and a non-linear tree obtained by gluing two segments at an inner open point.
\end{proof}

\begin{cor}\label{segmentlem}
If $T$ is a $\Gamma$-tree, and $I,J$ are constructible $\Gamma$-segments in $T$, then $I\cap J$ is either empty, or a divisorial point, or a constructible $\Gamma$-segment.
\end{cor}

\begin{cor}
Any non-divisorial point $x$ in a $\Gamma$-tree $T$ is closed.
\end{cor}
\begin{proof}
Fix any $\Gamma$-segment $I$ containing $x$. If $x$ is an edge point we can also achieve that $x$ is not an endpoint of $I$. Let $\{U_i\}$ be a basis of connected neighborhoods of $x$ and let $W_i=\oU_i$ be their closure. Clearly, $\cap_i(I\cap W_i)=\{x\}$, hence for any $y\neq x$ the interval $I\cap [x,y]$ is not contained in some $W_i$, and by the connectedness of $W_i$ this implies that $y\notin W_i$. Therefore $\{x\}=\cap_i W_i$ is closed.
\end{proof}

\subsubsection{Topology}
The $\Gamma$-metric structure alone suffices to test connectedness of subsets, but it does not determine the topology. In fact, there exist at least two natural topologies one can assign to it: the {\em metric topology}, where a basis of open neighborhoods of $x$ is formed by balls $B(x,r)$ consisting of all points with distance at most $r\in\Gamma$ from $x$, and the {\em weak topology}, whose basis is formed by pathwise connected components of the sets of the form $T\setminus\{x_1\.x_n\}$, where $x_i$ are non-divisorial points. Obviously, the weak topology is the weakest locally connected topology compatible with the $\Gamma$-metric structure, and it is easy to see that only it has a chance to be quasi-compact. We will see that adic curves are provided with the weak topology.

\subsubsection{$\Gamma$-graphs}
A {\em $\Gamma$-graph} is a topological space $G$ which is locally provided with the structure of a $\Gamma$-tree. The notion of $\Gamma$-segment in $G$ is also defined locally. Classification of points is the same as for trees. A $\Gamma$-graph is called {\em finite} if it is a union of finitely many quasi-compact $\Gamma$-segments. Equivalently, it is quasi-compact and has finitely many leaves.

\subsubsection{Skeletons and induced distance function}\label{skeletsec}

By a skeleton of a $\Gamma$-graph $G$ we mean a finite $\Gamma$-subgraph $\Delta\subseteq G$, which can contain isolated points, such that the leaves of $\Delta$ are divisorial and unbounded points, any isolated point of $\Delta$ is divisorial, and for any point $x\in G\setminus\Delta$ there exists a unique $\Gamma$-segment $[x_0,x]$ such that $\{x_0\}=[x_0,x]\cap\Delta$.

\begin{lem}\label{skeletonlem}
Assume that $G$ is a $\Gamma$-graph with a skeleton $\Delta$. For any point $x\in G\setminus\Delta$ with associated segment $[x_0,x]$ as above, $x_0$ is divisorial and identifying $[x_0,x]$ with a $\Gamma$-segment $[0,r_\Delta(x)]$ one obtains a lower semicontinuous function $r_\Delta\:G\to\cR^+_\Gamma$ such that $r_\Delta^{-1}(0)=\Delta$. In fact, $r_\Delta$ measures the distance of points of $G$ from $\Delta$.
\end{lem}

\begin{proof}
If $x_0$ is not divisorial, then its valence in $G$ is at most two and it is even one if $x_0$ is unbounded. The branch at $x_0$ towards $x$ is not in $\Delta$, hence the valence of $x_0$ in $\Delta$ is at most 1 and it is even zero if $x_0$ is unbounded. Both cases are forbidden by our assumptions on the leaves and isolated points of $\Delta$.

To prove the semicontinuity we should show that the preimage of any open interval $[r,\infty]$ is open. If $r=0$, then the preimage is all of $G$, so assume that $r\in\Gamma^+$ is a divisorial positive ranger and $x\in G_{\ge r}=r_\Delta^{-1}([r,\infty])$. Consider the interval $[x_0,x]$ as in (i) and let $y\in[x_0,x]$ be the infinitesimal point with $r_\Delta(y)=r^-$. Consider the open set $U$, which is the connected component of $G\setminus\{y\}$ containing $x$. Clearly, it lies in the connected component $W$ of $G\setminus\Delta$ which contains $y$. Since $W$ is a tree, it follows that for any point $z\in U$ the interval $[z_0,z]$ contains $[x_0,y]$ (in particular, $x_0=z_0)$. Therefore $r_\Delta(z)>r_\Delta(y)=r^-$, that is, $r_\Delta(z)\ge r$, and we obtain that $U$ is a neighborhood of $x$ contained in $G_{\ge r}$.
\end{proof}

\subsubsection{Metric deformation}
Keep the above notation and let us write $r_x=r_\Delta(x)$ for shortness. We identify the interval $[x_0,x]$ with $[0,r_x]\subset\cR_\Gamma$. For any $t\in[0,r_x]$ let $x_t\in[x_0,x]$ denote the point corresponding to $t$. Now we define a map $\Phi=\Phi_\Delta\:\cR_\Gamma^+\times G\to G$ as follows $\Phi(t,x)=x_t$ if $t\le r_x$ and $\Phi(t,x)=x$ otherwise. Our main result about $\Phi$ is that it is continuous. Since $\Phi(\infty,x)=x$ and $\Phi(0,x)=x_0$ this should be interpreted as follows: $\Phi$ is the canonical $\Gamma$-deformational retraction of $G$ onto $\Delta$ induced by the metric. The deformation is parameterized by the interval $\cR^+_\Gamma$, which we will write as $[\infty,0]$ inverting the orientation to make it analogous to the usual notation with the deformation parameter belonging to $[0,1]$. Our notation stems from the fact that we work in the additive setting, so our parametrization should be thought as the $-\log(t)$-parametrization of the usual interval of parameters $[0,1]$.

\begin{theor}\label{metricdeform}
If $G$ is a $\Gamma$-graph and $\Delta$ is a skeleton, then the retraction map $\Phi=\Phi_\Delta\:\cR_\Gamma^+\times G\to G$ is continuous.
\end{theor}
\begin{proof}
We will use the following observation concerning a pathwise connected subset $V \subseteq G$ and a ranger $t\in\cR_\Gamma^+$: if $V$ contains a point $x$ with $r_\Delta(x)\le t$, then $\Phi([\infty,t]\times V)\subseteq V$, and if $V$ does not contain such a point, then $\Phi([\infty,t]\times V)$ is a single point $y_t$ and $r_\Delta(y_t)=t$. Loosely speaking, $[t,\infty]$ deforms $V$ inside itself whenever there is a point $x$ fixed by the action, and once $t$ becomes smaller than the distance from $V$ to $\Delta$, all $V$ is contracted to a single point $y_t$ whose distance from $\Delta$ is $t$. To check this claim it suffices to prove it in the particular case, when $V$ is a $\Gamma$-segment, and this case is established by a simple chasing of a few cases.

Now, let us prove the theorem. Let $y=\Phi(t,x)$ and $U$ a neighborhood of $y$. It suffices to find a neighborhood $[\infty,t']\times V$ of $(t,x)$ mapped to $U$. Shrinking $U$ we can assume in the sequel that it is open and pathwise connected. Since $U$ is a neighborhood of $y_t=y$, there exists a divisorial ranger $t'\le t$ such that $y_{t'}\in U$, and hence $\Phi([\infty,t']\times U)\subseteq U$ by the above paragraph. Thus, if $x\in U$, then we just take $V=U$, but otherwise we will shrink $U$ as follows. Let $W$ be the connected component of $G\setminus\{y_{t'}\}$ that contains $x$. Then $\Phi([\infty,t')\times W)\subseteq W$ and $\Phi(\{t'\}\times W)=y_{t'}$ by the above paragraph, and hence $V=U\cap W$ satisfies $\Phi([\infty,t']\times V)\subseteq V$. It remains to note that $V$ is a neighborhood of $x$.
\end{proof}

\subsubsection{Quotients}

For simplicity, we study quotients of $\Gamma$-graphs under restrictive assumptions which simplify the arguments and cover our needs.

\begin{lem}\label{quotlem}
Assume that a pro-finite group $G$ acts continuously and isometrically on a $\Gamma$-tree $T$ such that there is a $G$-invariant point $q$ and stabilizers of divisorial points are open. Then the quotient $T'=T/G$ is a quasi-tree and the $\Gamma$-metric structure of $T$ descends to $T'$. In fact, if $x,y\in T$ are divisorial points, then their images $x',y'\in T'$ are divisorial, $\mu(x',y')=\mu(Gx,Gy)$, and the projection $p\:T\to T'$ maps $[x,y]$ homeomorphically (and then also isometrically) onto $[x',y']$ if and only if $y$ is the closest point to $x$ in $Gy$, that is, $\mu(x,y)=\mu(Gx,Gy)$.
\end{lem}
\begin{proof}
Provide $T$ with the partial order such that $x\le y$ if and only if $y\in[x,q]$. We view $T$ as a rooted tree, so that the root $q$ is the maximal point, all leaves different from $q$ are the minimal points, and $z=\max(x,y)$ exists for any pair $x,y\in T$ and satisfies $[z,q]=[x,q]\cap[y,q]$. For any $x\in T$ the $G$-orbit of $[x,q]$ is a rooted tree with set of leaves $Gx$. It follows that $I_x=[x,q]$ is mapped bijectively onto its image in $T'$ that will be denoted $I_{x'}=[x',q']$.

Recall that $I_{x'}$ is provided with the quotient topology coming from $GI_x$ and we claim that the map $I_{x}\to I_{x'}$ is a homeomorphism. If $x$ is not a leaf, then $I_x$ is contained in $I_y$ for a divisorial point $y$, and since $Gy$ is finite by our assumption, $Gx$ is finite too. Thus, $GI_x$ is a finite $\Gamma$-tree and our claim is easily verified. It remains to compare the topologies when $x$ is a leaf. We need to construct a system of neighborhoods of~$x$ in~$I_x$ whose images in~$I_{x'}$ are open.
This comes down to finding for any nondivisorial point $y>x$ a neighborhood of $x$ in~$T$ which is disjoint from $GI_y$.
But we know already that~$GI_y$ is a finite graph and since~$y$ is not divisorial, it follows that~$GI_y$ is closed in~$T$.
Its complement is a neighborhood of~$x$ of the required form.

Now, let us define $[x',y']$ for $x',y'\in T'$. Choose any preimages $x,y\in T$ and for $\sigma\in G$ set $z_\sigma=\max(x,\sigma(y))$. By compactness $\inf(z_\sigma)$ is attained for some $\sigma$, and replacing $y$ by $y_\sigma$ we can assume that $z=\max(x,y)\le\max(x,\tau(y))$ for any $\tau\in G$. It follows from the minimality of $z$ that $\{z\}=[x,z]\cap G[y,z]$, and hence $[x,y]$ is mapped bijectively onto its image in $T'$, which will be denoted $[x',y']$. The map $[x,y]\to[x',y']$ is a homeomorphism because its restrictions onto $[x,z]$ and $[y,z]$ are. If $x,y$ are divisorial we use this homeomorphism to induce the $\Gamma$-segment structure on $[x',y']$. It is easily seen to be compatible with intersections. Note that $\mu(x',y')=\mu(x,y)=\mu(x,Gy)=\mu(Gx,Gy)$.

It remains to check that $T'$ is a quasi-tree, that is, any $t'\in(x',y')$ separates $x'$ and $y'$. Let~$t$ be the (unique) preimage of~$t'$ in $(x,y)$. If $t\in(x,z]$, then $t$ lies in each $[x,\sigma(y)]$, hence $t$ separates $x$ from $Gy$ and by equivariance $Gt$ separates $Gy$ and $Gx$ (no connected component of $T\setminus Gt$ intersects both $Gx$ and $Gy$). In the same way, if $t\in(y,z]$, then it separates $y$ from $Gx$ and hence $Gt$ separates $Gx$ and $Gy$. This finishes the proof.
\end{proof}

\subsection{The sheaf of values}
We have defined the structure of $\Gamma$-segments in a full analogy with the usual metrized real intervals, and now we are going to suggest another way to pack this information. It is less natural in the context of usual geometry, but fits valuative geometry much more naturally.

\subsubsection{Valuations on groups and valuation monoids}
By a {\em valuation} on an abelian group $G$ we mean a homomorphism $\nu\:G\to\Gamma$ to an ordered group and its image will be denoted $\Gamma_\nu$. As in the case of rings, an equivalence class of valuations is determined by the {\em valuation monoid} $G^+_\nu=\nu^{-1}(\Gamma^+)$. A submonoid $M\subseteq G$ is a valuation monoid of some valuation if and only if $M^\gp=G$ and for any element $g\in G$ at least one element from the pair $\{g,-g\}$ lies in $M$. In such a case $\Gamma_\nu=G/M^\times$, $\Gamma_\nu^+=M/M^\times$ and the divisibility relation in $M$ induces the total order on $\Gamma_\nu$.

\subsubsection{Valuative interpretation of rangers}\label{valrengsec}
Note that a ranger $r$ is nothing but an isomorphism class of surjective homomorphisms $\nu_r\:\Gamma\oplus\ZZ t\to\Gamma_r$, where $\Gamma_r$ is an ordered group containing $\Gamma$. Thus we can identify $r$ and the corresponding valuation $\nu_r$ on the $\Gamma$-group $L=\Gamma\oplus\ZZ t$.

\subsubsection{Sheaves of linear functions}
Let $\cL_\cR$ be the constant sheaf of groups $L$ on $\cR$. We view its elements as linear maps $\cR\to\cR$ with $\gamma\in\Gamma$ corresponding to the constant map $\gamma$ and $t$ corresponding to the identity. Naturally, we say that the linear function $nt+\gamma$ is of {\em slope $n$}. By $\cL^+_\cR$ we denote the submonoid of non-negative functions, so $\cL^+_\cR([a,b])$ consists of all elements $nt+\gamma$ such that $na+\gamma\ge 0$ and $nb+\gamma\ge 0$. Obviously, the stalk $\cL^+_{\cR,r}$ at a ranger $r$ consists of all elements $nt+\gamma$ such that $nr+\gamma\ge 0$, or $\nu_r(nt+\gamma)\ge 0$ in the valuative language. So, $\cL^+_{\cR,r}=L^+_r$ is the valuation monoid of $r$.

\subsubsection{Sheaves of values}
Groups of values of rangers also form a sheaf: setting $\Gamma_\cR=\cL_\cR/(\cL^+_\cR)^\times$ and $\Gamma^+_\cR=\cL^+_\cR/(\cL^+_\cR)^\times$ we have the following result

\begin{lem}\label{gammasheaflem}
Let $r\in\cR$ be a ranger on an ordered group $\Gamma$, then:

(i) The stalks at $r$ are as follows: $\Gamma_{\cR,r}=\Gamma_r$ is the group of values of $r$ and $\Gamma^+_{\cR,r}=\Gamma_r^+$ is its monoid of non-negative elements.

(ii) The map $\Gamma\to\Gamma_\cR$ is injective and the stalk $(\Gamma_{\cR}/\Gamma)_r=\Gamma_r/\Gamma$ is a cyclic group, which equals $\ZZ$ if $r$ is not divisorial and is finite otherwise (hence even trivial when $\Gamma$ is divisible).
\end{lem}
\begin{proof}
(i) If a section $nt+\gamma$ is invertible in $\cL^+_{\cR,r}$, then $r$ is the open point given by $nr+\gamma=0$, so the inverse of $nt+\gamma$ tautologically extends to a neighborhood of $r$. Therefore $(L^+_r)^\times$ is the stalk of  $(\cL^+_\cR)^\times$ at $r$, and $\Gamma_{\cR,r}=L/(L^+_r)^\times=\Gamma_r$.

(ii) By the above $(L_r^+)^\times$ is trivial if $r$ is not divisorial and is generated by an element $nt+\gamma$ with $n\ge 1$ if $r$ is divisorial. The claim follows.
\end{proof}

\subsubsection{Piecewise linear functions}
Let $I\subseteq\cR$ be an interval with divisorial endpoints. A $\Gamma$-pl function $f\:I\to\cR$ is a locally linear function, i.e. it is locally of the form $nt+\gamma$. Corners (or change of slope) can only happen at open points, and it follows from the quasi-compactness of $I$ that there exists finitely many of them, so $I=\cup_{i=0}^n[x_i,x_{i+1}]$ with $x_i\in\Gamma_\QQ$ and the slopes are constant on the intervals.

\begin{lem}\label{plem}
Sections of $\Gamma_\cR$ are naturally identified with $\Gamma$-pl functions, and sections of $\Gamma^+_\cR$ are non-negative $\Gamma$-pl functions. More concretely, given a section $s\in\Gamma_\cR(I)$ let $I=\cup_{i=1}^nI_i$ be an open covering such that $s|_{I_i}$ lifts to a section $s_i\in\cL_\cR(I)$. Then the corresponding linear functions $f_i\:I_i\to\cR$ fit at the intersections and hence give rise to a $\Gamma$-pl function $f\:I\to\cR$ which depends only on $s$.
\end{lem}
\begin{proof}
On an intersection $I_{ij}=I_i\cap I_j$ we have that $s_i-s_j$ is a section of $(\cL^+_\cR)^\times$ and hence $f_i=f_j$ as functions $I_{ij}\to\cR$.
\end{proof}

\begin{rem}
(i) Note that a non-trivial gluing only happens at isolated open points of $I_{ij}$ and then the function at such a point might have a corner. The latter can be interpreted as follows. The stalk of $(\cL^+_\cR)^\times$ at a point $r$ is trivial if $r$ is closed, and otherwise it is of the form $n_r\ZZ$, where $n_r>0$ is the minimal number such that $n_rr\in\Gamma$. Using Godement resolution it is easy to see that $H^1(\cL^+_\cR)=\oplus_{r\in\Gamma_\QQ}(n_r\ZZ)^2/n_r\ZZ$, where $n_r\ZZ$ is embedded into its square diagonally. If $f$ is a section of $\Gamma_\cR$, its image in the summand $(n_r\ZZ)^2/n_r\ZZ$ is precisely the difference between the slopes on the two sides of $r$, that is, it is the angle of the corner of $f$ at $r$. Note also that $\Gamma_r/\Gamma=\ZZ/n_r\ZZ$, and if $n_r>1$ the non-integrality of the ranger imposes restrictions on the possible angles of corners.

(ii) The language of sheaves of values also conceptually encodes integrality/rationality properties of divisorial rangers, so it is clearly preferable when $\Gamma$ is not divisible. We have the following analogy with algebraic geometry: closed points $x\in\AA^1_k$ correspond to divisorial rangers, the primitive field extension $k(x)/k$ corresponds to the group extension $\Gamma_r/\Gamma$ with finite cyclic quotient, rational points correspond to integral rangers $r\in\Gamma$, the degree $n_x=[k(x):k]$ corresponds to $n_r$.
\end{rem}

\subsubsection{Metric}
Finally, let us explain how the $\Gamma$-metric is reconstructed from this data.

\begin{lem}\label{metriclem}
Let $[x,y]$ be an interval in $\cR$ with rational endpoints. Then its length is determined by the sheaf $\Gamma^+_\cR$ as follows: it is the minimal length of an interval of the form $f([x,y])$, where $f\:I\to\cR$ is a strictly monotonic pl function.
\end{lem}
\begin{proof}
Clearly, the minimal length is attained when the slopes everywhere are positive and minimal possible, that is, $f=t+\gamma$ is just a translation.
\end{proof}

\subsubsection{$\Gamma$-segments revisited}
In view of Lemma \ref{metriclem}, we can now replace the definition of $\Gamma$-segments by the following equivalent one: a $\Gamma$-segment is a pair $I=(I,\Gamma^+_I)$ consisting of a topological space $I$ provided with a sheaf~$\Gamma^+_I$ of $\Gamma^+$-monoids which is isomorphic to $(J,\Gamma^+_\cR|_J)$ for an interval $J$ in $\cR_\Gamma$. A morphism between $\Gamma$-segments is taken in the category of $\Gamma$-monoided spaces. It is easy to see that giving such a map $(I,\Gamma^+_I)\to(J,\Gamma^+_J)$ is equivalent to giving a $\Gamma$-pl map $f\:I\to J$. Locally at a non-divisorial point $x\in I$ the map is linear, i.e. given by $f(t)=nt+r$ and we call $|n|$ the {\em dilation} of $f$ at $x$ because it is the ratio by which the metric is expanded at $x$.

\subsubsection{Pinched $\Gamma$-segments}\label{pinchsec}
By a {\em pinched $\Gamma$-segment} we mean a $\Gamma$-segment $I=[x_1,x_2]$ whose sheaf of values is modified at one or both endpoints by replacing the usual stalk $\Gamma^+_{x_i}$ with $\Gamma^+$ -- this is the subsheaf whose sections are locally constant at the endpoint(s). In particular, one can (non-trivially) pinch a $\Gamma$-segment only at a non-divisorial end-point. Geometrically pinching means that we only allow constant functions locally at a pinch. Clearly the structure of the $\Gamma$-segment is reconstructed from the pinched structure just by continuity at the endpoints, but these are pinched segments that naturally show up in some applications.

\subsubsection{Relation to $\Gamma$-graphs}
Now, we can give an elegant construction of a $\Gamma$-metric on a quasi-graph $T$. By a {\em sheaf of $\Gamma$-values} on $T$ we mean a sheaf of $\Gamma^+$-monoids such that each quasi-interval $I=[a,b]$ with the restricted sheaf of values $\Gamma^+_I=\Gamma^+_T|_I$ is a $\Gamma$-segment or a pinched $\Gamma$-segment. Clearly this provides $T$ with a $\Gamma$-metric -- each quasi-interval acquires it by \S\ref{pinchsec} and the compatibility on the intersections is automatic. By a slight abuse of language the monoided space $(T,\Gamma^+_T)$ will also be called a (pinched) $\Gamma$-graph.

\section{Adic curves}\label{curvessec}

\subsection{Affinoid fields}
First we recall the definitions of affinoid fields and the Abhyankar inequality and deduce some consequences.

\subsubsection{Valued fields}
By a {\em valued field} $k$ we mean a field with a fixed valuation ring $k^+$. When needed we will use the expanded notation $(k,k^+)$. Equivalently one can consider the equivalence class of valuations $|\ |\:k\to \Gamma$ with ring of integers $k^+$. To such a field we associate the residue field $\tilk=k^+/m_{k^+}$ and the group of absolute values $|k^\times|$ with the set of rangers $\cR_{|k^\times|}=\cR(|k^\times|)$.

\subsubsection{Group of values}
We also consider the group of values $\Gamma_k=\log|k^\times|$, its divisible hull $\Gamma_{k,\QQ}=\Gamma_k\otimes_\ZZ\QQ$ and the set of rangers $\cR_k=\cR(\Gamma_k)$. Note that $S=\Spec(k^+)$ with the sheaf of values $\Gamma^+_S:=\log(\calO_S^\times/(\calO_S^+)^\times)$ can be canonically identified with $\Spa(\Gamma_k,\Gamma_k^+)$.

\subsubsection{Valued field extensions}
By an {\em extension of valued fields} we mean a field extension $l/k$ such that $k^+=l^+\cap k$. To such an extension we associate in addition to the usual algebraic invariants $f_{l/k}=[\till:\tilk]$ and $e_{l/k}=\#(\Gamma_l/\Gamma_k)$, the transcendence invariants $F_{l/k }=\trdeg(\till/\tilk)$ and $E_{l/k}=\dim_\QQ(\Gamma_{l,\QQ}/\Gamma_{k,\QQ})$. Classical fundamental and Abhynakar's inequalities state that $e_{l/k}f_{l/k}\le [l:k]$ and $F_{l/k}+E_{l/k}\le\trdeg(l/k)$, respectively.

\subsubsection{Affinoid fields}
Recall that an {\em affinoid field} $k$ is a valued field $(k,k^+)$ provided with the topology whose restriction to $k^+$ is the $\pi$-adic topology for a topologically nilpotent element $\pi\in k^+$ and such that~$k^+$ is open in~$k$. In particular, the case of $\pi=0$ is allowed and then the topology is discrete. Equivalently, this datum is given by valuation rings $k^+\subseteq\kcirc$ of $k$, where the ring of integers $k^+$ determines the valuation, the ring of power-bounded elements $\kcirc$ is a localization of $k^+$ of rank at most one, and $\pi\in k^+$ is such that $\Rad(\pi)=\kcirccirc$ is the maximal ideal of $\kcirc$. We will use the notation $(k,k^+)$ or simply $k$ if the structure is clear from the context. The affinoid field $(k,\kcirc)$ will become denoted $k_\eta$ (for reasons that will be clear below).

If $k=\kcirc$, then $\pi=0$ and the topology is discrete. Otherwise, $\kcirc$ is of rank one and the valuation on $k^+$ is microbial with a microbe $\pi$ being any non-zero element of $k^+$ non-invertible in $\kcirc$. We will say that $k$ is {\em discrete} or {\em microbial} accordingly. In the latter case, the valuation ring $k^+$ is composed from $\kcirc$ and the induced valuation ring $(\tilk_\eta)^+$ on the residue field of $k_\eta$.

\subsubsection{Completed residue fields}
Given an adic space $X$ and a point $x$, the residue field of $x$ is an affinoid field that will be denoted $k(x)$ or $(k(x),k(x)^+)$. Its completion will be denoted $\kappa(x)=(\kappa(x),\kappa(x)^+)$.

\subsubsection{Adic spectrum}
Let $S=\Spa(k,k^+)$ be the adic spectrum. Its points $s$ parameterize valuation rings $k^+_s$ such that $k^+\subseteq k^+_s\subseteq\kcirc$. Thus $S$ is homeomorphic to $\Spec(k^+)$ is the discrete case, and $S$ is homeomorphic to the closed subscheme $\Spec((\tilk_\eta)^+)$ of $\Spec(k^+)$ obtained by removing the generic point. In either case, $\kcirc=k(\eta)^+$ determines the generic point $\eta\in S$. Moreover, we can view $\eta=\Spa(k_\eta)$ as a pro-open subspace and we call $X_\eta=X\times_S\eta$ the {\em generic fiber of $X$}.

\subsubsection{Extensions of affinoid fields}
Naturally, by an {\em extension of affinoid fields} $l/k$ we mean an extension of valued fields which also respects the topologies: $\kcirc=\lcirc\cap k$. In particular, it can happen that $k$ is trivially valued and $l$ is microbial, but not vice versa.

Since completion preserves the residue field and the group of values, Abhyankar's inequality implies the following completed version for affinoid fields: $F_{\hatl/k}+E_{\hatl/k}\le\trdeg(l/k)$.

\begin{rem}\label{affieldrem}
(i) For any morphism of adic spaces $f\:X\to S$ and a point $x\in X$ with $s=f(x)$ consider the induced extension of adic fields $k(x)/k(s)$ and use the notation $E_{x/S}=E_{k(x)/k(s)}$ and $F_{x/S}=F_{k(x)/k(s)}$, or $E_x, F_x$ if $S$ is clear from the context.

(ii) If $f$ is of finite type and relative dimension $d$, then $\kappa(x)$ is the completion of a finitely generated extension of $k(s)$ of transcendence degree at most $d$, and by Abhyankar's inequality $E_{x/S}+F_{x/S}\le d$.
\end{rem}

Now, let us work out consequences of Abhyankar's inequality concerning the morphism $S_l=\Spa(l,l^+)\to S$. This mainly reduces to the question how specializations in $S_l$ and $S$ are related.

\begin{lem}\label{affspeclem}
Let $l/k$ be an extension of affinoid fields. Then

(i) The map $f\:S_l\to S$ is surjective.

(ii) Assume that $x,y\in S_l$ and $x$ is a specialization of $y$, then $F_x\le F_y$, $E_x\ge E_y$ and $E_x+F_x\le E_y+F_y$, and the first two inequalities are strict whenever $x\neq y$, $s_x=s_y$ and $F_y$ and $E_x$ are finite.
\end{lem}
\begin{proof}
Using the relation to schemes, we can reformulate the question in terms of the morphism $g\:\Spec((\till_\eta)^+)\to\Spec((\tilk_\eta)^+)$.
Then the first claim holds because being flat and local $g$ is surjective.

To prove (ii) we start with two reductions. First, we can localize the valuation rings achieving that $x$ and $s_x$ are closed. Second, replacing the valuation rings by quotients by $m_y$ and $m_{s_y}$ we achieve that $y$ and $s_y$ are the generic points, while $F_x$, $F_y$ and $E_x-E_y$ remain unchanged. Indeed, the residue fields are unchanged, while the groups of values $\Gamma_x$ and $\Gamma_y$ (resp. $\Gamma_{s_x}$ and $\Gamma_{s_y}$) are divided by the same convex subgroup corresponding to $\Gamma_y$ (resp. $\Gamma_{s_y}$).

Now we are reduced to the particular case, when $E_y=0$ and $F_y=\trdeg(k(x)/k(s_x))$. Hence $E_x\ge E_y$, and $E_x+F_x\le F_y$ by Abhyankar's inequality. Finally, if $s_x=s_y$ and $x\neq y$, then $\Gamma_{s_x}=0\neq \Gamma_{x}$. So, $E_x>0=E_y$ and $F_y\ge F_x+E_x>F_x$.
\end{proof}

Combining Remark~\ref{affieldrem}(i) and the above lemma we immediately obtain

\begin{cor}\label{affspeccor}
Assume that $X\to S$ is a morphism of relative dimension $d$ and $x\in X$ a point with $s=f(x)$. Then the induced morphism $$f\:\Spa(k(x),k(x)^+)\to S_s=\Spa(k(s),k(s)^+)$$ is surjective and has at most $d$ vertical specializations in the fibers, namely $$\sum_{t\in S_s}(\#(f^{-1}(t))-1)\le d+1.$$
\end{cor}

\subsection{Valuative classification of points on adic curves}
Now, we are going to classify points on adic curves according to their residue field and its basic invariants. By a {\em weak (adic) curve} over $S=\Spa(k,k^+)$ we mean an adic $S$-space~$X$ of weakly finite type and pure relative dimension 1, and we call it an {\em (adic) $S$-curve} if it is of finite type over $S$. For $x\in X$ we denote its image in $S$ by $s_x$.

\subsubsection{Bounded and unbounded points}
We say that a point $x\in X$ is {\em unbounded} if there exists an element $\omega\in k(x)$ such that $|\omega|>|a|$ for any $a\in k$. By definition of the Huber spectrum, unbounded points may only exist when $k$ is discrete. We will see that the classification of unbounded points is simple, so the set-theoretic difference between the spaces for discrete and microbial topologies is easy to describe.

\begin{rem}
In adic geometry one usually uses the notion ``continuous'' instead of ``bounded'', but our choice fits the terminology concerning rangers.
\end{rem}

\subsubsection{Classification of points}
Bounded points can be naturally subdivided further into classes according to the properties of the valued field extensions $k(x)/k(s_x)$ and its invariants $E_x$ and $F_x$ satisfying $E_x+F_x\le 1$. This yields the following classification, where the numeration is close to Berkovich's numeration and the terminology is similar to that of rangers (for reasons that will be clear later):

\begin{itemize}
\item[(1)] $x$ is a {\em classical} point if $[k(x):k]<\infty$. This happens if and only if the maximal ideal $m_x\subset\calO_{X,x}$ is non-zero. All other points are called {\em generic}.
\item[(2)] $x$ is a {\em divisorial} point if $F_x=1$.
\item[(3)] $x$ is a {\em cut} point if $E_x=1$ and for any $r<1$ in $|k(x)^\times|$ there exists $r_1<r_2<1$ in $|k^\times|$ such that $r_1<r<r_2$.
\item[(4)] $x$ is a {\em unibranch} point if it is generic, but $E_x=F_x=0$.
\item[(5)] $x$ is an {\em infinitesimal} point if there exists $r<1$ in $|k(x)^\times|$ such that $r_1<r$ for any $r_1<1$ in $|k^\times|$. In particular, $E_x=1$.
\item[(6)] $x$ is an {\em unbounded} point if there exists $r\in|k(x)^\times|$ such that $r<r_1$ for any $r_1\in|k^\times|$. In particular, $E_x=1$.
\end{itemize}

\subsubsection{Symmetric and asymmetric cuts}\label{fxsec}
Recall that by Corollary \ref{affspeccor} the map $f_x\:T_x=\Spec(\kappa(x)^+)\to S_x=\Spec(k(x_s)^+)$ is surjective and has at most one non-singleton fiber, which (if it exists) necessarily consists of two points. Furthermore, if it exists, then obviously $|k(s_x)^\times|^\QQ\subsetneq|k(x)^\times|^\QQ$ and hence $E_x=1$, leaving us with cases (3), (5) and (6). In fact, by unravelling the definitions, a point is infinitesimal if and only if $E_x=1$ and the double fiber sits over $s_x$, and the point is unbounded if and only if $E_x=1$ and the double fiber sits over $\eta=\Spa(k,k)$, which can happen only in the discrete case. In particular, cases (5) and (6) collide in the classical situation when $k(s_x)^+=k(s_x)$. Finally, in the case of a cut point it can happen that either $f_x$ is bijective, and then we say that the cut is {\em symmetric}, or $f_x$ has a double point over a proper generization $t$ of $s_x$ such that $k(t)^+\neq k(t)$, and in this case the cut is called {\em asymmetric}. We will later see that this notion has a metric interpretation, but already comparison to Remark~\ref{cutrem} explains the terminology.

\subsubsection{Invariance of types}
Types of points are preserved by quasi-finite morphisms and almost preserved by base change to the algebraically closed ground field.

\begin{lem}\label{geomtypelem}
(i) Assume that $f\:Y\to X$ is a morphism of weak $S$-curves which is not constant on any component of $Y_\eta$ and $y\in Y$ is a point. Then $y$ and $x=f(y)$ are of the same type.

(ii) Assume that $X$ is a weak $S$-curve, $\oX=X\times_S\oS$, where $\oS=\Spa(\whka)$, and $x\in X$ is a point with a preimage $\ox\in\oX$. Then the types of $x$ and $\ox$ are the same with the only exception that if $\ox$ is a $\whka$-point which is not a $k^a$-point, then $x$ is a generic unibranch point.
\end{lem}
\begin{proof}
The invariants $E_{l/k}$ and $F_{l/k}$ are not sensitive to completed algebraic extensions. It follows that everything in our classification distinguished by them is preserved under the maps $Y\to X$ and $\oX\to X$. Therefore, a switch can happen only when $E=F=0$ and a classical point is mapped to a generic unibranch point. Finally, it is easy to see that this only happens when $k(\ox)/k(x)$ is not algebraic, and this is precisely the exceptional case in (ii), which might happen only in the microbial case because otherwise there is no completion involved.
\end{proof}

In view of the lemma, we say that $x$ is {\em geometrically classical} if its lift $\ox$ is classical.

\begin{rem}
Berkovich classifies points by geometric type, so unibranch points in our sense split between types (1) and (4). However there are automorphisms of analytic discs not preserving the ground field which do not preserve the geometric metric on the disc and even mix the geometric types 1 and 4. This is the reason why we think that this slightly different classification is more natural.
\end{rem}

\subsection{The generic fiber}\label{genfib}
First, let us illustrate the situation by comparing it to algebraic and Berkovich geometries which are well explored in the literature. Let us describe the topological structure of the generic fiber $X_\eta$. For brevity we localize everything so that $\kcirc=k^+$, $S=\{\eta\}$ and $X=X_\eta$. Since we are only interested in topology, we can assume that $X$ is reduced.

\subsubsection{The discrete case}\label{trivgenfib}
If $k^+=k$, the topology on~$k$ has to be discrete. In this case $X$ has only three types of points -- divisorial and classical ones have trivial valuation on $k(x)$, and for any other point $x$ one has that $k(x)$ is a function field of an algebraic $k$-curve and $k(x)^+$ is a non-trivial valuation ring containing $k$, and hence a DVR corresponding to a point on a smooth projective model of $k(x)$. In particular, any such $x$ is both unbounded and infinitesimal, it is a specialization of a divisorial point and it can specialize further to a classical point as we explain below.

Locally $X$ is of the form $\Spa(A,A^+)$, where $\Spec(A)\into\Spec(A^+)$ is a dense open immersion of algebraic $k$-curves such that $\Spec(A^+)$ is normal at the complement. Moreover, these open immersions glue together giving rise to a dense open immersion of $k$-curves $\calU\into\calX$, which determines $X$. In the formalism of \cite[\S3.1]{temrz} one simply has that $X=\Spa(\calU,\calX)$, and its points are $k$-semivaluations on $\calU$ with center on $\calX$. So, divisorial and classical points correspond to generic and closed points of $\calU$, and an unbounded point $x\in X$ corresponds to a branch of $\calX$, that is, a closed point $\bfx\in\calX^\nor$ of the normalization of $\calX$. Note that $x$ generalizes to the divisorial point $y$ corresponding to the generic point $\bfy$ of the irreducible component of $\bfx$ and $(k(x),k(x)^+)=(k(\bfy),\calO_{\calX^\nor,\bfx})$. Also, if $\bfx$ sits over a point $\bfz\in\calU$, then $x$ specializes to the classical point $z$ corresponding to $\bfz$.

To summarize, in case $k^+ = k$, any point of~$X$ has at most one classical specialization, and generizations of a classical point $x$ are parameterized by the branches of the one-dimensional local ring $\calO_{X,x}$. In addition, $X$ is canonically sandwiched between the two analytifications $\calU^\ad=\Spa(\calU,\calU)$ and $\calX^\ad=\Spa(\calX,\calX)$. The open subspace $\calU^\ad\into X$ is obtained by removing all unbounded points that have no classical specialization in $X$, and the open immersion $X\into\calX^\ad$ saturates $X$ with classical specializations of unbounded points.

More generally, let us consider an affinoid field $(k,k^+)$ of higher rank but equipped with the discrete topology.
Then the generic point~$\eta$ of $S = \Spa(k,k^+)$ corresponds to the trivial valuation of~$k$ and $X_\eta$ is of the above described form.
A point~$x$ of~$X \setminus X_\eta$ is unbounded if and only if it has a generization~$x_\eta$ in~$X_\eta$ which is infinitesimal.
Then the residue field $\widetilde{k(x_\eta)}$ of the corresponding valuation is a finite extension of $\widetilde{k(\eta)}$.
Now~$x$ is the composition of~$x_\eta$ with a nontrivial valuation~$\tilde{x}$ of $\widetilde{k(x_\eta)}$.
This valuation~$\tilde{x}$ is a bounded point of~$X$, more precisely a classical point.
Moreover, any composition of an infinitesimal point~$x_\eta$ of~$X_\eta$ with a bounded classical point~$\tilde{x}$ with residue field $\widetilde{k(x_\eta)}$ gives an unbounded point in $X \setminus X_\eta$.
For fixed~$x_\eta$ and~$s_x$ there are only finitely many options for~$\tilde{x}$.
If the base field $(k,k^+)$ is henselian, there is even a unique such point~$\tilde{x}$.

\subsubsection{The microbial case}\label{microbgenfib}
If $k$ is microbial and $k^+=\kcirc$, then $k$ is real valued and the maximal locally Hausdorff quotient $\calX$ possesses a natural structure of a $\hatk$-analytic Berkovich curve, and $X$ can be constructed by saturating the associated adic space $\calX^\ad$ with finitely many additional infinitesimal points. Again, this is checked locally, and then one can assume that $X=\Spa(A,A^+)$, where $A$ is a $k$-affinoid algebra and $A^+$ is the preimage in $\Acirc$ of a finitely generated subalgebra $\tilA^+$ of $\tilA$ such that $\tilA$ is a localization of $\tilA^+$. It follows easily that $\calM(A)^\ad=\Spa(A,\Acirc)$ is open in $X$ and the complement is a finite set of infinitesimal points (valuations of rank 2).

Note that $\calX$ embeds into $X$ set-theoretically and the complement consists precisely of all points $x$ such that $k(x)^+$ is of rank 2. Since $\calX$ contains no unbounded points, this happens if and only if $x$ is infinitesimal and each such a point $x$ is obtained by choosing a divisorial point $y\in \calX$ and fixing a non-trivial $\tilk$-valuation on $\wHy$. As for Berkovich's classification of points into four types we have the following matching. Classical points correspond to Zariski closed Berkovich points, divisorial points correspond to points of type 2, cut points are automatically symmetric and they correspond to points of type 3, and unibranch points correspond to points of types 1 and 4, depending on their geometric radius -- whether it is zero or positive. Finally, infinitesimal points are often referred to as points of type 5.

\subsection{Specializations and generizations on adic curves}
To large extent the geometries of different fibers interact through specialization/generization relations, which will be described in this subsection. There are also mild specialization/generization relations occuring inside the same fiber and we will describe them too.

\subsubsection{Vertical generizations}
Recall that vertical generizations of a point $x\in X$ (on any adic space) are obtained by keeping $k(x)$ unchanged and localizing the valuation ring $k(x)^+$. In particular, they form the chain $X_x=\Spa(k(x),k(x)^+)$. We warn the reader that for $S$-spaces this is a bit misleading as over $S$ these are mainly ``horizontal'' generizations happening``parallel'' to generizations in $S$ in this context.

\begin{lem}\label{generizlem}
Let $X$ be a weak $S$-curve, $x\in X$ a point with $X_x=\Spa(k(x),k(x)^+)$ and $f_x\:X_x\to S_{s_x}=\Spa(k(s_x),k(s_x)^+)$ the induced map. In addition, let $y\in X_x$ be a vertical generization of $x$, then

(i) If $x$ is neither an asymmetric cut point, nor an infinitesimal point, then $f_x$ is bijective and either $y$ is divisorial or of the same type as $x$. Equivalently, $X_x=T^+\coprod T^-$ set-theoretically, where $T^+$ consists of divisorial points and is closed under generizations, and $T^-$ consists of points of the same non-divisorial type and is closed under specializations.

(ii) If $x$ is an asymmetric cut point, then the generization chain consists of points of the same type followed by an infinitesimal point $y$ followed by a divisorial point $z$ which is followed by divisorial points. In fact, $\{y,z\}$ is the double fiber and $f_x$ is bijective outside of this fiber.

(iii) If $x$ is classical, then any its vertical generization is classical.

(iv) If $k$ is discrete, then any generic point has a divisorial generization.

(v) In (some) microbial cases there exist unibranch and symmetric cut points without any divisorial generization.
\end{lem}
\begin{proof}
The properties of $f_x$ were discussed in \S\ref{fxsec}. Except distinguishing types of cut points, the type is controlled by $E_x$, $F_x$ and $m_x$, and one easily sees that (up to this distinguishing) the claim of (i) follows from Lemma~\ref{affspeclem}. Furthermore, infinitesimal points immediately generize to divisorial ones, so the only generization which mixes the types within the $E_x=1$ case goes from a cut point to an infinitesimal point, and this happens precisely for asymmetric cuts. This implies (i) and (ii). Furthermore, (iii) is trivial, and (iv) reduces to the observation that if the valuation is trivial, then $X$ is the closure of the finite set of divisorial valuations of $X_\eta$. The claim of (v) is clear already when $X=X_\eta$ because Berkovich curves may contain points of types 3 and 4.
\end{proof}

\subsubsection{Vertical specializations} \label{vertical_specializations}
Now, let us go in the opposite direction and describe the set $\ox_v$ of vertical specializations of a point $x\in X$ (on any adic space $X$). If $X=\Spa(A,A^+)$ is affine, set $\tilX^x=\Spa(\wkx,\tilA^+_x)$ where $\tilA^+_x$ is the image of $A^+$ under the homomorphism $A^+\to k(x)^+\to\wkx$ induced by $\phi_x\:A\to k(x)$ and the topology on $\wkx$ is discrete. It is easy to see that this definition is compatible with localizations and hence globalizes.

\begin{lem}\label{vsepclem}
If $X$ is an adic space and $x\in X$ is a point, then the space $\tilX^x$ is naturally homeomorphic to $\ox_v$.
\end{lem}
\begin{proof}
It suffices to check this in the affine case when $X=\Spa(A,A^+)$. Then a point $z\in\ox_v$ is a valuation which refines the valuation of $x$ on $k(x)$ and contains the image of $A^+$. That is, it is a valuation ring $k(z)^+$ of $k(z)\subseteq k(x)$ such that $\phi_x(A^+)\subseteq k(z)^+\subseteq k(x)^+$. Choosing $k(z)^+$ is equivalent to choosing its image in $\wkx$, which can be any valuation ring of $\wkx$ containing $\tilA^+_x$, and the latter is just a point of $\Spa(\wkx,\tilA^+_x)$.
\end{proof}

The above construction is functorial: a morphism $f\:Y\to X$ with $f(y)=x$ induces a morphism $\tilY^y\to\tilX^x$. Now, let us specialize this to the case of our interest $X\to S$. Note that $\tilS^s$ is the spectrum of the affinoid field $(\wt{k(s)},\wt{k(s)^+})$, where $\wt{k(s)^+}$ is the image of $k^+$ in $\wt{k(s)}$.

\begin{lem}\label{vertcurve}
Let $X$ be a weak adic $S$-curve, $x\in X$ a point and $s=s_x\in S$ the image of $x$. Consider the morphism $\tilf_x\:\tilX^x\to\tilS^{s}$

(i) If $x$ is not divisorial, $\tilf_x$ factors through the profinite morphism $\Spa(\wkx,k^+)\to\tilS^s$ and the map $\tilX^x\to\Spa(\wkx,k^+)$ is an open immersion (resp. isomorphism) if $f$ is separated (resp. proper). In addition, the map $\Spa(\wkx,k^+)\to\tilS^s$ is bijective whenever $k^+$ is henselian.

(ii) If $x$ is divisorial, then $\tilX^x$ is naturally homeomorphic to the set of generic points on an adic $\tilS^{s}$-curve $\tilX$ with $k(\tilX)=\wkx$. The curve can be chosen separated (resp. proper) whenever $X\to S$ is separated (resp. proper).
\end{lem}
\begin{proof}
The first claims in (i) and (ii) follow from the fact that $\tilX^x$ is covered by open subspaces of the form $\Spa(\wkx,\tilA_x)$, where $\tilA_x$ is finitely generated over $\wt{k(s)^+}$. In addition, $\wkx$ is a function field of a $\wks$-curve if $x$ is divisorial, and $\wkx/\wks$ is algebraic otherwise (but it can be infinite in the unibranch case). The first claims of (i) and (ii) follow easily. The claims about separated/proper cases follow from the valuative criteria. Finally, if $k^+$ is henselian, then its quotient $\wt{k(s)^+}$ is also henselian, hence in the case of (i) the valuation on $\wks$ extends uniquely to the algebraic extension $\wkx$. This precisely means that the map $\Spa(\wkx,k^+)\to\tilS^s$ is bijective.
\end{proof}

\subsubsection{Horizontal generizations}
To complete the picture it remains to describe the other type of specialization relations. Recall that horizontal specializations in adic spaces happen when the kernel of the point (or its semivaluation) drops, so on a curve $X$ a horizontal generization $x\prec y$ can happen only when $x$ is a classical point and $y$ is a generic point. Furthermore, if $0\neq t\in m_x$, then $|t|_y<|a|_y$ for any $a\in k$ because otherwise $|t|\ge |a|$ would define a neighborhood of $x$ not containing $y$. Therefore $y$ is an unbounded point. By Lemma~\ref{generizlem} the generic fiber $X_\eta$ contains unique vertical generizations $y'$ and $x'$ of $y$ and $x$ such that $y'$ is unbounded and $x'$ is classical. So, using the description of $X_\eta$ in \S\ref{trivgenfib} we obtain the following result:

\begin{lem}\label{horizontallem}
Assume that $X$ is a weak $S$-curve, then

(i) If $x\prec y$ is a horizontal specialization on $X$, then $k$ is discrete, $x$ is classical, $y$ is unbounded and both $x$ and $y$ lie in the same fiber $X_s$ of $X$ over $S$.

(ii) If $k$ is discrete, then the set of unbounded generizations of a classical point $x$ is parameterized by branches of the one-dimensional local ring $\calO_{X,x}$.

(iii) Any unbounded point has at most one classical specialization.
\end{lem}

\subsection{Fractal structure}

In this subsection we assume that $(k,k^+)$ has finite rank.
Let~$X$ be a weak $S$-curve and $x \in X_s$ a divisorial point for some $s \in S$.
In subsection~\ref{vertical_specializations} we have seen that the set of vertical specializations of~$x$ is in bijection with the space~$\tilX^x$ glued from affinoids of the form
\[
 \Spa(\widetilde{k(x)},\tilA_x^+).
\]
for affinoid opens $\Spa(A,A^+)$ of~$X$.
It is a discretely ringed adic space over $\Spa(\widetilde{k(s)},\widetilde{k(s)}^+)$.
Note that $\widetilde{k(x)}$ is the function field of the reduction curve~$\tilX_x$ over~$\widetilde{k(s)}$.
Every closed point of~$\tilX_x$ defines a discrete $\widetilde{k(s)}$-valuation of $\widetilde{k(x)}$ of rank~$1$.
The corresponding points of~$X$ in the closure of~$x$ are precisely the infinitesimal points of~$X_s$ specializing~$x$.
All other points of $\ox$ lie in some $X_{s'}$ for a proper specialization $s \succ s'$.

The specializations in $\Spa(\widetilde{k(x)},\widetilde{A}_x^+)$ of valuations of the above described type are precisely the unbounded points of $\Spa(\widetilde{k(x)},\widetilde{A}_x^+)$.
The space $\Spa(\widetilde{k(x)},\widetilde{A}_x^+)$ can be thought of as something similar to a curve.
However, it is missing the classical points.
In the following we outline an alternative approach to describe~$\ox$ that reinterprets the unbounded points of $\Spa(\widetilde{k(x)},\widetilde{A}_x^+)$ as classical points of a microbial curve over $(\widetilde{k(s)},\widetilde{k(s)}^+)$.

\subsubsection{Construction of a microbial curve}

We consider $(\widetilde{k(s)},\widetilde{k(s)}^+)$ equipped with the valuation topology.
This is the $\varpi$-adic topology for a topologically nilpotent element $\varpi \in \widetilde{k(s)}^+$
If $X = \Spa(A,A^+)$ is affinoid we set $\tilA_x := \tilA_x^+[1/\varpi]$ and consider the affinoid curve
\[
 \Spa(\tilA_x,\tilA_x^+)
\]
over $(\widetilde{k(s)},\widetilde{k(s)}^+)$.
Note that the quotient field of $\tilA_x$ equals $\widetilde{k(x)}$ as $\widetilde{k(x)}$ is the quotient field of the image of~$A$, hence also of the image of~$A^+$ and hence of $\tilA_x$.
We define a map
\[
 \varphi_x : \Spa(\tilA_x,\tilA_x^+) \longrightarrow \Spa(\widetilde{k(x)},\tilA_x^+)
\]
by sending a generic point to its restriction to $\widetilde{k(x)}$ and a classical point~$y$ with support $\gtm \subseteq \tilA_x$ to the composition
\[
 v_y = v_\gtm \circ y,
\]
where~$v_\gtm$ is the valuation of $\widetilde{k(x)}$ corresponding to the discrete valuation ring $(\tilA_x)_\gtm$.

The above construction glues and for any curve~$X$ over~$S$ and divisorial point $x \in X$ we obtain a microbial curve $\tilX^x_\an$ and a map.
\[
 \varphi_x : \tilX^x_\an \longrightarrow \tilX^x.
\]

\begin{prop}
 The above defined morphism defines a homeomorphism of~$\tilX^x_\an$ to $\tilX^x \setminus X_s$, or in other words, all points of $\tilX^x$ whose corresponding valuation does not restrict to the trivial valuation of $\widetilde{k(s)}$.
\end{prop}

\begin{proof}
 Without loss of generality we may restrict to the affinoid case $X = \Spa(A,A^+)$.
 By construction $\varphi_x$ maps generic points to bounded valuations and classical points to unbounded valuations.
 From this and the description of the map it is clear that~$\varphi_x$ is injective.
 Also surjectivity follows by treating bounded and unbounded valuations separately and using that unbounded valuations are of the form $v_\gtm  \circ y$ as above.
 Once~$y$ is a nontrivial valuation, it defines a classical point of~$\tilX^x_\an$.

 The topology of $\Spa(\widetilde{k(x)},\tilA_x^+)$ is generated by subsets of the form
 \[
  \Spa(\widetilde{k(x)},\tilA_x^+[\alpha_1,\ldots,\alpha_n])
 \]
 for $\alpha_1,\ldots,\alpha_n \in \widetilde{k(x)}$.
 Its preimage in $\Spa(\tilA_x,\tilA_x^+)$ is
 \[
  \Spa(\tilA_x[\alpha_1,\ldots,\alpha_n],\tilA_x^+[\alpha_1,\ldots,\alpha_n]).
 \]
 the subsets of this form constitute a basis of the topology of $\Spa(\tilA_x,\tilA_x^+)$, hence~$\varphi_x$ is a homeomorphism.
\end{proof}

\begin{cor} \label{specializations_as_curve}
 Let~$X$ be a smooth curve over~$S$ and $x \in X$ a divisorial point mapping to $s \in S$.
 Then $\ox \setminus X_s$ is naturally homeomorphic to a smooth microbial curve over $(\widetilde{k(s)},\widetilde{k(s)}^+)$.
\end{cor}

\subsubsection{Divisorial points}
A point $y \in \overline{x} \setminus X_s$ is divisorial as a point of $\tilX^x_\an$ if and only if it is divisorial as a point of~$X$.
The reason is that as a point of~$X$ the point~$y$ corresponds to a valuation that is the composite of~$x$ with the valuation of $\widetilde{k(x)}$ corresponding to~$y$ viewed as a point of~$\tilX^x_\an$.
A composite valuation is divisorial if and only if its constituents are divisorial.
Since~$x$ is divisorial, the desired equivalence follows.

%
%
%
%
%

\section{Metric structure of adic curves}\label{metricsec}
In this section we will mainly study the structure of a single fiber $X_s$ over $s\in S$. Up to a localization of $S$ we can (and often will) assume that $s$ is the closed point. Specializations in $X_s$ describe only a small amount of the structure of $X_s$ - they parameterize branches at divisorial points. A much more informative tool is a $\Gamma_{s,\QQ}$-graph structure that we will construct. In this section it will be natural to switch between additive and multiplicative valuations. The general principle is that we prefer to work additively once pl geometry, metric on skeletons, etc. shows up.

\subsection{Pseudo-adic spaces}
Recall that a pseudo-adic space is a pair $(X,T)$ consisting of an adic space and a proconstructible (in particular, spectral) and convex subset $T$. When no ambiguity is possible, we will use only $T$ to denote the space. For example, a pseudo-adic point $(S,s)$ will often be denoted by $s$ and the fiber $(X,X_s)$ over $s$ will be denoted $X_s$. A pseudo-adic subspace $T\subseteq X_s$ is called {\em constructible} if $T$ is the $s$-fiber of a constructible subset of $X$. In fact, we will only work with such pseudo-adic spaces.

\subsection{The projective line}
In this subsection we are going to describe the structure of the projective line $\PP^1_s=(\PP^1_S,\PP^1_s)$ over a geometric pseudo-adic point $s=(S,s)$. We fix a coordinate $t$, so this space can be obtained, for example, by gluing the unit $S$-discs $\Spa(k\langle t\rangle,k^+\langle t\rangle)$ and $\Spa(k\langle t^{-1}\rangle,k^+\langle t^{-1}\rangle)$ with centers at 0 and infinity and taking the $s$-fiber. The strategy and the outcome are very analogous to the description of Berkovich's affine line in \cite[\S1.4]{berbook}.

\subsubsection{Classical points}
The points of $\AA^1_s=\PP^1_s\setminus\{\infty\}$ are nothing else but (bounded) semivaluations on $k[t]$ that extend the valuation of $k$, where "bounded" refers only to the microbial case (in adic geometry one calls them valuations even if the kernel is non-trivial). Giving a classical $x$ is equivalent to giving a maximal ideal $m_x\subset k[t]$ and an extension of the valuation of $k$ to $k(x)=k[t]/m_x$. Giving a generic point $x$ is equivalent to extending $|\ |_k$ to a (bounded) $k$-valuation $|\ |_x$ on $k(t)$. In non-archimedean geometry one often uses the notation $|f(x)|$ instead of $|f|_x$, meaning that $f(x)\in k(x)$ and the valuation is that of $k(x)$, but we find the former notation more convenient, at least for this paper.


\subsubsection{Generalized Gauss valuations}
A $k$-semivaluation $|\ |_x$ on $k[t]$ is called {\em $t$-monomial} if $|\sum_i{c_it^i}|_x=\max_i|c_i|r^i$, where $r=|t|_x$ is the {\em radius} of the valuation. In particular, $r=1$ yields the classical Gauss extension of valuations. In general, the radius is not an element of $|k^\times|$ and it should be interpreted as a ranger, see below.

The set of all (bounded) $t$-monomial valuations is totally ordered by the domination relation and will be denoted $(0,\infty)_s$. Naturally, the subset of $\PP^1_s$ obtained by adding the classical points $0$ and/or $\infty$ to $(0,\infty)_s$ will be denoted $[0,\infty]_s$, $[0,\infty)_s$ and $(0,\infty]_s$, and we will omit $s$ when this cannot cause confusions.

A $t$-monomial semivaluation is determined by the restriction of $|\ |_x$ onto the monoid $k^\times\oplus t^\NN$. In case of valuations the latter extends to a homomorphism $k^\times\oplus t^\ZZ\onto\Gamma_x$ and, conversely, each such homomorphism onto an ordered group gives rise to a generalized Gauss valuation via the max formula. Therefore in the discrete case we obtain an ordered bijection $r\:(0,\infty)_s\toisom\cR_{|k^\times|}$, and, since this isomorphism respects boundedness of valuations and rangers, $r\:(0,\infty)_s\toisom\cR^b_{|k^\times|}$ in the microbial case. In particular, the {\em radius $r_x=r(x)$} of the generalized Gauss valuation $|\ |_x$ is naturally a ranger on $|k^\times|$.

\subsubsection{$k$-split discs} \label{k-split-discs}
More generally, for any $a\in k$ we denote by $[a,\infty]_s$ the translation of $[0,\infty]_s$ by $a$. Its finite points are {\em $t_a$-monomial} with respect to the coordinate $t_a=t-a$, and we use the notation $|\ |_{a,r}$ to denote the point of radius $r$ around $a$.

By a {\em $k$-split disc} $D_s(a,r)$ of radius $r\in\cR_{|k^\times|}$ with center at $a$ we mean the pseudo-adic subspace of $\PP^1_s$ defined by all inequalities of the form $|t_a|_x<\gamma$ and $|t_a|_x\le\gamma'$ with $\gamma,\gamma'\in|k^\times|^\QQ$ such that $\gamma>r$ and $\gamma'\ge r$. By a slight abuse of language we will encode this set of inequalities by a single inequality $|t_a|_x\le r$ with the ranger. Note also that the inequality $|t_a|\le \gamma^-$ is equivalent to the single inequality $|t_a|<\gamma$, and any other inequality $|t_a|\le r$ can be encoded only by non-strict inequalities $|t_a|\le \gamma_i$ with $\gamma_i\in|k^\times|^\QQ$.

Clearly, $|\ |_{a,r}$ is the maximal point of $D_s(a,r)$, and $D_s(a,r)=D_s(a',r')$ if and only if $r=r'$ and $|a-a'|\le r$, In particular, $|\ |_{a,r}=|\ |_{a',r'}$ if and only if these conditions hold, or, in other words, $[a,\infty]_s\cap[a',\infty]_s$ is identified with the $|k^\times|$-subsegment $[|a-a'|,\infty]_s$ in both $|k^\times|$-segments.

\subsubsection{$k$-split unibranch points}
Finally, one can also define analogs of Berkovich's points of type 4. They are obtained as intersections of $k$-split discs as stated in the following lemma. Such points will be called {\em $k$-split unibranch}.

\begin{lem}\label{type4lem}
Assume that $\calD=\{D_i\}_{i\in I}$ is a family of pairwise non-disjoint $k$-split discs such that $\cap_i D_i$ contains no classical points. Then the intersection $\cap_i D_i$ consists of a single point $|\ |_\calD$.
\end{lem}
\begin{proof}
In any pair of non-disjoint discs, one of them contains the other one, hence this family is a chain and the maximal points $|\ |_i$ of these discs form a chain with respect to the domination, where $x\ge y$ if $|f|_x\ge|f|_y$ for any $f\in k[t]$. For each $i$ choose a presentation $D_i=D(a_i,r_i)$. Each $D_i$ contains a strictly smaller disc $D_j$, hence $D_i=D(a_j,r_i)$ and we can replace $D_i$ by any $D_i(a_j,r)$ with $r_j<r<r_i$ without altering the intersection. In particular, we can achieve that each radius $r_i$ lies in $|k^\times|^\QQ$.

The intersection contains no classical points, hence for any $f(t)\in k[t]$ a small enough disc $D_i$ contains no roots of $f$. Taking the presentation $f=\sum_l a_l(t-a_i)^l$ we obtain that $|a_0|>|a_l|r_i^l$ for any $l>0$, and hence $|f|_x=|a_0|$ for any point $x\in D_i$. This stabilization implies that $|\ |_\calD=\inf_i|\ |_i$ is well defined as a valuation on $k[t]$ with values in $|k^\times|^\QQ$. It remains to observe that by the above argument for any $f\in k[t]$ the value of $|f|_x$ is the same for any $x\in\cap_iD_i$, and hence the intersection $\cap_iD_i$ consists of the single point $|\ |_\calD$.
\end{proof}

The following example shows that one cannot weaken the condition about classical points by considering $k$-points only.

\begin{exam}
If $k$ is henselian but not defectless, then there exists $\alp\in k^a$ such that $\inf_{c\in k}|\alp-c|$ is not attained. In such a case, the set $\{D_i\}_{i\in I}$ of all $k$-split discs containing $\alp$ is such that $\cap D_i$ contains no $k$-points, but contains the classical point corresponding to $\alp$. So, this family is a whole (non-split) disc, rather than a point.
\end{exam}

\subsubsection{Classification of points in the geometric case}
We have established enough examples of points to describe $\PP^1_s$ in the geometric case.

\begin{theor}\label{classth}
Assume that the affinoid field $k$ is algebraically closed. Then any point $x\in\PP^1_s$ falls within one of the following cases:

(1) Classical points: $x=a$ (resp. $x=\infty$) is a $\hatk$-point. In this case, $|f|_x=|f(a)|_k$ and $x$ is given by a condition $|t-a|=0$ (resp. $|1/t|=0$).

(2) Divisorial points: $x=p_{a,r}$ with $r\in|k^\times|$. In this case, $\wkx=\tilk(u)$, $|k(x)^\times|=|k^\times|$, and $x$ is given by the conditions $|t-a|=r$ and $|t-a-cc_i|=r$, where $|c|=r$ and $\{c_i\}$ is a subset of $k^+$ which is mapped surjectively onto $\tilk$.

(3) Cut points: $x=p_{a,r}$, where $r\in\cR_k$ is a cut. In this case, $\wkx=\tilk$, $|k(x)^\times|=|k^\times|\oplus r^\ZZ$ and $x$ is determined by the conditions $r_i<|t-a|<r'_i$, where $r_i,r'_i\in|k^\times|$ are such that $r_i<r<r'_i$.

(4) Unibranch points: $x=\cap_iD_i$, where $D_i=D(a_i,r_i)$. In this case, $\wkx=\tilk$, $|k(x)^\times|=|k^\times|$ and $x$ is determined by the conditions $|t-a_i|\le r_i$.

(5) Infinitesimal points: $x=p_{a,r}$, where $r=\gamma^+$ (resp. $r=\gamma^-$) is an infinitesimal ranger. In this case, $\wkx=\tilk$, $|k(x)^\times|=|k^\times|\oplus r^\ZZ$ and $x$ is determined by the conditions, $|t-a|>\gamma$ and $|t-a|\le r_i$ for $r_i\in|k^\times|$ with $r_i>\gamma$ (resp. $|t-a|<\gamma$ and $|t-a|\ge r_i$ for $r_i<\gamma$).

(6) In the discrete case, unbounded points: $x=p_{a,r}$, where $r\in\{-\infty,\infty\}$ is an unbounded ranger. In this case, $\wkx=\tilk$, $|k(x)^\times|=|k^\times|\oplus r^\ZZ$ and $x$ is given by the conditions $0\neq |t-a|<r$ for any $r\in|k^\times|$ or $|1/t|\neq 0$ and $|t|>r$ for any $r\in|k^\times|$.
\end{theor}
\begin{proof}
It suffices to consider the case of $x\neq\infty$, so $|\ |_x$ is a $k$-semivaluation on $k[t]$. Consider the family of all discs $D_i=D(a_i,r_i)$ containing $x$, where $r_i$ is a ranger of the form $s_i$ or $s_i^-$ with $s_i\in\Gamma_k$. If these discs do not have a common classical point, then by Lemma~\ref{type4lem} we are in case (4). In particular, for any $f\in k[t]$ there exists $c\in k$ such that $|f-c|_x<|f|_x$, and hence $k(x)$ and $k$ have the same residue field and group of values.

Assume now that such an $a\in k$ exists, so we can present each disc as $D(a,r_i)$. Clearly $\{r_i\}$ is a convex set of rangers unbounded from above. If this set is empty or unbounded from below, then the point is unbounded. In particular, $k$ is discrete, and we are in the situation of (6). Otherwise we have that $\cap_i D_i=D(a,r)$, where $r=\inf_i r_i$. Since the value of each $|t-c|_x$ with $c\in k$ is determined precisely by the conditions $|a-c|\le r_i$ we have that $x$ is the maximal point of $D(a,r)$, that is $x=p_{a,r}$.

Let us describe the conditions which cut off $x$ from the disc. For any other point $y\in D(a,r)$ there exists $c\in k$ such that $|t-c|_y<|t-c|_x$ and hence $y$ is contained in a smaller disc with center at $c$, so $p$ is obtained by removing all discs of smaller radii. It remains to note that if $r$ is not divisorial, then any classical point is contained in a smaller disc with center at $a$, hence it suffices to impose conditions $|t-a|>r'$ or $|t-a|\ge r'$ for $r'\in\Gamma_k$ with $r'<r$. And if $r\in\Gamma_k$, we should remove all distinct discs of radius $r^-$, and this yields the description in (2).

Finally, $t_a^\NN$ is an orthogonal $k$-basis of $k[t]$ with respect to $|\ |_x$, and if $|t_a|=|c|\in\Gamma_k$, we can make it even orthonormal by replacing $t_a$ by $t_a/c$. It easily follows that $|k(x)^\times|=|k^\times|r^\ZZ$ with $r=|t_a|$. Also, $\wkx=\tilk(u)$ for $u=\wt{t_a/c}$ in case (2), and $\wkx=\tilk$ otherwise.
\end{proof}


\subsection{Geometric metric on projective lines}\label{geomsec}
In the sequel we will mainly consider the microbial case, and only remark which (minor) changes are required in the discrete case.

\subsubsection{Sheaves of values}
We saw in Section \ref{orderedsec} that the most economical and conceptual way of encoding the $\Gamma$-metric structure is through sheaves of ordered groups or corresponding valuation monoids. Adic spaces naturally possess such sheaves: the {\em sheaf of absolute values} $|\calO_X^\times|:=\calO_X^\times/(\calO_X^+)^\times$, whose sections are of the form $|f|$ with an invertible function $f$ and stalks are $|\calO_X^\times|_x=|k(x)^\times|$ and its additive version $\Gamma_X=\log(|\calO_x^\times|)$ whose sections should be thought of as functions of the form $\log|f|$ and whose stalks are $\Gamma_{X,x}=\Gamma_x$. Naturally, we call $\Gamma_X$ the {\em sheaf of values} of $X$. We also consider the subsheaves of valuation monoids $|\calO_X^\times|^+=(\calO_X^+\cap\calO_X^\times)/(\calO_X^+)^\times$ and its additive counterpart $\Gamma_X^+=-\log(|\calO_X^\times|)$.

\subsubsection{Geometric sheaves of values}\label{geomsheaf}
Sheaves of values are preserved by automorphisms of the adic space, hence given an $S$-adic space $X$ with $\oX=X\times_S\oS$, the sheaves $\Gamma_\oX^+\subseteq\Gamma_\oX$ descend to sheaves on $X$ that we denote $\oGamma_X^+\subseteq\oGamma_X$ and call the {\em geometric sheaves of values} of $X$. Looking at stalks one easily sees that $$\Gamma_X\subseteq\Gamma_{k,\QQ}+\Gamma_X\subseteq\oGamma_X\subseteq\Gamma_{X,\QQ}.$$ For a subset $T\subseteq X$ we denote by $\oGamma_T=\oGamma_X|_T$, $\Gamma^+_T=\Gamma^+_X|_T$, etc. the restrictions of the corresponding sheaves to $T$.

\subsubsection{The geometric case}
We start with the projective line in the geometric case.

\begin{theor}\label{geomtreeth}
Assume that $k$ is algebraically closed and microbial. Then $X=\PP^{1}_s$ together with $\Gamma^+_X$ forms a pinched $\Gamma_k$-tree, whose vertices are divisorial points, edge points are cut and infinitesimal points, and pinched leaves are classical and unibranch points.
\end{theor}
\begin{proof}
Let us consider first the subset $T=\PP^{1,\mon}_s$ of all monomial (i.e. generic non-unibranch) points. For any two points $x_1,x_2\in T$ we can choose a coordinate $t$ on $\PP^1_S$ such that $x_1,x_2\in[0,\infty]_s$ and $x_1<x_2$. Thus, $x_i=p_{r_i}$ are $t$-monomial of radii $r_i\in\cR_{|k(s)^\times|}$ and $T$ contains the interval $I=[x_1,x_2]$ homeomorphic to the $\Gamma_k$-segment $J=[\log r_1,\log r_2]$. For simplicity, we will identify $I$ and $J$ as topological spaces. Each $\Gamma_k$-pl function on $J$ is locally linear of the form $\gamma+nt$ with $\gamma\in\Gamma_k$ and $n\in\ZZ$, hence it lifts to a section $at^n$ of $\cO_T^\times$ and we obtain an embedding $\Gamma_J\into\Gamma_T|_I$. Since $\Gamma_k$ is divisible, Lemma~\ref{gammasheaflem} and Theorem~\ref{classth} imply that this map induces an  isomorphism on stalks and hence is an isomorphism (the stalks are isomorphic to $\Gamma_k$ at divisorial points and to $\Gamma_k\oplus r^\ZZ$ otherwise). Finally, each point of $I$ separates $x$ and $y$, hence $T$ is a $\Gamma_k$-tree.

It remains to describe the situation at a classical or unibranch point $x\in X$. Note that $x$ is connected to $\infty$ by a unique interval $[x,\infty]$ where $(x,\infty)$ consists of maximal points of discs containing $x$. Furthermore, $x$ coincides with the intersection of discs containing it, and it follows that $x$ is uniquely connected to any other point $y$ by an interval $[x,y]$ -- go along the interval $[x,\infty]$ to the smallest point dominating $y$ and then descend to $y$. So, indeed $X$ is topologically a $\Gamma$-tree with leaves being the classical and unibranch points. Furthermore, the leaves are pinched because $|f|$ is constant in a neighborhood of a unibranch point and it is constant in a neighborhood of a classical point whenever $f$ is invertible.
\end{proof}

\subsubsection{Explicit metric}\label{metricrem}
As observed in the proof (and the argument applies in general, rather than only to unibranch points), the $\Gamma_k$-metric structure of $\PP^1_s$ can be described as follows. Given $x,y\in\PP^1_s$ the interval $[x,y]$ is obtained by going from $x$ to the minimal point $z$ dominating both $x$ and $y$ and then going down from $z$ to $y$, that is, $[x,\infty]\cap[y,\infty]=[z,\infty]$ and $[x,y]=[x,z]\cup[z,y]$ consists of points which dominate one point, but not the other, and the single point $z$ which dominates both $x$ and $y$. If $x,y$ are divisorial, then $x=p_{a,r}$, $y=p_{b,s}$ and $z=p_{a,l}=p_{b,l}$ with $l=|a-b|$, and the length of $[x,y]$ is $\log(l^2/rs)$.

\subsubsection{Radius function}\label{radiusrem}
The radius function $r\:\PP^{1,\mon}_s\to\cR_{|k(s)^\times|}^b$ extends to a function $r\:\PP^1_s\to\cR_{|k(s)^\times|}$ by defining $r(x)$ to be the infimum of radii of discs containing $x$. In particular, $r(x)=0$ if and only if $x\neq\infty$ is classical. This naively looking function is very useful, but one should remember that, as in Berkovich geometry, it is not continuous but only upper semicontinuous: the preimage of an interval $[0,\gamma)$ is a (huge) union of open discs, hence it is open.

\subsubsection{Unpinching}
One can just enlarge the sheaf $\Gamma^+_X$ by hands so that the $\Gamma$-tree $X=\PP^1_s$ becomes non-pinched, but this is a rather meaningless construction, while the natural question is if such an enlargement can be obtained by a geometric construction on $X$. We discuss some natural constructions below.

\begin{rem}\label{logrem}
(i) The metric at a leaf point $x$ is pinched because any function of the form $|f|$ with invertible $f$ is locally constant at $x$. So one can wonder if one can naturally increase the pool of such functions.

(ii) The solution at classical points is straightforward -- consider also absolute values of non-invertible non-zero functions. A nice way to do this is via log geometry. Let $Z$ be a set of classical points and $M\subseteq\calO_X\setminus\{0\}$ the associated log structure: the sheaf of non-zero functions whose zeros lie in $Z$. Then $\Gamma^+_{(X,M)}=\log(M/(\calO_X^+)^\times)$ provides $X$ with a pinched $\Gamma$-tree structure such that a leaf is not pinched if and only if it lies in $Z$. In particular, the trivial log structure $M=\calO_X^\times$ yields the above definition of $\Gamma_X$, while the choice of the maximal log structure $\calO_X\setminus\{0\}$ yields a $\Gamma$-tree in which all classical points are not pinched.

(iii) As for unibranch points, in the geometric case (i.e. $k=k^a$) one can use the following descent trick: find an immediate extension $K/k$ such that $K$ is algebraically closed and spherically complete. Then there is a natural (only) topological embedding $X\into X_K=X\wtimes_kK$ (in the case of Berkovich spaces this is constructed in \cite[Corollaire~3.7 and 3.14]{Poineau}). This embedding sends unibranch points to cut or infinitesimal points, since $X_K$ contains no unibranch points. Thus, the ``missing absolute values'' can be induced from $X_K$.
\end{rem}

Let us elaborate more about unibranch points. This is a bit speculative, as some examples were not worked out in detail in the literature.

\begin{rem}
(i) Even without the descent trick, there are natural functions, such as logarithms of norms of differential forms and different functions that can have a non-zero constant slope locally at a unibranch $x$ (in fact, they were constructed in Berkovich geometry but the definition can be easily extended to the adic case). For example, $\|dt\|$ is the radius function on $X=\PP^1_s$, so its logarithm has slope one at any unibranch point. In the geometric case, these functions are pl at the unibranch points, as can be shown, for example, using the descent trick.

(ii) However, we will show in subsequent papers that if $k$ is not algebraically closed such functions can have infinitely many corners when approaching a unibranch $x$. So, they are pl on the complement of unibranch points, but do not have to be pl at such points. Probably, no meaningful pl structure can be defined at such points.
\end{rem}

\subsubsection{Arbitrary ground fields}
Now we consider the case of an arbitrary affinoid ground field $k$. Let $S=\Spa(k)$, $\oS=\Spa(\whka)$, $s\in S$ and $\os\in\oS$ the closed points, $X$ an adic $S$-space and $\oX=X\times_S\oS$ the corresponding geometric space. The group $G_{k^h}=\Gal(k^s/k^h)$ preserves the valuation of $k^s$ and hence acts on $\oX_\os$.

\begin{lem}\label{hlem}
On the level of topological spaces, $X_s$ is the quotient of $\oX_\os$ by the action of $G$.
\end{lem}
\begin{proof}
We set $S^h=\Spa(k^h)$, $X^h=X\times_SS^h$ and denote by $s_h\in S^h$ the closed point.
The natural morphism $(S^h,s_h) \to (S,s)$ of pseudo-adic spaces is a limit of h-isomorphisms in the sense of paragraph~\ref{h_iso_sec} below.
Hence, also its base change $X^h_{s_h} \to X_s$ is a limit of h-isomorphisms by Lemma~\ref{h-isom}.
In particular,~$X^h$ is naturally homeomorphic to~$X$ by Lemma~\ref{h-isom}.
Therefore, we can assume that $k=k^h$. Now the claim follows by passing to the limit from the fact that $X$ is the quotient of $X_l=X\otimes l$ by $\Gal(l/k)$ for a finite Galois extension $l/k$ since $X_l\to X$ is an open map by \cite[Proposition~1.7.8]{Huber-book}).
\end{proof}

In particular, the lemma implies that $\PP^1_s$ is the quotient of the closed fiber $\PP^1_\os$ of $\PP^1_{\whka}$ by the action of $G_{k^h}$. We will use the notation $\oGamma_k:=\Gamma_{k,\QQ}=\Gamma_{\whka}$.

\begin{theor}\label{treeth}
Let $k$ be a microbial affinoid field, $S=\Spa(k)$ and $s\in S$ a point. Then $X=\PP^{1}_s$ together with $\oGamma^+_X$ is a pinched $\oGamma_k$-tree, whose vertices are divisorial points, edge points are cut and infinitesimal points, and pinched leaves are classical and unibranch points.
\end{theor}
\begin{proof}
One just combines Theorem \ref{geomtreeth} and Lemma~\ref{quotlem}. We use that any disc in $\PP^1_\os$ contains a point from $k^a$, hence the stabilizer of its maximal point is open in $G_{k^h}$.
\end{proof}

\begin{rem}
Lemma \ref{quotlem} also states that the $\oGamma$-distance between divisorial points $x,y\in X$ is the distance between the sets of their preimages in $\oX$.
\end{rem}

\subsubsection{The discrete case}\label{discrsec}
The discrete case can be dealt with similarly. The only difference to the microbial case is that the topological structure at the classical points is slightly different. Let $\cR_s^-$ be obtained from $\cR_s$ by adding the new minimal point $-\infty^2$ which is the specialization of the minimal ranger $-\infty$ of $\cR_s$. In particular, $\cR_s$ is open in $\cR_s^-$. In the discrete case intervals starting at classical points are homeomorphic to intervals in $\cR_s^-$ containing $-\infty^2$.

\subsection{Local degree}\label{degsec}
In Berkovich geometry arbitrary smooth curves can be combined from discs, annuli and finitely many more complicated divisorial points. For adic curves the situation is similar, but discs and annuli should be replaced by their henselian covers that will be called h-discs and h-annuli. Our goal in this and the subsequent subsections is to introduce these notions. Since they make sense in large generality and are important for adic geometry we do not stick to the case of curves. Also, for the sake of comparison we briefly introduce a few more degree functions, not only the henselian one.

\subsubsection{Degree functions}
Let $f\:Y\to X$ be a morphism of adic spaces which is \'etale at $y\in Y$ with $x=f(y)$. Consider the following quantities $$n_{y/x}=[k(y)^h:k(x)^h],\  n^r_{y/x}=[k(y)^u:k(x)^u],\ n^w_{y/x}=[k(y)^t:k(x)^t]$$ which we call the {\em henselian local degree}, the {\em ramified local degree} and the {\em wild local degree} of $f$ at $y$. We will usually omit the word henselian and simply call $n_{y/x}$ the {\em local degree} of $f$ at $y$. Note that the naive residual degree $[k(y):k(x)]$ coincides with the local degree for degree~$1$ points (e.g. on Berkovich spaces), but it is not a valuable invariant in general. If $f$ is \'etale, we define the local, ramified and wild degree functions $Y\to\NN$ by $n_f(y)=n_{y/x}$, $n_f^r(y)=n^r_{y/x}$ and $n_f^w(y)=n^w_{y/x}$, respectively.

\subsubsection{Geometric degree functions}
In addition, we define the local, ramified or wild {\em geometric degree at} $y$ to be $\on_{y/x}=n_{\oy/\ox}$, $\on^r_{y/x}=n^r_{\oy/\ox}$ or $\on^w_{y/x}=n^w_{\oy/\ox}$ where $\oy\in\oY$ and $\ox\in\oX$ are compatible lifts of $x$ and $y$. Clearly, this definition is independent of the lifts, and $\on_{y/x}=[k^ak(y)^h:k^ak(x)^h]$, $\on^r_{y/x}=[k^ak(y)^u:k^ak(x)^u]$, $\on^w_{y/x}=[k^ak(y)^t:k^ak(x)^t]$. If $f$ is \'etale, this gives rise to geometric degree functions $\on_f$, $\on_f^r$ and $\on_f^w$ from $Y$ to $\NN$.

\subsubsection{Degree-one loci}
Note that $n^r_{y/x}=1$ if and only if $f$ is strongly \'etale (or residually unramified) at $y$ and $n^w_{y/x}=1$ if and only if $f$ is residually tame at $y$. These loci are open by \cite[Section~4]{adictame}, and we will show in a sequel paper that as a bootstrap one easily obtains the more general statement that the functions $n_f^r$, $n_f^w$, $\on_f^r$ and $\on_f^w$ are upper semicontinuous.

In this paper, however, we will need to work with the locus of points $y$ with $n_{y/x}=1$. In such a case we will say that $f$ is an {\em h-isomorphism at} $y$. Note that this happens if and only if $\calO_{X,x}^+\to\calO_{Y,y}^+$ is strictly local-\'etale, while $\calO_{X,x}\to\calO_{Y,y}$ is only local-\'etale. Recall that strictness means that the induced homomorphism of residue fields is an isomorphism. In the usual algebraic geometry the strict \'etaleness locus is not open and the only non-tautological though trivial property is as follows: (*) if $f\:Y\to X$ is \'etale, then the strict \'etaleness locus is closed under specializations.

In the case of adic spaces it is very easy to come up with an example of a curve cover, even over a complete field of height one, which is an $h$-isomorphism at an infinitesimal point $y$, but not at its divisorial generization. Analogously to the case of schemes, the $h$-isomorphism locus is not closed under generizations. Surprisingly, this is the only pathology that can happen and we have the following counterpart of (*) in adic geometry. In the sequel by  a {\em maximal point} of a spectral topological space $T$ we mean any point which has no non-trivial generizations in $T$.

\begin{theor}\label{hisomth}
Let $f\:Y\to X$ be a morphism of adic spaces which is an $h$-isomorphism at a generic point $y$. Assume that $T\subseteq Y$ is a spectral subspace such that $y$ is a maximal point of $T$. Then $f$ is an $h$-isomorphism along a neighborhood of $y$ in $T$.
\end{theor}

\begin{rem}
(i) An equivalent formulation is that for an \'etale morphism of pseudo-adic spaces the $h$-isomorphism locus is a neighborhood of any of its maximal point.

(ii) In the case of an \'etale morphism $f\:Y \to X$ of $S$-curves we obtain that the $h$-isomorphism locus $T\subseteq Y_s$ of $f$ is a neighborhood of any of its non-infinitesimal generic point. The same is true for classical points, but is not covered by the theorem. As for infinitesimal points the theorem implies that if $T$ contains an infinitesimal point, then it also contains the branch at $x$ in the direction opposite to the generization.
\end{rem}

\subsubsection{Upper semicontinuity}
We will prove the more general claim about semicontinuity of $n_f$ at maximal points of a pseudo-adic space, which contains Theorem~\ref{hisomth} as the particular case when $n_f=1$.

\begin{theor}\label{semicontth}
Assume that $f\:Y\to X$ is a morphism of adic spaces which is \'etale at a generic point $y$ and $T\subseteq Y$ a spectral subset such that $y$ is a maximal point of~$T$. Then there exists a neighborhood $U$ of $y$ in $T$ such that $n_f(y)\ge n_f(y')$ for any $y'\in U$.
\end{theor}
\begin{proof}
The claim is local at $y$, so we will shrink $Y$ at $y$ and $X$ at $x=f(y)$ when needed. Note that $R=f(T)$ is spectral and shrinking $X$ and $Y$ we can assume that $x\in R$ has no non-trivial generizations in $R$. The finite extension $k(y)/k(x)$ is separable and we denote by $l/k(x)$ its Galois closure and set $G=\Gal_{l/k(x)}$ and $G'=\Gal_{l/k(y)}$. Choose a primitive element $\alp$ for $l/k(x)$ and let $f_\alp(t)\in k(x)[t]$ be its minimal polynomial. Shrinking $X$ (and $Y$) we can assume that $f_\alp\in\calO_X[t]$ defines an \'etale $G$-cover $h\:Z\to X$. The points $z_i$ in the fiber $h^{-1}(x)=\{z_1\.z_m\}$ satisfy $k(z_i)=l$ and correspond to the extensions of the valuation of $k(x)$ to $l$. Choose $z=z_i$ so that its image $y'$ in $Y'=Z/G'$ corresponds to the extension of valued fields $k(y)/k(x)$. Shrinking $Y$ we can achieve that the embedding of valued fields $k(y)=k(y')\into k(z_1)$ extends to an $X$-morphism $Y\to Y'$ sending $y$ to $y'$. Furthermore, it is a local isomorphism at $y$ since $y$ is generic and hence $\calO_{Y',y'}=k(y')=k(y)=\calO_{Y,y}$. So, shrinking $Y$ again we achieve that $Y$ is a neighborhood of $y=y'$ in $Z/G'$. Let $D\subseteq G$ be the stabilizer of $z$, and so $m=[G:D]$. By transitivity of the action of $G$ on $\{z_1\.z_m\}$ all degrees  $n_{z_i/x}$ are equal, and hence $n_{z/x}=|G|/m=|D|$ (of course, this can be also obtained by very basic valuation theory). In the same manner, $n_{z/y}=|D'|$, where $D'=D\cap G'$, and hence $n_{y/x}=n_{z/x}/n_{z/y}=[D:D']$.

Now comes the point. Since $z_1\.z_m$ have no generizations in the spectral topological space $S=h^{-1}(R)$, they possess pairwise disjoint open neighborhoods $S_1\.S_m$ in $S$. For a small enough neighborhood $U$ of $x$ in $R$ we have that $h^{-1}(U)\subseteq\coprod_iS_i$, hence shrinking $X$ we can assume that there is a clopen splitting $S=\coprod_{i=1}^m S_i$ with $z_i\in S_i$. For any point $z'\in S_1$ with images $y'\in Y$ and $x'\in X$ let $H$ be the stabilizer of $z'$ in $G$ and let $H'=H\cap G'$. The same argument as in the previous paragraph shows that $n_{y'/x'}=[H:H']$. Any element of $H$ takes $S_1$ to itself and hence stabilizes $z=z_1$. Therefore, $H\subseteq D$ and $H'=D'\cap H$. By standard group theory $[D:D']\ge[H:H']$, that is, $n_{y/x}\ge n_{y'/x'}$. It remains to note that by the openness of the map $Z\to Y$, the image of $S_1$ in $Y$ is a neighborhood of $y$ in $T$.
\end{proof}

\subsubsection{Ramified points}
In case of smooth curves one naturally extends the definition of $n_f$ to the points where $f$ is not \'etale by taking into account the classical ramification index: $n_f(y)=e_{y/x}[k(y)^h:k(x)^h]$, where $x=f(y)$. For the sake of completeness we describe in the next two sections a more general situation, where the local degree can be defined. It will not be used in the sequel.

\subsubsection{Local degree of morphisms of schemes}
Let $f\:Y\to X$ be a flat morphism of schemes. If $f$ is finite, then the degree over a point $x\in X$ is defined to be the rank of the free $\calO_x$-module $(f_*\calO_Y)_x$. This notion can be localized using henselization: if $f$ is only assumed to be quasi-finite at $y$, then the henselization $\calO^h_x\to\calO_y^h$ is flat and finite, and we define the local degree $n_{y/x}$ to be the $\calO_x^h$-rank of $\calO_y^h$. This is compatible with the global degree: if $f$ is finite over $x$ and $f^{-1}(x)=\{y_1\.y_n\}$ then the degree of $f$ over $x$ is the sum of local degrees $n_{y_i/x}$.

\begin{rem}
A geometric interpretation of this definition is that for a fine enough morphism $X'\to X$ strictly \'etale over $x$ the base change $Y\times_XX'$ contains a component $Y'$ which contains a preimage $y'$ of $y$ and is finite over $X'$. The local degree $n_{y/x}=n_{y'/x'}$ is just the degree of $Y'/X'$. Moreover, the construction works in the slightly larger generality of normalized finite type, see \cite[\S2.5]{Temkin-insepunif}.
\end{rem}

\subsubsection{Local degree for adic spaces}
Given a quasi-finite flat morphism $f\:Y\to X$ of adic spaces we define the {\em (local) degree at} at a point $y\in Y$ with $x=f(y)$ to be $n_{y/x}=n_{\cO_y/\cO_x}[k(y)^h:k(x)^h]/[k(y):k(x)]$, where $n_{\cO_y/\cO_x}$ is the local degree of the morphism $\Spec(\cO_y)\to\Spec(\cO_x)$. In particular, if $f$ is \'etale at $y$ or $y$ is generic we simply have that $n_{y/x}=[k(y)^h:k(x)^h]$.

\begin{rem}
(i) In a sense, we correct the contribution of the residue field extension to the local degree so that only the henselian degree matters. For analytic points, the residue fields are henselian, so no correction is needed and one simply has that $n_{y/x}=n_{\cO_y/\cO_x}$.

(ii) We will not need this, but it is easy to see that $n_{y/x}$ is the $(\cO_x^+)^h$-rank of $(\cO^+_y)^h$, that is, we really compute the local degree of the local rings composed from the local rings of the adic spaces and the valuation rings of the residue fields.
\end{rem}

\subsection{Henselian discs and annuli}\label{henselsec}

\subsubsection{$h$-isomorphisms} \label{h_iso_sec}
Let $(Y,T)\to(X,S)$ be an \'etale morphism of pseudo-adic spaces. We say that $f$ is an {\em $h$-isomorphism} if $f\:T\to S$ is a bijection and $f$ is an $h$-isomorphism at any point of $T$.
\begin{lem}\label{h-isom}
(i) Any $h$-isomorphism $f\:(Y,T)\to(X,S)$ of pseudo-adic spaces induces an isomorphism of monoided spaces $(T,\Gamma_T)\to(S,\Gamma_S)$, and $\wkx=\wky$ for any $y\in T$ with $x=f(y)$.

(ii) The class of $h$-isomorphisms is closed under compositions and base changes.
\end{lem}
\begin{proof}
(i) The map $T\to S$ is a homeomorphism by openness of \'etale morphisms, and for any $y\in T$ with $x=f(y)$ the extension $k(y)/k(x)$ is unramified and hence $\Gamma_{T,y}=\Gamma_{k(y)}=\Gamma_{k(x)}=\Gamma_{S,x}$. The isomorphism $\wkx=\wky$ is imposed by the definition.

(ii) This is a straightforward check.
\end{proof}

\subsubsection{h-discs and h-annuli}
By an {\em h-disc} or an {\em h-annulus} over $s\in S$ we mean a pseudo-adic space $(X,T)$ which admits an $h$-isomorphism $(X,T)\to D_s(0,r)=(\PP^1_S,D)$ or $(X,T)\to A_s(0;r_1,r_2)=(\PP^1_S,A)$, whose target is a disc or an annulus over $(S,s)$. A function $t\in\Gamma(\cO_X)$ inducing such an $h$-isomorphism will be called an {\em h-coordinate}.

\subsubsection{$h$-equivalence}
The notions of $h$-isomorphisms and $h$-discs cover our needs in this paper, so we just outline a natural generalization which will not be needed but is probably worth studying. Since $h$-isomorphisms are non invertible it is natural to define an {\em $h$-equivalence} $(Y,T)\dashrightarrow(X,S)$ of pseudo-adic spaces as a sequence of $h$-isomorphisms and their formal inverses. By Lemma~\ref{h-isom} an $h$-isomorphism induces a well defined isomorphism of monoided spaces $(T,\Gamma_T)\to(S,\Gamma_S)$. In addition, one can easily work with them using the following result.

\begin{lem}
The family of $h$-isomorphisms admits a calculus of right fractions. In particular, any $h$-equivalence $(Y,T)\dashrightarrow(X,S)$ can be realized as a composition $(Y,T)\leftarrow (Z,R)\to (X,S)$ of $h$-isomorphisms.
\end{lem}
\begin{proof}
We should show that an $h$-equivalence $(Y,T)\rightarrow (Z',R')\leftarrow (X,S)$ possesses such a realization. In view of Lemma~\ref{h-isom} we can take $$(Z,R)=(Y,T)\times_{(Z',R')}(X,S).$$
\end{proof}

\begin{rem}
(i) In this paper, $h$-isomorphisms and local degrees naturally show up when working with points of higher rank. In fact, without them we even cannot formulate some results about the structure of adic curves. This indicates that there is a strong geometric incentive to localize the category of pseudo-adic spaces by inverting $h$-isomorphisms. An interesting direction to check is whether this operation can be naturally achieved just by replacing the structure sheaf by its henselization. However, in this paper we decided to solve the problem in a more ad hoc manner via considering $h$-isomorphisms when needed.

(ii) In the same fashion, it is natural to extend the notions of $h$-discs and $h$-annuli by allowing spaces which are only $h$-equivalent to standard discs or annuli. We decided not to work this out because it will be convenient to work with $h$-coordinates. In particular, in Theorem~\ref{locunif} below we will establish existence of $h$-discs in our sense (which is a stronger claim).
\end{rem}

\subsubsection{The skeleton}
Let $(X,T)$ be either an $h$-disc or an $h$-annulus. The {\em boundary} $\partial(T)$ of $T$ is the preimage of the sets of extremal Gauss points $x_r$ or $\{x_{r_1},x_{r_2}\}$. The {\em skeleton} of $T$ is either the boundary point or the $\Gamma$-segment connecting the two boundary points, respectively.

\subsubsection{Constructibility}
We will mainly work with constructible h-discs and h-annuli $(X,T)$. The constructibility is preserved by taking images and preimages hence $(X,T)$ is constructible if and only if the usual disc or annulus it is mapped onto in $\PP^1_s$ is constructible. The latter is easily interpreted in terms of radii, and the conclusion is that a constructible h-disc/h-annulus $T$ has divisorial boundary if and only if it is open and it has infinitesimal boundary without divisorial generization inside $T$ if and only if it is closed.

\subsubsection{Morphisms of h-discs}
We start with studying units on h-discs.

\begin{lem}\label{invdisc}
Let $(X,T)$ be a quasi-compact h-disc over $s\in S=\Spa(k)$ and $f\in\Gamma(\calO^\times_T)$ an invertible function on $T$. Then $|f|$ is constant on $T$.
\end{lem}
\begin{proof}
Assume that, conversely, $|f|$ is not a constant. Extending the ground field we can assume that $k=k^a$. The image of the induced function $f\:T\to\PP^1_s$ is quasi-compact, does not contain 0 and intersects the $\Gamma$-segment $I=[0,\infty]_s$ because $\PP^1_s\setminus I$ is a disjoint union of closed discs on which $|f|$ is a constant. Therefore there exists a generalized Gauss point $p\in I$ of radius $r$, such that $f(T)\cap[0,p]=\{p\}$. In particular, $r=\min_{x\in T}|f|_x$. Let $x\in f^{-1}(p)$. Assume first that $x$ is not the maximal point of $T$.

If $p$ is a cut or infinitesimal point, then $|f|_x=r\in\cR_s\setminus|k^\times|$ and hence the slope of $\log|f|$ at $x$ is non-zero. It follows that for a $\Gamma$-segment $J$ though $x$ the function $|f|$ does not have a local minimum at $x$. A contradiction to the minimality of $r$. If $x$ is a divisorial point, then $\ox$ is identified with $\PP^1_{\wks}$ and hence it is mapped surjectively onto $\overline{\{p\}}=\PP^1_{\wks}$. In particular, $f(T)$ contains the point of $I$ of radius $r^-$, which is, again, a contradiction.

If $x$ is the maximal point of $T$, then the same arguments show that either $p$ is divisorial and $f(T)$ contains the point of radius $r^+$ or $p$ is an edge point and $f(T)$ contains Gauss points of larger radius. In either case, set $g=1/f$, consider the induced function $g\:T\to\PP^1_s$ and note that $g(x)$ is not the minimal point on $g(T)\cap I$. Therefore the same argument as earlier applied to $g$ yields a contradiction.
\end{proof}

Now we can describe morphisms between h-discs. The analogous result in Berkovich geometry is very simple, but we have to work a bit harder because it is harder to explicitly describe functions on h-discs.

\begin{theor}\label{mordiscs}
Let $S=\Spa(k)$ with an algebraically closed affinoid field $k$, let $f\:(X,T')\to(X,T)$ be a morphism between $h$-discs over $s\in S$ and let $p'\in T'$ be the maximal point and $p=f(p')$, then

(i) $f(T')$ is an h-disc with maximal point $p$.

(ii) For any $\Gamma$-segment $I'=[a',p']\subset T'$ its image $I=f(I')$ is the $\Gamma$-segment $[a,p]$ where $a=f(a')$.

(iii) The map $I'\to f(I)$ is a $\Gamma$-pl homeomorphism and if $q'\in I'$ is a non-corner point of the map, then the slope at it coincides with $n_f(q')$. In particular, $n_f$ is locally constant with jumps at divisorial points. Furthermore, if $q'\neq a'$ and $W'$ is the disc with maximal point $q'$, then $n_f(q')$ is the number of preimages of $a=f(a')$ in $W$ counted with multiplicities (or the degree of the map $W'\to f(W')$).

(iv) For any $\Gamma$-segment $I=[a,p]\subset T$ its preimage $f^{-1}(I)$ is the union of the $\Gamma$-segments $[a'_i,p']$ where $f^{-1}(a)=\{a'_1\.a'_n\}$.
\end{theor}
\begin{proof}
Composing $f$ with a function induced by an h-coordinate on $T$ we reduce to the case when $T$ is a usual disc in $X=\PP^1_S$ with center at 0 and $f$ is given by a function $f\in\Gamma(\calO_{X'})$. We start with proving (ii). By continuity it suffices to check this for classical points $a'$. Replacing $f$ by $f-a$ we can assume that $a=0$. Let $a'=a'_1\.a'_n$ be the fiber $f^{-1}(0)$ listed with multiplicities. Choose an h-coordinate $t$ on $T$. Then $f=u\prod_{i=1}^n(t-c_i)$, where $c_i=t(a'_i)$ and $u$ is a unit, and hence $|u|$ is constant by Lemma~\ref{invdisc}, and up to a constant $|f|$ is the product of $|t-c_i|$. This implies that $|f|$ is strictly increasing on $I'$ and hence takes it onto $I=[0,p]$. Thus, (ii) is proved, and it immediately implies (iv).

Next, consider (iii). The above formula for $|f|$ implies that the slope at a point $q'\in(a',p']$ equals the number of those points from $\{a'_1\.a'_n\}$ that lie in $W'$. In particular, the slope function is monotonic on $I$. It remains to relate these numbers to $n_{q'/q}$, where $q=f(q')$.

If $q'$ is classical, then the slope at $q'$ equals the multiplicity $d$ of 0 in $f^{-1}(0)$. Note that $k(q')=k=k(q)$, and $e_{q'/q}=d$ because $t$ and $f$ are local uniformizers at $q'$ and $q$, respectively. So, $n_{q'/q}=d$.

If $q'$ is an edge point, then the same formula yields that $|f|=c|t|^d$ in  a neighborhood of $q'$ on $I'$, where $c\in|k^\times|$. Since $|k(q')^\times|=|k^\times|\oplus |t|_{q'}^\ZZ$ and $|k(q)^\times|=|k^\times|\oplus |f|_q^\ZZ$, we obtain that $d=e_{k(q')/k(q)}$ is the ramification index. Since $k(q)$ is defectless by the stability theorem, see \cite[Theorem~1.1]{defectless}, and $\wt{k(q)}=\tilk$ is algebraically closed, we have that $n_{q'/q}=[k(q')^h:k(q)^h]=e_{k(q')/k(q)}f_{k(q')/k(q)}=d$.

If $q'$ is a divisorial point, then the same argument with the stability theorem shows that $n_{q'/q}=[k(q')^h:k(q)^h]=[\wt{k(q')}:\wt{k(q)}]=f_{k(q')/k(q)}$. Let $r$ be the infinitesimal point on $[0,q]$ lying in the closure of $q$ and let $r'_1\.r'_l$ be its preimages in $T'$. Then the number $d$ of preimages of 0 in $W$ equals the sum of its preimages in the discs with maximal points $r_1\.r_l$. The latter is equal to $\sum_{i=1}^le_{r_i/r}$ by the case of edge points, and comparing the degrees at 0 and the generic point of the map $\PP^1_{\wks}\to\PP^1_{\wks}$ corresponding to $k(q')/k(q)$ we obtain that $d=f_{k(q')/k(q)}$.

Finally, if $q'=a'$ is a unibranch point, then the slope $d$ is constant in a neighborhood of $a'$ because it is non-decreasing on $I$. Shrinking the disc $T$ around $a$ and replacing $T'$ by the preimage we can achieve that the slope is precisely $d$ along the whole $I$. In particular the map is generically of degree $d$, and hence $n_{a'/a}\le d$. On the other hand, for any edge point $q'\in I'\cap U'$ with $q=f(q')$ we have that $n_{q'/q}=d$ by the case of edge points proved above. So, by Theorem~\ref{semicontth} we obtain that $n_{a'/a}\ge d$.

It remains to prove (i). By continuity, it suffices to prove that $f(T')$ contains each classical point $a\in T$. Translating $f$ by $a$ we can assume that $a=0$. If $0\notin f(T')$, then $f$ is invertible on $T'$ and hence $|f|$ is a constant by Lemma~\ref{invdisc}. It follows that $f(T)$ does not contain the infinitesimal point in the closure of $p$ pointing towards 0. This is impossible because in view of (ii) the map $\PP^1_{\wks}\to\PP^1_{\wks}$ corresponding to $k(p')/k(p)$ takes $\AA^1_\wks$ to $\AA^1_\wks$, and the induced map $\AA^1_\wks\to\AA^1_\wks$ is necessarily surjective.
\end{proof}

\subsubsection{Morphisms of h-annuli}
In the case of annuli there also exist units coming from h-coordinate functions.

\begin{lem}\label{invannulus}
Let $(X,T)$ be a quasi-compact h-annulus over $s\in S=\Spa(k)$ and fix an h-coordinate $t$ on $T$. Then for any unit $f\in\Gamma(\calO^\times_T)$ there exists $n\in\ZZ$ and $r\in \lvert k^\times \rvert$ such that $|f|=r|t^n|$ on $T$. In particular, $\log|f|$ is linear on the skeleton $e$ of $T$.
\end{lem}
\begin{proof}
By Lemma \ref{invdisc} $|f|$ is constant on each connected component of $T\setminus e$, hence it suffices to show that there exists $n$ such that $|f|=r|t|^n$ on $e$, that is, $\log|f|$ is linear on $e$. If this is not so, then there exists a point $x\in e$ such that $\log|f|$ has a break at $e$. In particular, $x$ is a divisorial point not on the boundary, and hence $\ox=\PP^1_{\wks}$. The sum of slopes of $\log|f|$ over all branches from $x$ is zero (this follows from the fact that the sum of orders of a meromorphic function on $\PP^1_{\wks}$ is zero), hence there exists an infinitesimal point $y\in\ox$ such that $y\notin e$ and the slope of $\log|f|$ along the branch given by $y$ is non-zero. It follows that $|f|$ is non-constant on the closed disc with maximal point $y$, and this contradicts Lemma~\ref{invdisc}
\end{proof}

\begin{cor}\label{annulimor}
Assume that $k=k^a$. Let $f\:(X',T')\to(X,T)$ be a morphism whose source is an $h$-annulus with skeleton $e'$ and whose target is either an h-annulus or an h-disc. Then $f(e')$ is a finite subgraph of $T$, the map $e'\to f(e')$ is $\Gamma$-pl and its dilation factor at a non-corner point $q'\in e'$ is $n_f(q')$.
\end{cor}
\begin{proof}
Composing $f$ with an h-coordinate we can assume that $f$ is a map to a usual disc. First, let us prove the claim locally at an edge point $q'\in e'$ with $q=f(q')$. Replacing $f$ by $f-c$ with $c\in k$ we can achieve that $q\in[0,\infty)_s$. Since $q'$ is the intersection of h-subannuli of $T'$, shrinking  $T'$ we can assume that $f$ does not vanish on $T'$ and then by Lemma~\ref{invannulus} $|f|=r|t|^n$ on $T'$, where $t$ is an h-coordinate on $T'$ and $r\in|k^\times|$. Also, $n\neq 0$ because $q\in[0,\infty]_s$, and replacing $f$ by $f^{-1}$ if needed we can achieve that $n>0$. Since $\log|t|$ is a coordinate on the $\Gamma$-interval $e'$ and $\log|f|=n\log|t|+\log(r)$, we obtain that $f$ maps $e'$ injectively into $[0,\infty]_s$ and with dilation $n$.

We have proved that $f$ is $\Gamma$-pl at edge points. It is also $\Gamma$-pl at divisorial points, as the latter are open. So, by quasi-compactness of $e'$ we obtain that $f(e')$ is a finite $\Gamma$-graph and $e'\to f(e')$ is a $\Gamma$-pl map. Comparison of the dilation factor and $n_f$ is done precisely as in the proof of Theorem~\ref{mordiscs}(iii).
\end{proof}

\subsection{Local uniformization of smooth adic curves}\label{locunifsec}
In this subsection we will describe a basis of neighborhoods of points on adic curves over a geometric point $(S,s)$. It is an analogue of the local description (or uniformization) of Berkovich curves over an algebraically closed $k$ via elementary neighborhoods, see \cite[\S3.6]{berihes}.

\subsubsection{Local uniformization of fields}
The really non-trivial ingredient we will use is the following uniformization result for one-dimensional extensions of valued fields  proved in \cite[Theorem~2.1.8]{Temkin-stable}:

\begin{theor}\label{valunif}
Assume that $K/k$ is a finitely generated extension of transcendence degree one of valued fields with an algebraically closed $k$. Then there exists an element $t\in K$ such that $K/k(t)$ is unramified.
\end{theor}

\subsubsection{Elementary neighborhoods}
An open quasi-compact subset $U\subseteq X_s$ is called {\em elementary} if there exists a divisorial point $x\in U$ such that $U\setminus x$ is a disjoint union of closed h-discs and at most one semi-open h-annulus. Open h-discs and h-annuli are particular cases of elementary sets.

\subsubsection{Local uniformization of adic points}
Now we can formulate the main result of \S\ref{locunifsec}.

\begin{theor}\label{locunif}
Let $X$ be a smooth adic curve over $S=\Spa(k)$ with an algebraically closed affinoid field $k$ and let $s\in S$. Then any point $x\in X_s$ possesses a basis of neighborhoods $\calU$ whose elements are affinoid and elementary. Moreover, depending on the type of $x$ one can achieve that each $U\in\calU$ is as follows:

(i) an $h$-disc whenever $x$ is classical, unibranch or unbounded,

(ii) an $h$-annulus whenever $x$ is a cut point,

(iv) such that $U\setminus\{x\}$ is a union of h-discs whenever $x$ is divisorial,

(v) such that $U\setminus\{y\}$ is a union of h-discs and an h-annulus containing $x$, if $x$ is infinitesimal and $y$ is its generization.
\end{theor}

The proof studies a few cases and takes quite a few subsections below, but let us deduce an important corollary first.

\begin{cor}\label{locunficorcor}
Under the assumptions of Theorem~\ref{locunif} $X_s$ and $\Gamma^+_{X_s}$ is a $\Gamma$-graph with pinched leaves.
\end{cor}
\begin{proof}
It suffices to show that an elementary domain $U$ is a $\Gamma$-tree with pinched leaves. We know that this is so for the connected components of $U\setminus\{x\}$, where $x$ is an appropriate divisorial point. Any two points $y,z$ lying in different components of $U\setminus\{x\}$ are connected by a unique quasi-interval $I=[y,y']\cup\{x\}\cup[z',z]$, where $y'$ and $z'$ are the specializations of $x$ lying in the components of $y$ and $z$, and $I$ is a $\Gamma$-interval because $\Gamma_{X,x}=\Gamma_k$ and $x$ is open in $I$.
\end{proof}

\subsubsection{Classical points}
Assume first that $x$ is classical. Choose a uniformizer $t\in\calO_{X,x}$ and consider the induced map $f:U\to W=\PP^1_S$ from a neighborhood of $x$. Shrinking $U$ and $W$ around $x$ and $0$ we can achieve that $W$ is a disc, the induced morphism $g\:U\to W$ is a finite morphism and $x$ is the only preimage of $0$. As~$t$ is a uniformizer of $\calO_{X,x}$, the morphism~$f$ is étale at~$x$. Its degree over~$0$ equals~$1$, so the overall degree is~$1$ and~$f$ is an isomorphism.

\subsubsection{Generic points: choice of a parameter}
Assume now that $x$ is generic. We will show below that there exists a dense subfield $K\subseteq k(x)$ finitely generated over $k$ and of transcendence degree one. Of course, this is an issue only in the microbial case, as otherwise one can just take $K=k(x)$. Given such a $K$ we will use the previous theorem to take any $t\in K$ such that $K/k(t)$ is unramified. We will then shrink $X$ so that $t$ induces a morphism $f\:X\to \PP^1_S$ and will study $x$ through its relation to $y=f(x)$. Note that $f$ is strongly \'etale at $y$, and if $y$ is not divisorial, then $\wt{k(x)}=\tilk$ and hence $f$ is even an $h$-isomorphism at $y$.

To construct $K$ start with any $t\in k(x)\setminus k$. Locally at $x$ it induces a quasi-finite morphism $f\:U\to \PP^1_S$. Set $z=f(x)$, then $k(x)$ is finite over $k(z)$. We claim that one can also achieve that $k(x)/k(z)$ is separable, so assume for a while that $p=\cha(k)>0$. The $p$-rank can only drop under completions, hence $[k(z):k(z)^p]\le p$. On the other hand, extracting a $p$-th root of $t$ induces a radicial cover of $\PP^1_k$, hence $t\notin k(z)^p$ and $[k(z):k(z)^p]=p$. It follows that also $[k(x):k(x)^p]=p$, and choosing now $t\in k(x)\setminus k(x)^p$ we achieve that $k(x)=k(x)^p(t)=k(x)^pk(z)$ and hence $k(x)/k(z)$ is separable. Now, it easily follows from Krasner's lemma that there exists a finite separable extension $K/k(t)$ such that $K$ is dense in $k(x)$.

\subsubsection{Cut, unibranch and unbounded points}
If $x$ is cut, unibranch or unbounded, then it is maximal and hence by Theorem \ref{hisomth} $f$ establishes an $h$-isomorphism $U\to V$ of a neighborhood of $x$ onto a neighborhood of $y$. If $y$ is unibranch, then it is an intersection of discs $V_i$ in $\PP^1_s$. If $y$ is a cut point, then it is an intersection of annuli $V_i$ whose skeleton contains $y$. If $y$ is an unbounded point, then it is an intersection of punched discs, where we punch out the specialization of $y$ in $\PP^1_s$. A cofinal family of these (punched) discs and annuli lies in $V$, and their preimages $U_i\subseteq U$ are (punched) h-discs or h-annuli as required.

\subsubsection{Divisorial points}
If $x$ is divisorial, then we can choose the coordinate $t$ so that $y$ is the Gauss point of $Y:=(\PP^1_k)_s$. Shrinking $X$ we can assume that $X=\Spa(B)$ is affinoid, $x$ is the only preimage of $y$, and $f(X)$ is contained in the unit disc $D=\Spa(A)$, where $(A,A^+)=(k\langle t\rangle,k^+\langle t\rangle)$. Consider the specialization (or reduction) map $D\to\tilD=\Spec(A^+/k^{\circ\circ}A^+)=\AA^1_{k_\eta^+}$. It restricts to the map $D_s\to\tilD_s=\AA^1_{\wt{k(s)}}$ whose fiber over the generic point is the Gauss point, and whose fiber over a closed point $\tilt=\tila$ is the closed disc $D(a,1^-)$ for a lift $a \in k$ of~$\tila$.

Note that $y$ is the intersection of affinoid neighborhoods of the form $V_i=\Spa(A_i)$ with $A_i=A\langle(t-a_1)^{-1}\.(t-a_n)^{-1}\rangle$.
In particular, replacing $Y$ by some $V_i$ and shrinking $X$ accordingly, we achieve that $X\to Y$ is strongly \'etale.
The base change of $X$ to the henselization~$Y_y^h$ of~$Y$ at~$y$ has a connected component containing a preimage of~$x$ and this connected component is then finite over~$Y_y^h$ and contains precisely one preimage of~$x$.
We can descend this to a finite level, i.e., we can find a strongly étale morphism $Y' \to Y$ that is an $h$-isomorphism at~$y$ such that $Y' \times_Y X$ has a connected component~$X'$ with a unique preimage~$x'$ of~$x$ such that $f':X' \to Y'$ is finite and $[k(x'):k(y')] = [k(x)^h:k(y)^h]$.
In particular $f':X' \to Y'$ is finite strongly étale of degree $[k(x)^h:k(y)^h]$.
Using Theorem~\ref{hisomth} we shrink~$Y'$ to achieve that $Y' \to Y$ is an h-isomorphism.
Using Lemma~\ref{h-isom}~(ii) we check that $X' \to X$ is then an h-isomorphism as well.
The map $\tilf'\:\tilX'\to\tilY'$ of reductions is also finite, and shrinking $Y'$ further we can assume that it is unramified over $\tilY'_s$.
Now the preimage of any disc $D(a,1^-) \subseteq Y$ in~$Y'$ is an $h$-disc over $D(a,1^-)$ and in~$X'$ it splits completely into $n_{y/x}$ many h-discs over $D(a,1^-)$ corresponding to the preimages of $\tila \in \tilY'$ in $\tilX'$.
Since $X' \to X$ is an h-isomorphism, the preimage of $D(a,1^-)$ in~$X$ also is a disjoint union of h-discs over $D(a,1^-)$.
Therefore, $f$ splits over any disc of $Y\setminus\{y\}$ and all connected components of $X\setminus\{x\}$ are h-discs mapped h-isomorphically onto discs in $Y$.

\subsubsection{Infinitesimal point}
If $x$ is infinitesimal and $x'$ its divisorial generization, then we choose $t$ so that $x$ is mapped to the generalized Gauss point $y$ of radius $1^-$ and $x'$ is mapped to the Gauss point $y'$. By the same arguments as in the case of divisorial points, shrinking $X$ we can achieve that $X\setminus\{x'\}$ is a disjoint union of h-discs and the fiber over $D(0,1^-)$ (which we cannot remove because of $x$). However, $x$ is maximal in $T=f^{-1}(D(0,1^-))$, hence $f$ is an $h$-isomorphism in a neighborhood of $x$ in $T$. Since the annuli $A(0;r,1^-)$ with $r<1$ form a fundamental family of neighborhoods of $y$ in $D(0,1^-)$, for $r\in|k^\times|$ close enough to~$1$ the preimage $f^{-1}(A(0;r,1^-)$ contains a component $A$ which is a neighborhood of $x$ in $T$ and mapped h-isomorphically onto $A(0;r,1^-)$. The required neighborhood of $x$ is then of the form $(X \setminus T) \cup A$.

\subsection{Global structure of adic curves}
Our next task is to globalize the local description provided by the local uniformization theorem. We stick to the case when $k$ is microbial, since the $\Gamma$-graph structure is nicer in this case.

\subsubsection{Skeletons}
Assume that $(X,X_s)$ is a reduced curve over a geometric pseudo-adic point $(S,s)$. By a {\em skeleton} of $X_s$ we mean a finite $\Gamma$-graph $\Delta\subset X_s$ such that $X_s\setminus\Delta$ is a disjoint union of constructible closed h-discs and the leaves of $\Delta$ are either divisorial or classical points. In particular, $\Delta$ is a skeleton of the underlying $\Gamma$-graph of $X_s$ in the sense of \S\ref{skeletsec} because for any point $x\in X_s\setminus\Delta$ the connected component $Y$ containing it is a constructible disc whose maximal infinitesimal point $y^-$ has a generization $y\in\Delta$ and the $\Gamma$-segment $[x,y^-]$ extends to a $\Gamma$-segment $[x,y]$, which is the minimal $\Gamma$-segment connecting $x$ and $\Delta$.

\subsubsection{Triangulations}
For shortness, we will address punched h-discs as semi-infinite h-annuli; their {\em skeleton} is the interval connecting the maximal point of the disc with the punch. Following Ducros, by a {\em triangulation} of a fiber $X_s$ of an $S$-curve $X$ we mean a finite set $V$ of divisorial and classical points of $X_s$ such that $X_s\setminus V$ is a disjoint union of constructible closed h-discs and h-annuli. To any triangulation one canonically associates the skeleton $\Delta_V\subset X_s$ whose vertices are the points of $V$ and edges are $\Gamma$-segments, which are skeletons of the h-annular components of $X_s\setminus V$.

In fact, any skeleton $\Delta$ is associated with a triangulation $V$, but proving this is essentially equivalent to proving the triangulation theorem below. But if a triangulation of $V$ is given, then it is easy to see that adding to it any finite set of divisorial points of $\Delta$ we obtain another triangulation whose associated skeleton is $\Delta$.

\subsubsection{Refinements}
Also, constructing a refinement of a skeleton is easy.

\begin{lem}\label{skeletondisc}
Let $(S,s)$ be a geometric pseudo-adic point and $(X,X_s)$ a reduced $s$-curve with a triangulation $V$ and associated skeleton $\Delta$. Assume that $\Delta'\subset X_s$ is a finite $\Gamma$-subgraph and $V'\subset\Delta'$ is a finite set of classical and divisorial points such that $\Delta\subseteq\Delta'$, $\pi_0(\Delta)=\pi_0(\Delta')$, $V\subseteq V'$ and $V'$ contains all points of $\Delta'$ of valence different from 2. Then $V'$ is a triangulation of $X_s$ and $\Delta'$ is the associated skeleton.
\end{lem}
\begin{proof}
Since $X_s\setminus\Delta$ is a disjoint union of $\Gamma$-trees, for any point $x\in X_s$ there exists a unique $\Gamma$-segment $[x,x']$ (maybe with $x=x'$) such that $[x,x']\cap\Delta=\{x'\}$. Furthermore, the explicit description of discs and annuli implies that $V\cup\{x,x'\}$ is a triangulation and $\Delta\cup[x,x']$ is the associated skeleton. In particular, we can enlarge $V$ by adding any divisorial point of the skeleton or any divisorial or classical point of a disc component of $X_s\setminus\Delta$ together with the point of $\Delta$ of minimal distance from $X_s$. The claim of the lemma now follows by a simple induction on the number of points in $V'\setminus V$.
\end{proof}

\subsubsection{Reduction to the smooth case}
Given a reduced curve $X$, let $\pi\:\tilX\to X$ be its normalization. In particular, if $D$ is the singular locus of $X$, then $D_s=D\cap X_s$ and $\tilD_s=\pi^{-1}(D_s)$ are finite sets of classical points and $\tilX_s\setminus\tilD_s=X_s\setminus D_s$. Since any triangulation of $X_s$ must contain $D_s$ we obtain that there is a natural bijection between triangulations of $X_s$ and triangulations of $\tilX_s$ containing $\tilD_s$. For this reason in the sequel we will only study triangulations of smooth curves.

\subsubsection{The triangulation theorem}
Here is the main global result over a geometric point. Its analogue in Berkovich geometry is a version of the analytic semistable reduction sometimes called the triangulation theorem.

\begin{theor}\label{triangth}
Let $k$ be an algebraically closed affinoid field, $S=\Spa(k)$, $s\in S$ a point and $X$ a smooth $S$-curve. Then any finite subset $V_0\subset X_s$ of classical and divisorial points is contained in a triangulation $V$ of $X_s$. Furthermore, if $V_0$ itself contains a triangulation, then there exists a minimal triangulation $V$ containing $V_0$.
\end{theor}
\begin{proof}
Working separately with the connected components, we can assume that $X_s$ is connected. Assume first that $V_0$ contains a triangulation $V'$, and let $\Delta'$ be the associated skeleton of $X_s$. Let $\Delta$ be the minimal connected $\Gamma$-graph containing $\Delta'$ and the elements of $V_0$, it exists because $X_s\setminus\Delta'$ is a disjoint union of $\Gamma$-trees. Finally, let $V$ be the union of $V_0$ and all elements of $\Delta$ of valence different from 2. Then Lemma~\ref{skeletondisc} implies that $V$ is the minimal triangulation containing $V_0$ and $\Delta$ is its associated skeleton.

If $V_0$ is arbitrary we will have to use local uniformization. By quasi-compactness, we can find a finite cover $X_s=\cup_{i=1}^n U_i$ with each $U_i$ as in Theorem~\ref{locunif}. Each $U_i$ has a canonical triangulation: $V_i=\{x_i\}$ if $U\setminus\{x_i\}$ is a union of h-discs, and $V_i=\{x_i,y_i\}$ if $U\setminus\{x_i\}$ also contains an h-annulus $A$ with $y_i$ being the second end-point of $A$. Let $\Delta_i$ be the associated skeleton -- $\{x\}$ or $[x,y]$. Choose any connected finite graph $\Delta$ with divisorial vertices which contains $\cup_i\Delta_i$. For example, for any $U_i,U_j$ with a non-empty intersection we can choose a divisorial point $v_{ij}\in U_i\cap U_j$ and segments $l_{ij}=[v_{ij},v_i]\subset U_i$ and $l_{ji}=[v_{ij},v_j]\subset U_j$. Let $\Delta$ be the union of all $\Delta_i$ and $l_{ij},l_{ji}$ and let $V$ be the union of $V_0,V_1\.V_n$ and all elements of $\Delta$ of valence not equal to 2. We claim that $V$ is a triangulation the associated skeleton $\Delta$. This reduces to a local check in an h-disc or h-annulus lying in some $U_i$ and follows from the same claim as earlier.
\end{proof}

\subsubsection{Arbitrary ground fields}\label{deltasec}
Now, let us study curves over an arbitrary affinoid field $k$. As was noticed in the introduction, the natural sheaf of values $\Gamma_X$ can behave rather pathologically this time, so we will simply use the geometric metric $\oGamma_X$ induced from $\oX=X\times_S\oS$, see \S\ref{geomsheaf}. By a {\em geometric skeleton} we mean a finite $\oGamma$-graph $\Delta\subset X_s$ whose preimage $\oDelta$ in $\oX_\os$ is a skeleton. Note that $\oDelta$ is $G_{k^h}$-equivariant, hence the deformational retraction of $\oX_\os$ onto $\oDelta$ induced by the metric descends to the deformational retraction of $X_s$ onto $\Delta$. In particular, $\Delta$ is a graph skeleton of $X_s$ and the obtained retraction onto it is induced by the geometric metric on $X_s$.

A {\em geometric triangulation} $V$ of $X_s$ is a finite set of divisorial and classical points such that its preimage $\oV$ in $\oX_\os$ is a triangulation and the action of $G_{k^h}=G_{k^s/k^h}$ on the associated skeleton $\oDelta$ is such that for each edge $\overline{e}$ and $\sigma\in G_{k^h}$ either $\sigma$ takes $\overline{e}$ to another edge or acts trivially on $\overline{e}$. Since $\oDelta$ is $G_{k^h}$-equivariant, it descends to a geometric skeleton $\Delta$ {\em associated with} $V$. The condition on the action guarantees that $V$ contains all non-edge points of $\Delta$. In principle it is not necessary and is omitted, for example, in \cite[p.34]{berbook}.

\begin{theor}\label{curvesekeletonth}
Let $k$ be a microbial affinoid field, $S=\Spa(k)$, $s\in S$ a point, $\oGamma=\Gamma_{s,\QQ}$ and $X$ a smooth $S$-curve. Then:

(i) $(X_s,\oGamma_{X_s})$ is a $\oGamma$-graph with pinched leaves, with vertices being the divisorial points, the edge points being the cut and infinitesimal points, and the pinched leaves being the unibranch and classical points.

(ii) For any finite set of classical and divisorial points $V_0\subset X_s$ there exists a geometric triangulation $V$ containing $V_0$. If $V_0$ itself contains the image of a triangulation $\oV_0$ of $\oX_\oS$, then there exists a minimal such $V$.
\end{theor}
\begin{proof}
We start with proving (ii). Let $\pi\:\oX_s\to X_s$ be the projection. By Theorem~\ref{triangth}, enlarging $V_0$ we can assume that it contains the image of a triangulation $\oV_0$, so it suffices to prove only the second part of (ii). The stabilizers in $G_{k^h}$ of non-unibranch points are open, hence the set $\pi^{-1}(V_0)$ is finite and by Theorem~\ref{triangth} there exists a minimal triangulation $\oV'$ containing it. Clearly, $\oV'$ is $G_{k^h}$-equivariant. Since $G_{k^h}$ preserves the associated skeleton $\oDelta$ and the metric, any element $\sigma\in G_{k^h}$ taking an edge $\overline{e}$ of $\oDelta$ to itself either acts by identity or by reflection. For each $\overline{e}$ possessing such a reflecting $\sigma_{\overline{e}}\in G_{k^h}$ we add the midpoint of $\overline{e}$ to $\oV'$ obtaining a larger triangulation $\oV$ with the same skeleton $\oDelta$. Obviously, $\oV$ is the minimal triangulation of $\oX_s$ which descends to a triangulation $V$ of $X_s$.

Furthermore, $(\Delta=\oDelta/G_{k^h},\oGamma|_\Delta)$ is just the quotient of $(\oDelta,\Gamma|_\oDelta)$ obtained by identifying some vertices and some edges, so it is a finite $\oGamma$-graph. Moreover, each connected component $D$ of $X_s\setminus\Delta$ is a quotient of a corresponding component $\oD$ of $\oX_\os\setminus\oDelta$ by its stabilizer. The maximal infinitesimal point of $\oD$ is fixed by the action and the stabilizers of divisorial points are easily seen to be open (e.g. using a quasi-finite map to $\PP^1_s$). By Lemma~\ref{quotlem} $D$ is a pinched $\oGamma$-tree. This shows both that $T$ is a $\oGamma$-graph and $\Delta$ is its skeleton.
\end{proof}

\subsubsection{The closure of a divisorial point}

Let~$x$ be a divisorial point of a smooth $S$-curve~$X$ and write~$s$ for the image of~$x$ in~$S$.
In Corollary~\ref{specializations_as_curve} we identified $\ox \setminus X_s$ with the microbial curve~$\tilX^x_\an$.
Being a smooth curve over $(\widetilde{k(s)},\widetilde{k(s)}^+)$, the space $\tilX^x_\an$ comes with a geometric sheaf of values that we denote by~$\overline{\Gamma}_{\tilX^x_\an}$.
We can view it as a sheaf on $\ox \setminus X_s$.
On the other hand, we can restrict~$\overline{\Gamma}_X$, the geometric sheaf of values of~$X$, to $\overline{x}$.
These two sheaves are connected as follows.

\begin{prop}
 There is a natural short exact sequence of sheaves on $\ox$
 \[
  0	\longrightarrow \overline{\Gamma}_{\tilX^x_\an} \longrightarrow \overline{\Gamma}_X \longrightarrow \overline{\Gamma}_x \longrightarrow 0.
 \]
\end{prop}

\begin{proof}
 We can prove the claim on $\overline{X} = X \times_S \overline{S}$, then it descends to~$X$.
 Therefore, we may assume that the base field~$k$ is algebraically closed.
 Since all points of $\ox$ are specializations of~$x$, there is a natural injection
 \[
  \cO_{X,x}^+ \hookrightarrow \cO_X
 \]
 of the constant sheaf $\cO_{X,x}^+$ into the sheaf~$\cO_{X}$ (restricted to $\ox \setminus X_s$).
 It induces an injection
 \[
  (\cO_{X,x}^+)^\times/(\cO_X^+)^\times \hookrightarrow \cO_X^\times/(\cO_X^+)^\times
 \]
 with cokernel $\cO_X^\times/(\cO_{X,x}^+)^\times$.
 We have
 \[
  \Gamma_x = \cO_X^\times/(\cO_{X,x}^+)^\times, \qquad \Gamma_X = \cO_X^\times/(\cO_X^+)^\times.
 \]
 Moreover, the projection $(\cO_{X,x}^+)^\times \twoheadrightarrow \cO_{\tilX^x_\an}^\times$ induces an isomorphism
 \[
  (\cO_{X,x}^+)^\times/(\cO_X^+)^\times \overset{\sim}{\longrightarrow} \cO_{\tilX^x_\an}^\times/(\cO_{\tilX^x_\an}^+)^\times = \Gamma_{\tilX^x_\an}.
 \]
\end{proof}

\subsubsection{Morphisms and the pl structure}
As a corollary of semistable reduction we can prove that morphisms respect the pl structure in the following sense.

\begin{theor}\label{morth}
Let $f\:Y\to X$ be a quasi-finite morphism of smooth pseudo-adic curves over $s\in S$ with a microbial $S=\Spa(k)$.

(i) If $\Delta\subset Y$ is a finite $\oGamma$-subgraph, then $\Delta_X=f(\Delta)$ is a finite $\oGamma$-subgraph of $X$, the induced map $\Delta\to\Delta_X$ is $\oGamma$-pl and for any non-divisorial point $y\in\Delta$ the dilation of $f|_\Delta$ at $y$ equals $\on_f(y)$.

(ii) If $\Delta\subset X$ is a finite $\oGamma$-subgraph, then $f^{-1}(\Delta)$ is a finite $\oGamma$-subgraph of $Y$.
\end{theor}
\begin{proof}
The claim easily reduces to the geometric case, so assume that $k=k^a$. Choose a triangulation $V$ of $X$ and a triangulation $W$ of $Y$ such that $f^{-1}(V)\subseteq W$. Then the claim reduces to the particular case of maps between h-annuli and h-discs, and this is covered by Theorem~\ref{mordiscs} and Corollary~\ref{annulimor}.
\end{proof}

\subsubsection{Simultaneous semistable reduction}
Also, we now obtain the following version of semistable reduction for a quasi-finite morphism of curves $f\:Y\to X$. By a {\em geometric triangulation} of $f_s$ we mean geometric triangulations $W$ and $V$ of $Y_s$ and $X_s$ such that $f^{-1}(V)=W$. It is easy to see that for any connected component $D$ of $X\setminus V$ its preimage $f^{-1}(D)$ is a disjoint union of finitely many connected components $D_i$ and the preimage of the geometric skeleton of $D$ is the union of the geometric skeletons of $D_i$. In particular, $f^{-1}(\Delta_V)=\Delta_W$. If $f$ is generically \'etale, one might also want $W$ to include all ramification points, but we do not impose this condition here.

\begin{theor}\label{simulth}
Let $k$ be an affinoid field, $S=\Spa(k)$, $s\in S$ a point and $f\:Y\to X$ a quasi-finite morphism of adic $S$-curves. Then for any finite sets of classical and divisorial points $V_0\subset X_s$ and $W_0\subset Y_s$ there exists a geometric triangulation $(W,V)$ of $f_s$ such that $V_0\subseteq V$ and $W_0\subseteq W$. Moreover, if $V_0$ contains a geometric triangulation of $X_s$ and $W_0$ contains a geometric triangulation of $Y_s$, then there exists a minimal such $(W,V)$.
\end{theor}
\begin{proof}
Enlarging $V_0$ and $W_0$ we can assume that they contain geometric triangulations $\tilV$ and $\tilW$ of $X_s$ and $Y_s$, respectively. So, it suffices to prove only the second claim. Let $V$ be the minimal geometric triangulation containing $V_0\cup f(W_0)$. To simplify the notation, we will increase it throughout the proof.

By Theorem~\ref{morth} $f(\Delta_\tilW)$ is a finite $\Gamma$-graph whose vertices are divisorial and classical points. It is contained in the union of $\Delta_V$ and finitely many connected components of $X_s\setminus\Delta_V$. Hence using Lemma~\ref{skeletondisc} it is easy to see that there exists a minimal way to replace $V$ by a larger geometric triangulation so that $f(\Delta_\tilW)\subset\Delta_V$. It then follows that $f^{-1}(\Delta_V)$ is a connected graph in $Y_s$ containing $\Delta_\tilW$ and with leaves at divisorial and classical points, and hence it is a geometric skeleton itself.

It still can happen that $f^{-1}(V)$ is not a triangulation, but we will now increase the triangulations without changing the skeletons: first we add to $V$ the images of all non-edge points of $f^{-1}(\Delta_V)$ and then define $W$ to be the preimage of $V$. It is easy to see that at this stage one finally gets a triangulation of $f$ and all our steps were forced -- any geometric triangulation containing $(W_0,V_0)$ has to contain all points we have added.
\end{proof}

\begin{rem}
It is not difficult to extend the results of this section to the case when $k$ is discrete. As explained in \S\ref{discrsec} one only has to include unbounded points at the end of unbounded segments. Also, one can extend our results to weak curves, since such a curve can be obtained from an adic curve by adding finitely many infinitesimal points of valence 1. For example, in the triangulation theorem one should also allow  connected components consisting of a single infinitesimal point, and, as a matter of definition, it makes sense to include it as a leaf into the skeleton.
\end{rem}

\section{Skeletons as valuative graphs}

In section \ref{metricsec} we have investigated the structure of the fibers of a smooth curve~$X$ over $\Spa(k,k^+)$.
In particular, we constructed skeletons and proved local uniformization for algebraically closed base fields.
However, up until now, we have studied the fibers separately, so the skeletons in each fiber do not have to be related.
In this section we will develop a framework that allows us to construct simultaneous skeletons in all fibers.
We do this by introducing a new type of space that we call valuative graph, which is glued out of certain spectra of monoids.
This results in an enhanced notion of a skeleton: a subset of~$X$ that carries the structure of a valuative graph and whose fibers are skeletons in the sense of section \ref{metricsec}.

\subsection{The unit interval}

\subsubsection{Augmented monoids}
We work in the category of \emph{augmented monoids} \footnote{Berkovich calls them \emph{$\FF_1$-algebras.}}, i.e. commutative monoids with a unit~$1$ and a contracting element~$0$.
Often the augmented monoids we use arise by adding~$0$ to a monoid~$M$ with unit.
In this case we use the notation $M_\star := M \cup \{0\}$.

\subsubsection{The monoid of functions}
We fix a totally ordered group~$\Gamma$ (written multiplicatively) and denote by~$\Gamma^+$ its submonoid of elements less or equal to~$1$.
Moreover, we pick a cofinal element $\pi \in \Gamma^+_\star$, meaning that for all $\gamma \in \Gamma$, there is $n \in \NN$ such that $\pi^n < \gamma$.
In our application~$\pi$ will be the value of a pseudouniformizer in the analytic case and~$0$ in the discretely valued case.
We write $\Gamma[t]$ and $\Gamma^+[t]$ for the monoids obtained by adjoining one free element~$t$.
Explicitly $\Gamma[t] = \Gamma \oplus t^\NN$ and $\Gamma^+[t] = \Gamma^+ \oplus t^\NN$.

\subsubsection{Continuous valuations}
A \emph{valuation} of $\Gamma[t]_\star$ (or a similar monoid) is a monoid homomorphism
\[
 \lvert \cdot \rvert \colon \Gamma[t]_\star \longrightarrow V_\star,
\]
where~$V$ is a totally ordered valued group, sending~$1$ to~$1$ and~$0$ to~$0$.
It is \emph{continuous} if~$\lvert \pi \rvert$ is cofinal in $V_\star$.
Two valuations~$\lvert \cdot \rvert$ and~$\lvert \cdot \rvert'$ are \emph{equivalent} if $\lvert a \rvert \le \lvert b \rvert$ if and only if $\lvert a \rvert' \le \lvert b \rvert'$.

\subsubsection{The base}
The equivalence classes of valuations of~$\Gamma_\star$ correspond to quotients of~$\Gamma$ by a convex subgroup~$\Delta$.
The quotient represents a continuous valuation if and only if $\Delta \ne \Gamma$.
We define $S := \Spa(\Gamma_\star,\Gamma^+_\star)$ as the set of equivalence classes of continuous valuations of~$\Gamma_\star$.
We include the second component~$\Gamma^+_\star$ in the notation which calls for the condition that the valuation be less or equal to~$1$ on~$\Gamma^+_\star$ even though this condition is automatic.
The reason for this redundancy is that we want to be consistent with the notation for the spectrum of $\Gamma[t]_\star$.
For $s \in S$ we set $\Gamma_s := \Gamma/\Delta$, where~$\Delta$ is the convex subgroup of~$\Gamma$ corresponding to~$s$.

\subsubsection{Topology of the base}
We equip the set $S = \Spa(\Gamma_\star,\Gamma^+_\star)$ with the topology generated by the subsets $S_\gamma := \{s \in S \mid \lvert \gamma(s) \rvert \le 1\}$ for varying $\gamma \in \Gamma$.
For given~$\gamma$, the subset~$S_\gamma$ can also be described as $\Spa((\Gamma/\Delta_\gamma)_\star,(\Gamma^+[1/\gamma]/\Delta_\gamma)_\star)$, where~$\Delta_\gamma$ is the convex subgroup of~$\Gamma$ generated by~$\gamma$.
The convex subgroup~$\Delta_\gamma$ corresponds to a point $s \in S$ and we also use the notation $S_s = S_\gamma$.
Note that in case~$\Gamma$ has infinite ranke, there might be convex subgroups $\Delta \subseteq \Gamma$ that are not generated (as a convex subgroup) by a single element.
For these points $s \in S$, the subset $S_s = \Spa((\Gamma_s/\Delta_s)_\star,(\Gamma^+[\Delta_s]/\Delta_s)_\star)$ is not open.

\subsubsection{The unit interval}
We define the adic spectrum
\[
 \II_S := \Spa(\Gamma[t]_\star,\Gamma^+[t]_\star)
\]
as the set of equivalence classes of continuous valuations of $\Gamma[t]_\star$ that are less or equal to~$1$ on $\Gamma^+[t]_\star$.
In the following we will equip this set with a topology and a structure sheaf and call the resulting space the \emph{unit interval}.
Restriction of a valuation from $\Gamma[t]_\star$ to~$\Gamma_\star$ defines a natural map $f\:\II_S \to S$.

\subsubsection{Connection to rangers} \label{unitinterval}
A point $x\in \II_S$ is a continuous valuation of $\Gamma[t]_\star$ and we write $s = f(x)$.
If $\lvert t(x) \rvert = 0$,~$x$ factors as $\Gamma[t]_\star \to \Gamma_\star \to \Gamma_{s,\star}$ where the first map sends the ideal generated by~$t$ to~$0$.
Otherwise the valuation factors as $\Gamma[t]_\star \to \Gamma(t)_\star \to V_\star$ for some totally ordered abelian group~$V$ containing~$\Gamma_s$ and where~$\Gamma(t) = \Gamma \oplus t^\ZZ$.
In the latter case~$x$ corresponds to the ranger $(V,\lvert t(x) \rvert)$ in $\cR_{\Gamma_s}$ and equivalence of rangers translates to equivalence of continuous valuations.¸

Unfortunately there are two different zeros: one corresponds to the valuation that sends~$t$ to zero and the other one is the ranger zero corresponding to the unbounded valuation mapping~$t$ to an element that is smaller than any element of~$\Gamma_s$.
In order to distinguish the two, we call the former~$0$ and the latter $0^+$.
The ranger~$0^+$ only appears in the discrete case (when $\pi = 0$) as the corresponding valuation is not continuous in the analytic case (when $\pi > 0$).
Now, the fiber $\II_{S,s}$ can be identified with  $\{0\} \cup [0^+,1]$ for $[0^+,1] \subseteq \cR_s=\cR_{\Gamma_s}$ in the discrete case and with $\{0\} \cup (0^+,1] \subseteq \cR_s$ in the analytic case.

\subsubsection{Types of points}
The types of points on the unit interval~$\II_S$ can be classified using the connection to rangers described in \cref{unitinterval}.
Accordingly, we call a point on~$\II_S$ \emph{unbounded}, \emph{divisorial}, a \emph{cut point}, or \emph{infinitesimal} if the corrensponding ranger is of this type.
An unbounded point only exists in the discrete case and there is one more subtlety:
The point called~$0$ defined by the valuation $\Gamma[t]_\star \to \Gamma_\star$ mapping~$t$ to zero is not captured by rangers.
We call it a \emph{classical point} in analogy to classical points on rigid spaces.
%
%

%
%


\subsubsection{Topology of the unit interval} \label{top_unit_interval}
For $s \in S = \Spa(\Gamma_\star,\Gamma^+_\star)$ and $a,b \in \cR_{\Gamma_s} \cup \{0_s\}$, we define the \emph{subinterval} of~$\II_S$
\[
 [a,b]_s := \{x \in \II_S \mid a \le \lvert t(x) \rvert \le b\}.
\]
We call the interval $[a,b]_s$ \emph{rational} if~$S_s$ is open in~$S$,~$b$ is divisorial, and~$a$ is either divisorial or~$0$.
Suppose that~$\Gamma$ is divisible.
In this case we can write a rational interval $[a,b]_s$ with $a \ne 0$ as
\[
 [a,b]_s = \Spa(\Gamma_s(t),\Gamma_s^+[a/t,t/b])
\]
and
\[
 [0,b]_s = \Spa(\Gamma_s[t],\Gamma_s^+[t/b]).
\]
We equip~$\II_S$ with the toplology generated by all rational intervals.
Note that the structure map $\II_S \to S$ is continuous in this topology as the preimage of an open subset~$S_s$ is the rational interval $[0,1]_s$.

\begin{prop} \label{spectral}
 The unit interval~$\II_S$ is a spectral space.
\end{prop}

\begin{proof}
 The set of equivalence classes of continuous valuations of $\Gamma^+[t]_\star$ is a spectral space by \cite[Theorem~5.30]{Ray_valuationsmonoids}.
 Restricting to valuations with $\lvert \pi \rvert \ne 0$, we obtain the continuous valuations of $\Gamma[t]_\star$ and imposing the condition $\lvert t \rvert \le 1$ gives us the unit interval $\II_S = \Spa(\Gamma[t]_\star,\Gamma^+[t]_\star)$.
 Both conditions cut out pro-constructible subsets, hence~$\II_S$ is spectral.
\end{proof}

\subsubsection{The structure sheaves}
We define the presheaf $\cO_X$ on the rational subintervals of the unit interval $X=\II_S$ by sending $[a,b]_s$ with $a \ne 0$ to $\Gamma_s(t)_\star$ and $[0,b]_s$ to $\Gamma_s[t]_\star$.
Similarly,~$\cO_X^+$ sends $[a,b]_s$ for $a \ne 0$ to $\Gamma_s^+[a/t,t/b]_\star$ and $[0,b]_s$ to $\Gamma_s^+[t/]_\star$.
It is possible to show that~$\cO_X$ and~$\cO_X^+$ satisfy the sheaf condition, but we don't need this here.
We define $\cP_X$ and~$\cP_X^+$ as the sheafy quotients of monoids
\[
 \cP_X = \cO_X/\cO_X^{+\times}, \qquad \cP_X^+ = \cO_X^+/\cO_X^{+\times}.
\]
The letter~$\cP$ stands for piecewise monomial (or piecewise linear in additive notation) and~$\cP_X$ should be considered as the structure sheaf of~$X$ together with the subsheaf~$\cP_X^+$ of integral functions.
The stalk of~$\cP_X$ at a point $x \in \II_S \setminus \{0\}$ over $s \in S$ is of the form $\Gamma_{x,\star}$ for a totally ordered group~$\Gamma_x$ isomorphic (as a group) to~$\Gamma_s(t)$ (in case~$x$ is divisorial) or~$\Gamma_s$ (in all other cases).
Moreover, $(\Gamma_x,\lvert t(x) \rvert)$ represents the ranger corresponding to~$x$.
Similarly, the stalk of~$\cP_X^+$ at~$x$ identifies with $\Gamma_{x,\star}^+$, where $\Gamma_x^+ \subseteq \Gamma_x$ is the valuation monoid.
The stalk of~$\cP_X$ at~$0$ is $\Gamma[t]_\star$ and $\cP_{X,0}^+ = \Gamma^+ \cup t\Gamma[t]\cup \{0\}$.

\subsubsection{Generalizations} \label{genrem}
The unit interval is only a special instance of a space obtained by taking continuous valuations of a topological monoid~$M$ that are less or equal to~$1$ on a submonoid~$M^+$.
In case~$M$ is a group with discrete topology, the theory has been worked out in \cite{Goldzweig}.
Moreover, the spectrum of continuous valuations of topological monoids has been studied in \cite{Ray_valuationsmonoids} regarding spectrality.
Working out the general picture goes beyond the scope of this article but will be subject of an independent project.
It would be of particular interest to develop the theory of higher dimensional spectra, i.e. $\Spa(\Gamma[t_1,\ldots,t_n]_\star,\Gamma^+[t_1,\ldots,t_n]_\star$ that could play a role in piecewise linear geometry for higher dimensional adic spaces.

Even more generally, one can consider relative pl geometry over a base monoided valuation space. Such spaces can be useful, for example, in studying families of skeletons of fibers of an adic (or Berkovich) morphism $X\to S$. In particular, they should naturally show up in inductive constructions of deformational retractions of adic spaces along the lines of the theory of Hrushovski-Loeser from \cite{HL}.

\subsection{Valuative graphs}

\subsubsection{Monoided valuation spaces}
A \emph{monoided valuation space} is a triple $(X,\cP_X,\cP_X^+)$, where~$X$ is a topological space and $\cP_X^+ \subseteq \cP_X$ are sheaves of topological augmented monoids on~$X$ such that for each point $x \in X$, the stalk~$\cP^+_{X,x}$ is the preimage under the projection $\cP_{X,x} \to \cP_{X,x,\star}^\times$ of $V_\star$ for a sharp valuative submonoid $V \subseteq \cP_{X,x}^\times$ with $V^\gp = \cP_{X,x}^\times$.
The latter condition endows~$\cP_{X,x}^\times$ with the structure of a totally ordered abelian group and~$\cP_{X,x}^+ \cap \cP_{X,x}^\times$ is the monoid of nonnegative elements.

For two monoided valuation spaces $(Y,\cP_Y,\cP^+_Y)$ and $(X,\cP_X,\cP^+_X)$, a morphism from the former to the latter is a pair $(f,f^\sharp)$, where $f \colon Y \to X$ is a continuous morphism and $f^\sharp \colon f^{-1}\cP_X \to \cP_Y$ is a homomorphism of sheaves of topological augmented monoids such that for any $y \in Y$ the induced map of stalks $\cP_{X,f(y)} \to \cP_{Y,y}$ respects the total orders.

Note that $\II_S$ together with the structure sheaves~$\cP_{\II_S}$ and~$\cP^+_{\II_S}$ defines a monoided valuation space.
The same is true for all subintervals $[a,b]_s$.
If~$X$ is an adic curve over an affinoid field $(k,k^+)$, we obtain the monoided valuation space $X^\times := (X,\cP_X,\cP_X^+)$ for the sheaf of piecewise monomial functions $\cP_X = \cO_X/\cO_X^{+\times}$ and $\cP_X^+ =\cO_X^+/\cO_X^{+\times})$.
This is a well behaved space if~$k$ is algebraically closed.
Otherwise we rather work with the geometric variant $X^{\overline{\times}} = (X,\overline{\cP}_X,\overline{\cP}_X^+)$ obtained by descending~$\cP_{\overline{X}}$ and~$\cP^+_{\overline{X}}$ to~$X$.

\subsubsection{Valuative graphs}
Let $\Gamma$ be a totally ordered abelian group with valuation monoid~$\Gamma^+$ and set $S = \Spa(\Gamma_\star,\Gamma_\star^+)$.
A \emph{valuative graph}~$\Delta$ over a point $s \in S$ is a monoided valuation space that is locally isomorphic to $\II_{S_s}$.
Morphisms of valuative graphs are morphisms of monoided valuation spaces.
A point $\delta \in \Delta_s$ is called a \emph{leaf}, an \emph{edge point} or a \emph{vertex} if the number of nontrivial specializations of~$\delta$ is~$1$,~$2$, or at least~$3$, respectively.

Roughly speaking,  we can glue segments at their endpoints (or at any other divisorial point) producing graphs whose edges are segments and whose vertices are divisorial points.
Suppose, for instance, we are given two segments $[a_1,b_1]_s$ and $[a_2,b_2]_s$ and let $c_1 \in \Gamma_s$ be such that $a_1 \le c_1 \le b_1$.
Then
\[
	[c_1,c_1]_s = \Spa(\Gamma_s[t]_\star,\Gamma_s^+[t/c_1,c_1/t]_\star) \cong \Spa(\Gamma_{s,\star},\Gamma_{s,\star}^+) \cong [a_2,a_2]_s.
\]
So we can glue the two segments by identifying~$c_1$ with~$a_2$.
In this way we produce a graph with one vertex and three edges emanating from it.

\subsubsection{The affine line} \label{aff_line_monoid}
We obtain the affine line by glueing countably many intervals:
\[
	\AA^1_S = \bigcup_{n=0}^\infty [0,\alpha^n] = \bigcup_{i=1}^\infty \Spa(\Gamma[t]_\star,\Gamma^+[t/\alpha^n]_\star)
\]
for some $\alpha \in \Gamma$ such that $1/\alpha$ is cofinal, so~$\alpha^n$ gets arbitrarily big.
Alternatively, we can also write this as
\[
	\AA^1_S = [0,1] \cup \bigcup_{n=0}^\infty [\alpha^n,\alpha^{n+1}].
\]
In the discrete case the affine line is contained in the affinoid $\Spa(\Gamma[t]_\star,\Gamma_\star^+)$ as an open subset but $\Spa(\Gamma[t]_\star,\Gamma_\star^+)$ has one more point that is not contained in~$\AA^1_S$: the unbounded valuation~$\infty^-$ of $\Gamma[t]_\star$ sending~$t$ to a value that is bigger than any element of~$\Gamma$.

\subsubsection{The projective line} \label{proj_line_monoid}
The projective line $\PP^1_S$ is obtained by glueing two intervals:
\[
	\PP^1_S = [0,1] \cup [1,\infty],
\]
where $[1,\infty] = \Spa(\Gamma[1/t]_\star,\Gamma^+[1/t]_\star)$ is isomorphic to~$\II_S$ by sending~$t$ to $1/t$.
We get an open embedding $\AA^1_S \hookrightarrow \PP^1_S$ whose image consists of all points of~$\PP^1_S$ except for $\infty$ (the classical point of $[1,\infty]$) and in the discrete case $\infty^-$ (the unbounded ranger that is bigger than any element in~$\Gamma$).

\subsubsection{Intervals in the projective line} \label{proj_monoid_intervals}
The intervals $[a,b]_s$ (for $a,b \in \cR_{\Gamma_s} \cup \{0\}$, $\lvert a \rvert \le \lvert b \rvert \le 1$) described in \ref{top_unit_interval} naturally embed into~$\PP^1_S$.
Moreover, we can drop the restriction $\lvert b \rvert \le 1$ and consider intervals $[a,b]_s$ for all $a \le b$ in $\cR_{\Gamma_s} \cup \{0,\infty\}$.
In this terminology $\PP^1_S = [0,\infty]$.
We also define intervals of the form $(a,b)_s$, $(a,b]_s$, $[a,b)]_s$ by removing one or two endpoints from $[a,b]_s$.
We call an interval over $s \in S$ \emph{constructible} if it is constructible as a subset of the spectral space~$\PP^1_{S_s}$.
Explicitly, these are the intervals of the form $[a,b]_s$ where~$a$ is divisorial or~$0$ and~$b$ is divisorial or~$\infty$, $[a,b)_s$ where~$a$ is divisorial or~$0$ and~$b$ is divisorial, $(a,b]_s$ where~$a$ is divisorial and~$b$ is divisorial or~$\infty$, or $(a,b)$, where~$a$ and~$b$ are divisorial.

%
%
%

\subsection{Skeletons}

\subsubsection{Valuative skeletons}
Let~$X$ be a reduced curve over $S = \Spa(k,k^+)$ for an algebraically closed affinoid field $(k,k^+)$ and $s \in S$.
We consider the monoidal valuation space $X^\times := (X,\cP_X,\cP_X^+)$ for $\cP_X = \cO_X/\cO_X^{+\times}$ and $\cP_X^+ = \cO_X^+/\cO_X^{+\times}$.
A \emph{valuative skeleton} of~$X_s$ is a monoidal valuation subspace $\Delta \subseteq X^\times$ that is a quasi-compact valuative graph such that~$\Delta_s$ is a skeleton of~$X_s$ in the sense of section~\ref{locunifsec} and every point of~$\Delta$ specializes to a point in~$\Delta_s$.

If $(k,k^+)$ is not algebraically closed, we consider the base change~$\overline{X}$ to the algebraic closure $(\bar{k},\bar{k}^+)$ of $(k,k^+)$.
The sheaves~$\cP_{\overline{X}}$ and $\cP^+_{\overline{X}}$ are invariant under automorphisms of~$\overline{X}$ and thus descend to sheaves~$\overline{\cP}_X$ and~$\overline{\cP}_X^+$ on~$X$ (compare section~\ref{geomsec}).
In this way we obtain a monoided valuation space $X^{\overline{\times}} = (X,\overline{\cP}_X,\overline{\cP}_X^+)$.
A \emph{geometric valuative skeleton} of~$X_s$ is a monoidal valuation subspace $\Delta \subseteq X^{\overline{\times}}$ such that the preimage~$\overline{\Delta}$ of~$\Delta$ in~$\overline{X}$ is a valuative skeleton.

\begin{rem}
 We expect to be able to define in a clean way a notion of a skeleton of~$X$ (without referring to a particular $s \in S$).
 Its fibers over $s \in S$ should be skeletons of~$X_s$ in the sense of section~\ref{locunifsec}.
 However, for the purpose of this article, it is enough to define valuative skeletons with respect to a single~$X_s$.
\end{rem}

\subsubsection{Valuative graphs in $\PP^1_k$}
Suppose $(k,k^+)$ is algebraically closed with pseudouniformizer~$\varpi$ and value group~$\Gamma$.
We set $S = \Spa(k,k^+)$ and denote by $S^{\overline{\times}}$ the monoided valuation space $(S,\overline{\cP}_S,\overline{\cP}_S^+)$.
It is naturally a valuative complex being the spectrum of $(\bar{k},\bar{k}^+)$ where we consider~$\bar{k}$ and~$\bar{k}^+$ as multiplicative monoids.

We now study $X = \PP^1_k = \bigcup_n \Spa k\langle \varpi^n T \rangle \cup \{\infty\}$.
Let us fix a point $s \in S$.
For a classical point over~$s$ corresponding to $a \in k$ we consider the subset $[a,\infty]_s$ of all generalized Gauß valuations over $k(s)$ that are $(T-a)$-monomial plus the point at infinity.
Restricting $\overline{\cP}_X$ and $\overline{\cP}_X^+$ to $[a,\infty]_s$ we obtain a valuative graph that is naturally isomorphic to $\PP^1_{S_s^{\overline{\times}}}$ studied in section \ref{proj_line_monoid}.
We also call it the ray emanating from the classical point~$a$.
Via the identification with $\PP^1_{S_s^{\overline{\times}}}$ we obtain a naming scheme for the points on $[a,\infty]_s$: the point $a_r$ is the point in $[a,\infty]_s$ corresponding to the ranger $r \in \cR_{\lvert k(s)^\times \rvert}$ interpreted as a point of $\PP^1_{S_s^\times}$.
For two classical points $a, b$ the corresponding rays intersect in $[a_{\lvert (a-b)(s)\rvert},\infty]_s = [b_{\lvert (a-b)(s) \rvert},\infty]_s$.

For two monomial points $x_1$ and $x_2$ over~$s$ lying on a common ray $[a,\infty]_s$, it is now clear how to define $[x_1,x_2]_s$ (see section \ref{proj_monoid_intervals}) and one easily checks that this is independent of the choice of~$a$).
If $x_1$ and $x_2$ do not lie on a common ray (but both monomial and over~$s$), we can choose classical points $a_1$ and~$a_2$ such that $x_i \in [a_i, \infty]_s$.
Then $[a_1,\infty]_s \cap [a_2,\infty]_s = [x,\infty]_s$ for a divisorial point $x$ over~$s$.
Then we define $[x_1,x_2]$ as the valuative graph obtained by glueing $[x_1,x]$ to $[x_2,x]$.

Now let~$x$ be a unibranch point over~$s$.
It does not lie on any valuative subgraph of~$\PP^1_S$ but we can approximate it in the following way.
Let~$I$ be the set of all rational discs over~$s$ in~$\PP^1_S$ containing~$x$, which is a totally ordered set.
For $i \in I$ we denote by $a_i$ the Gauß point of the disc~$i$.
Then $\bigcup_{i \in I} [a_i,\infty]_s$ is a valuative graph and for all $j \in I$ the point~$x$ lies in the closure of $\bigcup_{i \ge j} [a_i,a_j]_s$.
Then we define the ray $[x,\infty] := \{x\} \cup \bigcup_{i \in I} [a_i,\infty]$, which is a monoided valuation space when equipped with the restrictions of~$\overline{\cP}_X$ and~$\overline{\cP}_X^+$.

\subsubsection{Split discs and annuli}
For $a \in k$, $s \in S = \Spa(k,k^+)$, and $r \in \cR_{\lvert k(s)^\times\rvert}$, we define the \emph{split $k$-disc} $\DD_s(a,r)$ as the pseudo-adic space defined by $\lvert T-a \rvert \le r$ and $\lvert \alpha \rvert \le 1$ for all $\alpha \in k(s)^+$.
In contrast to section \ref{k-split-discs} we do not restrict to the fiber over~$s$ but include all generalizations of the disc defined in section~\ref{k-split-discs}.
The disc $\DD_s(a,r)$ is a rational open if $r$ is divisorial and $S_s$ is open in $S$ (e.g. if~$k^+$ has finite rank).
Its Gauß point is the Gauß point of the fiber over~$s$.

Suppose that instead of~$r$ we are given two rangers $r_1 \le r_2$ in~$\cR_{\lvert k(s)^\times\rvert}$.
Then we can define the \emph{split $k$-annulus} $\AA_s(a;r_1,r_2)$ by the equations $r_1 \le \lvert T-a \rvert \le r_2$ and $\lvert \alpha \rvert \le 1$ for all $\alpha \in k(s)^+$ as before.
The annulus $\AA_s(a;r_1,r_2)$ is a rational open if $r_1$ and~$r_2$ are divisorial and $S_s$ is open in $S$.
The annulus has two boundary points~$x_1$ and~$x_2$ characterized by being $(T-a)$-monomial with $\lvert (T-a)(x_i) \rvert = r_i$.

\subsubsection{Skeletons of split discs and annuli} \label{skeleton_disc_annulus}

The skeleton of the disc $\DD_s(a,r)$ with Gauß point~$\eta$ is the valuative graph given by $\Spa(\overline{k(\eta)},\overline{k(\eta)}^+)$ (overline meaning algebraic closure) together with the sheaves of pl-functions and integral pl functions.
The skeleton of the annulus $A_s(a;r_1,r_2)$ is the segment connecting~$x_1$ and~$x_2$.
More precisely, we consider the \emph{geometric} interval $[r_1,r_2]_{\bar{s}}$ in $\PP^1_{S_{\bar{s}}}$.
We define the map $[r_1,r_2]_{\bar{s}} \to \PP^1_S$ by sending a valuation $\overline{\Gamma}[T]_\star \to V_\star$ to the $(T-a)$-monomial Gauß valuation $k[T] \to V_\star$ with the same value for~$T-a$.
It is straightforward to check that this map is injective and identifies the structure sheaf of $[r_1,r_2]_s$ with the restriction of $\overline{\cP}_{\PP^1_{S_s}}$.
So it defines a monoidal subspace of $\PP^1_S$ that is a valuative graph.

The above construction also defines skeletons for $h$-discs and $h$-annuli.
In fact, the skeletons are invariant under $h$-isomorphisms by Proposition~\ref{h-iso-monoidal} below.

\begin{prop} \label{h-iso-monoidal}
	Let $(k,k^+)$ be an affinoid field, $X \to S = \Spa(k,k^+)$ a smooth curve, and $s \in S$ a point.
	We consider an $h$-isomorphism of the form $f \colon (Y,Y_s) \to (X,X_s)$.
	Let~$X_{(s)}$ and~$Y_{(s)}$ be the subspaces of~$X$ and~$Y$ consisting of all generalizations of points in~$X_s$ and~$Y_s$, respectively.
	Then~$f$ induces an isomorphism of monoided valuation spaces $f_{(s)} \colon (Y_{(s)},\cP_{Y_{(s)}},\cP^+_{Y_{(s)}}) \to (X_0,\cP_{X_{(s)}},\cP^+_{X_{(s)}})$.
\end{prop}

\begin{proof}
	Let us first show that $f_{(s)}$ is a homeomorphism.
	Since~$f$ is étale and thus open, it suffices to check bijectivity.
	This comes down to showing that for every $x \in X_s$, the base change of~$f$ to $\Spa(k(x),k(x)^+)$ induces a bijection when restricted to the generalizations of the unique preimage of~$x$.
	It is given by a map of Huber pairs $(k(x),k(x)^+) \to (A,A^+)$ for an étale $k(x)$-algebra~$A$.
	Since~$f$ is an $h$-isomorphism, there is a unique maximal ideal~$\mathfrak{m}$ of~$A^+$ over~$x$ and the resulting homomorphism $k(x)^+ \to A^+_\mathfrak{m}$ is an unramified extension of valuation rings.
	The points of~$\Spec A^+_\mathfrak{m}$ correspond to the generalizations of the unique preimage~$y$ of~$x$ in~$Y$ and we get the claimed bijection.

	In order to show that the homomorphism of monoids $f^{-1}\cP_X \to \cP_Y$ is an isomorphism on~$Y_{(s)}$, we look at stalks.
	At points~$y$ with generic support the stalk of~$\cP_Y$ takes the form $\Gamma_{y,\star}$ and similarly for $x = f(y)$.
	Since~$f$ is strongly étale at~$y$, it induces an isomorphism $\Gamma_y \cong \Gamma_x$, and hence an isomorphism of stalks of~$\cP_Y$ and~$\cP_X$.
	Let us now look at a point $y \in Y_{(s)}$ whose support is non-generic and denote by $x \in X_{(s)}$ its image.
	The stalks~$\cO_{X,x}$ and~$\cO_{Y,y}$ are discrete valuation rings.
	Choosing a uniformizer~$T$ of~$\cO_{X,x}$, we obtain a representation $\cO_{X,x} = (T^{\NN} \times \cO_{X,x}^\times)_\star$ as multiplicative monoids.
	Accordingly,
	\[
		\cP_{X,x} = \cO_{X,x}/\cO_{X,x}^{+\times} = (T^{\NN} \times \cO_{X,x}^\times/(\cO_{X,x}^{+\times})_\star = \Gamma_x[T]_\star.
	\]
	Since~$f$ is étale at~$y$, $T$ is also a uniformizer for $\cO_{Y,y}$.
	Moreover, $\Gamma_y = \Gamma_x$ as~$f$ is even strongly étale.
	Computing~$\cP_{Y,y}$ in the same way as above, we conclude that $\cP_{X,x} = \cP_{Y,y}$.
	
	Finally, it is easy to check that the isomorphism $\cP_{X,x} \cong \cP_{Y,y}$ identifies $\cP_{X,x}^+$ with $\cP_{Y,y}^+$, so~$f$ induces an isomorphism of monoided valuation spaces $(Y_{(s)},\cP_{Y_{(s)}},\cP^+_{Y_{(s)}}) \to (X_{(s)},\cP_{X_{(s)}},\cP_{X_{(s)}}^+)$.
\end{proof}

\subsubsection{Standard skeletons}
Let~$\DD$ be a constructible open $h$-disc, i.e. its Gauß point~$p$ is divisorial.
We define its standard valuative skeleton as the space $\Spa(\overline{k(p)},\overline{k(p)}^+)$ (as a valuative graph, i.e. only with the monoidal structure).
It is isomorphic to $\Spa(\overline{\Gamma}_\star,\overline{\Gamma}_\star^+)$.
If $\DD \to \DD_s(a,r)$ is an $h$-isomorphism to a split $k$-disc, it induces an isomorphism of standard skeletons by the definition of $h$-isomorphisms.

For an $h$-annulus~$\AA$ with end points $p_1$ and~$p_2$ the picture is similar but we need to resort to Proposition~\ref{h-iso-monoidal} in order to construct the skeleton:
If $\AA \to \AA_s(a;r_1,r_2)$ is an $h$-isomorphism and $I \subseteq \AA_s(a;r_1,r_2)$ is the standard skeleton from section~\ref{skeleton_disc_annulus}, we get a unique lift $I \hookrightarrow \AA$ using Propositon~\ref{h-iso-monoidal}.
We call~$I$ the standard skeleton of~$\AA$; it is the unique segment connecting the endpoints~$p_1$ and~$p_2$ and we also write $I = [p_1,p_2]_s$.

\begin{prop} \label{morvaldiscs}
	Let $S= \Spa(k,k^+)$ for an affinoid field~$(k,k^+)$ and $f : (X',X_s') \to (X,X_s)$ a quasi-finite morphism of $h$-discs over a point $s \in S$.
	Let~$p' \in X_s'$ be the Gauß point and set $p = f(p')$.
	\begin{enumerate}
		\item Suppose that $I' \subseteq X'$ is an interval of the form $I' = [a',p']_s$.
			Then $f(I')$ with the induced structure sheaf from~$X$ is an interval over~$s$ and the map $I \to f(I)$ is a morphism of valuative graphs and a homeomorphism on the underlying topological spaces.
		\item For any interval $I \subseteq X$ of the form $[a,p]_s$, the preimage $f^{-1}(I)$ is the quasi-compact valuative graph obtainded by glueing $[a'_i,p']_s$ for all preimages~$a'_i$ of~$a$.
	\end{enumerate}
\end{prop}

\begin{proof}
	This is a valuative graph version of Theorem~\ref{mordiscs}.
	After reducing to the case where~$k$ is algebraically closed, we consider the morphism $f^{\times} \colon X'^\times:=(X',\cO_{X'}/\cO_{X'}^{+\times},\cO^+_{X'}/\cO_{X'}^{+\times}) \to (X,\cO_X/\cO_X^{+\times},\cO^+_X,\cO_X^{+\times}) =: X^\times$ of monoided valuation spaces.
	Proposition~\ref{h-iso-monoidal} tells us that~the subspace $X^\times_{(s)} \subseteq X^\times$ of all generalizations of~$X_s$ is isomorphic to a $k$-split disc, and similarly for~$X'^\times$.
	Now we can read off from the proof of Theorem~\ref{mordiscs} that~$f^{\times}$ is given by $\lvert u \prod_i (T-c_i) \rvert$ for certain $c_i \in k(s)^+$, which is clearly piecewise monoidal (or piecewise linear in additive language) and induces a homeomorphism from~$I'$ to $f(I') = f^\times(I')$.
	This shows~(1) and the same reasoning applies to~(2).
\end{proof}

\begin{prop} \label{morvalann}
 	Let $S= \Spa(k,k^+)$ for an affinoid field~$(k,k^+)$ and $f : (X',X_s') \to (X,X_s)$ a quasi-finite morphism from an $h$-annulus to an $h$-disc or annulus over a point $s \in S$.
 	Let $I' \subseteq X'$ be the standard skeleton.
 	Then $f(I')$ is a quasi-compact valuative graph and $I' \to f(I')$ is piecewise monomial (i.e. a morphism of valuative graphs).
\end{prop}

\begin{proof}
 	This follows from Lemma~\ref{invannulus} as in the proof of Corollary~\ref{annulimor}: the morphism~$\lvert f \rvert$ is locally of the form $\lvert f \rvert = r \lvert t^n\rvert$.
\end{proof}

\begin{theor} \label{vallocunif}
 	Let $S = \Spa(k,k^+)$ for an affinoid field $(k,k^+)$ and $X \to S$ a quasi-compact smooth curve.
 	For every $s \in S$ there exists a valuative skeleton of~$X_s$.
 	More precisely, for every geometric triangulation $\{p_1,\ldots,p_n\}$ of~$X_s$, there is a unique geometric valuative skeleton containing $\{p_1,\ldots,p_n\}$ whose vertices are contained in $\{p_1,\ldots,p_n\}$.
\end{theor}

\begin{proof}
	We first reduce to~$k$ being algebraically closed.
 	For a given triangulation~$\Delta$ we consider the annular compnents $A_1,\ldots,A_r$ of $X_s \setminus \Delta$.
 	The standard skeleton of~$A_i$ is of the form $[p_i^+,q_i^-]_s$ with infinitesimal points~$p_i^+$ and~$q_i^-$.
 	It is contained in the open interval $[p_i,q_i]_s$ where~$p_i$ and~$q_i$ are the divisorial generalizations of~$p_i^+$ and~$q_i^-$ in the affine line.
 	In~$X_s$ the point~$p_i^+$ has a unique generalization $p_i' \in \Delta$.
 	Since the valuative graph $\Spa(k(p_i'),k(p_i')^+)$ is isomorphic to $\Spa(\Gamma_{s,\star},\Gamma^+_{s,\star}) = \Spa(k(p_i),k(p_i)^+)$ and~$\{p_i\}$ is open in $[p_i,q_i]_s$, we can glue $[p_i,q_i]_s$ to~$p_i'$.
 	Similarly, we glue $[p_i,q_i]_s$ to $q_i' \in \Delta$ at the other end point.
 	As a result, we obtain a valuative skeleton of~$X_s$ as required.
 	Note that triangulations exist by Theorem~\ref{curvesekeletonth}
\end{proof}

\subsubsection{Images and preimages of valuative graphs}
	As a consequence of Theorem~\ref{vallocunif} we can prove a valuative enhancement of Theorem~\ref{morth}.

\begin{theor}
	Let $S = \Spa(k,k^+)$ for an affinoid field $(k,k^+)$ and $f \colon X' \to X$ a quasi-finite morphism of quasi-compact smooth curves.
	\begin{enumerate}
	 \item For any quasi-compact valuative graph $\Delta' \subseteq X'$, its image $f(\Delta')$ is a valuative graph and the morphism $\Delta' \to f(\Delta')$ is piecewise monomial.
	 \item For any quasi-compact valuative graph $\Delta \subseteq X$, its preimage $f^{-1}(\Delta)$ is a valuative graph.
	\end{enumerate}
\end{theor}

\begin{proof}
 First we decompose~$\Delta'$ into subgraphs that are glued from intervals over a fixed $s \in S$.
 Using Theorem~\ref{curvesekeletonth} we construct triangulations $V' \subseteq X'$ and $V \subseteq X$ over~$s$ such that $f^{-1}(V) \subseteq V'$.
 Let~$\Delta_{V'}$ and~$\Delta_{V}$ be the respective valuative skeletons of~$X'_s$ and~$X_s$ corresponding to~$V$ and~$V'$ via Theorem~\ref{vallocunif}.
 We can decompose the valuative graph~$\Delta'$ into edges that are either contained in an edge of~$\Delta_{V'}$ or have trivial intersection with~$\Delta_{V'}$.
 Applying Proposition~\ref{morvaldiscs} and Proposition~\ref{morvalann} to the individual edges gives the first result.
 Assertion~$(2)$ is proved in the same way.
\end{proof}


\appendix
\section{Quasi-graphs}\label{A}
To study adic curves we heavily used $\Gamma$-graphs whose metric is expressed in terms of an ordered group $\Gamma$. In fact, there exists a more general purely topological notion of generalized topological graphs, that we call quasi-graphs. Its segments correspond just to ordered sets rather than ordered groups. We call the related objects quasi-lines, quasi-trees, etc. These are natural notions which seem to be missing in the literature, so we feel free to choose a terminology of our own.

\subsection{Definitions}

\subsubsection{Quasi-trees}\label{quasitreesec}
A point $x$ of a topological space $T$ {\em separates} points $y,z\in T$ if they lie in two different connected components of $T\setminus\{x\}$. We say that a topological space $T$ is a {\em quasi-tree} if it satisfies the following conditions:

\begin{itemize}
\item[(0)] $T$ is $T_0$.
\item[(1)] Any two points $a,b\in T$ are contained in a connected subspace $[a,b]$ such that any element of $(a,b)=[a,b]\setminus\{a,b\}$ separates $a$ and $b$.
\item[(2)] If $x,y,z$ are such that $[x,y]\cap[x,z]=\{x\}$, then $[y,z]=[y,x]\cup[x,z]$.
\end{itemize}

\begin{rem}
(i) Axiom (1) is the main one. It implies that $[a,b]$ is the minimal connected subspace of $T$ containing $a$ and $b$. In particular, it is unique and the notation $[a,b]$ is non-ambiguous. We call it the {\em quasi-segment} connecting $a$ and $b$. Note also that $(a,b)$ is precisely the set of all elements separating $a$ and $b$.

(ii) In the Hausdorff case we are used to, (1) implies (2). However, it is important to allow topological spaces with non-trivial specializations ($x\prec y$ if any neighborhood of $x$ contains $y$) and this forced us to add Axiom (2). In fact, it is only needed to exclude the situation when one doubles the endpoint in an interval in an asymmetric way: for example, glue two copies of $[0,1]$ along $(0,1]$ and define the topology so that $0'$ is the specialization of $0$. Then $[0,0']$ and $[0,1]$ are quasi-intervals, but $0'\cup[0,1]$ is not. In a sense, the novelty of Axiom 2 is that it forbids loops of a degenerate form, which cannot be detected by Axiom (1) because $(0',0)=\emptyset$. In its turn, axiom (0) is not covered by (2) only in the case when $T$ consists of two points -- a sort of an even more degenerate loop, when $x$ specializes $y$ and $y$ specializes $x$.
\end{rem}

\subsubsection{Quasi-graphs and branches}
A {\em quasi-graph} is a topological space $G$ which is locally homeomorphic to a quasi-tree. A {\em branch} of a quasi-graph $G$ at a point $x$ is a germ of quasi-segments $[x,y]$ in $G$. By our convention $[x,y]$ and $[x,z]$ define the same germ if $[x,y]\cap[x,z]$ contains a quasi-segment $[x,t]$ (with $x\neq t$). In particular, we require that $t\neq x$ even in the case when $x$ is open in the quasi-segments. We call a point $x\in G$ a {\em leaf}, an {\em edge point} or a {\em vertex} if the number of branches at x is 1, 2 or at least 3, respectively.

\begin{lem}\label{branchlem}
If $G$ is a quasi-tree, then the branches at $x$ are in a one-to-one correspondence with the connected components of $G \setminus\{x\}$.
\end{lem}
\begin{proof}
If $[x,y]$ and $[x,z]$ define different germs, then $[x,y]\cap[x,z]=\{x\}$ and hence $[y,z]=[y,x]\cup[x,z]$. In particular, $x$ separates $y$ and $z$. The inverse implication is even simpler.
\end{proof}

\subsubsection{Quasi-intervals}
By a {\em quasi-interval} we mean a quasi-tree $I$ which does not contain vertices. For example, subspaces of the form $(a,b)$ and $[a,b]$ of a quasi-tree are quasi-intervals. Any leaf of $I$ is called an {\em endpoint}, and we will later show in Corollary~\ref{nonlinelem} that, as one might expect, $I$ contains at most two endpoints. We say that a quasi-interval $I$ is a {\em quasi-line} or a {\em quasi-segment} if it contains no endpoints or contains two endpoints, respectively. In the latter case we will use the notation $I=[a,b]$ or $I=[-\infty,\infty]$.

\subsubsection{The order}
By an orientation of a quasi-segment $I$ we mean an order of the endpoints. It induced a total order on the whole quasi-segment, with the endpoints being the minimal and maximal elements: $c\le d$ if $c\in [-\infty,d]$. The transitivity  easily follows using that for any $c\in I$ the quasi-intervals $[-\infty,c]$ and $[c,\infty]$ cover $I$ and have only $c$ in the intersection. Any quasi-interval is a union of quasi-segments, and once one of them is oriented there is a unique way to orient all the rest compatibly, and the induced orders agree. In particular, there are two ways to orient any quasi-interval (which is not a singleton), and they induce opposite total orders of the underlying sets.

\subsubsection{Ordered sets}
In the sequel, by an ordered set $S$ we always mean a totally ordered set. We freely use the interval notation $[a,b]$, $(c,d]$, etc. The notation like $(-\infty,a]$ are used for unbounded intervals given by a single inequality. For shortness, we use this notation also when the minimal element exists. The {\em standard topology} of $S$ is the topology whose base is formed by the intervals $(a,b)$, $(-\infty,a)$ and $(b,\infty)$. In particular, $S$ is disconnected if and only if one of the following holds: (1) there exists an {\em immediate successive pair}, that is, a pair $x<y$ such that $(x,y)$ is empty, or (2) there exists a {\em cut}, that is, a decomposition $S=A\coprod B$, where $B>A$, $B$ has no minimum and $A$ has no maximum.

\subsection{Quasi-lines}
Now, let us discuss the topology of quasi-intervals and its relation to the order. To avoid dealing with cases, we will discuss only quasi-lines.

\subsubsection{Non-closed points}
Since quasi-lines can be non-Hausdorff it makes sense to classify specialization relations. It turns out that they are of the simplest possible form only: specialization chains are of length one -- from open points to closed points. In fact, such a topology is not a bug, but the feature needed to provide connectedness when the ordered set has immediate successors.

\begin{lem}\label{qlinelem}
If $L$ is an oriented quasi-line and $x\in L$ is a point, then one of the following possibilities hold:

(i) If $x$ is closed, then intervals $(-\infty,x)$ and $(x,\infty)$ are open and their closures are $(-\infty,x]$ and $[x,\infty)$.

(ii) If $x$ is not closed, then $(-\infty,x)$ and $(x,\infty)$ are closed and so $x$ is open. In addition, the closure of $x$ is of the form $\ox=\{x,x^+,x^-\}$, where $x^+$ and $x^-$ are closed points, which are immediate successor and predecessor of $x$.

(iii) For any predecessor-successor pair in $L$ one point in the pair is open and the other is closed.
\end{lem}
\begin{proof}
Since $(-\infty,x)$ is clopen in $L\setminus\{x\}$ but not in $L$, we have that precisely one interval is open and one interval is closed in the pair $(-\infty,x)$ and $(-\infty,x]$. The same is true for $(x,\infty)$ and $[x,\infty)$ and using the connectedness once again we see that only two possibilities are possible: both $(-\infty,x]$ and $[x,\infty)$ are closed, or both are open. In the first case we obviously have that all conditions of (i) are satisfied. In the second case we have that $x$ is open and using the connectedness of $L$ we obtain that the closure $\ox$ of $x$ contains points $x^-$, $x^+$ such that $x^-<x$ and $x^+>x$. The two point subspace $\{x,x^+\}$ of $L$ is not discrete, and hence removing a point from $L$ cannot separate this pair. It follows that $x^+$ is the immediate successor of $x$, and in the same way $x^-$ is the immediate predecessor of $x$. Finally, claim (iii) follows from (i) and (ii).
\end{proof}

\begin{cor}\label{intervalcor}
Any bounded from below (resp. above) subset $S$ of an oriented quasi-line $L$ possesses an infimum (resp. supremum) in $L$. In particular, any convex subset $I$ is a quasi-interval. In addition, if the infimum (resp. supremum) $x$ of $S$ is not contained in $S$, then $x$ is a closed point.
\end{cor}
\begin{proof}
Suppose a bounded from below set $S$ contains no minimal element. Then the set $I=\cup_{s\in S}[s,\infty)$ is open because for each $s\in S$ there exists $t\in S$ with $t<s$ and $[t,\infty)$ contains an open neighborhood of $[s,\infty)$ -- either $[t,\infty)$ itself or $(t,\infty)$. Since $I$ is not closed, its closure contains an additional point, which is easily seen to be the infimum of both $S$ and $I$.
\end{proof}

\subsubsection{Quasi-compactness}
As another consequence we obtain the expected quasi-compactness result.

\begin{lem}\label{qcomlem}
Any quasi-segment $I=[a,b]$ is quasi-compact.
\end{lem}
\begin{proof}
Assuming that $\{U_j\}_{j\in J}$ is an open cover of $I$, consider the subspace $I_0$ of all points $x\in I$ such that the interval $[a,x]$ has a covering by finitely many $U_j$'s. Clearly, $I_0$ is open and convex and hence either $I_0=I$ and we win, or $I_0=[a,y)$ for a closed point~$y$ by Corollary~\ref{intervalcor}. In the latter case choose $U_l$ containing $y$.
Then $U_l$ contains a point $z$ with $z<y$. Consequently, $U_l$ and a finite cover of $[a,z]$ yield a finite cover of $[a,y]$, contradicting that $y\notin I_0$.
\end{proof}

\subsubsection{Non-linear graphs}
Now we can provide a criterion for a quasi-tree to be not quasi-linear. We say that points $a,b,c$ in a quasi-tree $T$ are {\em non-collinear} if neither of them separates the other two.

\begin{lem}\label{triplelem}
Assume that $T$ is a quasi-tree and $a,b,c\in T$ three distinct points. Then $a,b,c$ are not collinear if and only if there exists a vertex $x$ such that these points lie in three distinct components of $T\setminus\{x\}$. If this is the case, then such a vertex is unique.
\end{lem}
\begin{proof}
The inverse implication is obvious. Conversely, assume that the points are not collinear and consider the intersection $I=[a,b]\cap[a,c]$. Clearly, $I\neq\{a\}$, as otherwise $a\in [b,c]$. Therefore, $I$ is a non-singleton quasi-interval which is strictly smaller than $[a,b]$ and $[a,c]$. If $I=[a,x]$ and it is immediate from Axiom (2) of the quasi-trees that $x$ separates all three elements $a$, $b$ and $c$. The uniqueness is also clear, as $\{x\}=[a,b]\cap[a,c]\cap[b,c]$.

Assume now that $I$ contains no maximal point. By Corollary~\ref{intervalcor} the quasi-interval~$I$ admits suprema $x\in [a,b]$ and $y\in [a,c]$ that do not lie in $I$. Hence, $x,y$ lie in the closure of $I$ and $[a,x)=I=[a,y)$. If $[x,y]=\{x,y\}$ (i.e. one of the points specializes the other) then using that $y\notin[a,x]$ and $x\notin[a,y]$ we obtain a contradiction to  Axiom (2). Thus, $(x,y)\neq\emptyset$ and there exists a point $z\in T$ which separates $x$ and $y$. For any point $d\in I$ the closure of $I_d=I\setminus[a,d]$ contains both $x$ and $y$, hence $z\in I_d$. But by our assumption $\cap_dI_d=\emptyset$, a contradiction.
\end{proof}

\begin{cor}\label{nonlinelem}
A quasi-tree $T$ is not a quasi-interval if and only if there exist three non-collinear points $a,b,c\in T$. In particular, any quasi-interval contains at most two leaves.
\end{cor}

\subsection{Quasi-segment completions of ordered sets}
Finally, we will describe how a quasi-line can be characterized in terms of ordered sets, and how an ordered set can be completed to a quasi-segment in a minimalist way.

\subsubsection{A characterization of quasi-intervals}
The topology of a quasi-interval $S$ is determined by the order (with respect to one of the orientations) and the set $S_0$ of closed points. Indeed, the base of topology is formed by quasi-intervals of the form $(a,b)$, where $a$ (resp. $b$) is either a closed point or $-\infty$ (resp. $\infty$).

\begin{lem}\label{orderlem}
Assigning to an oriented quasi-interval $\calS$ the underlying ordered set $S$ with the subset of closed points gives rise to a one-to-one correspondence between isomorphisms classes of quasi-intervals and pairs $(S,S_0)$, where $S$ is an ordered set which does not have cuts and $S_0$ is a subset such that any pair $x,y\in S_0$ is separated by an element of $S$ and if a non-maximal (resp. non-minimal) element $x$ lies in $S\setminus S_0$, then it possesses an immediate successor $x^+$ (resp. predecessor $x^-$) and the latter lies in $S_0$.
\end{lem}
\begin{proof}
Given an ordered set $S$ with a subset $S_0$ we construct a topological space $\calS$ as explained above the lemma. Clearly, $\calS$ is a quasi-interval if and only if it is $T_0$ and connected, so we should check that the latter happens if and only if the conditions about cuts and $S_0$ are satisfied. Clearly, the topology is $T_0$ if and only if any pair $x,y\in S\setminus S_0$ is separated by an element of $S_0$ as otherwise $x$ and $y$ are indistinguishable for the topology. So, we assume in the sequel that $\calS$ is $T_0$.

If $x<y$ lie in $S_0$ and form an immediate succession pair, then $\calS=(-\infty,x)\coprod(y,\infty)$ is disconnected. If $S=S^-\coprod S^+$ is a cut, then $S_-=\cup_{x\in S^-}(-\infty,x)$ and by the $T_0$ hypothesis, one can only take quasi-intervals with $x\in S^-\cap S_0$. Thus, $S^-$ is open, and similarly $S^+$ is open, proving that $\calS$ is disconnected. A similar argument shows that if $x$ is non-maximal, lies in $S\setminus S_0$ and has no immediate successor, then both $(-\infty,x]$ and $(x,\infty)$ are open.

Conversely, assume that the conditions on $S$ and $S_0$ are satisfies, but $\calS$ is disconnected, say, $\calS=X\coprod Y$ with non-empty open $X$ and $Y$. Choose $x\in X$ and $y\in Y$ and assume without restriction of generality that $x<y$. Let $T$ be the set of points $z\in[x,y]$ such that $[x,z]\subseteq X$. Clearly, $T$ is open. By the assumption on cuts, either $T=[x,t]$ or $T=[x,t)$ for some $t\in[x,y]$. In the first case, $t\notin S_0$, hence $t^+$ exists an lies in $Y$. But by definition of the topology, $t$ is a generization of $t^+$, hence $t\in Y$, a contradiction. If $T=[x,t)$, then $t\in Y$, but any of its neighborhoods intersects $T$ because otherwise we must have that $t^-$ exists and lies in $S_0$ yielding a pair $t^-,t\in S_0$ which is an immediate succession pair.
\end{proof}

\subsubsection{Minimal completion}
The lemma provides the following canonical way to complete an ordered set to a quasi-interval so that the induced topology on $S$ is the order topology. All points of $S$ are closed in the completion and open points are added to connect immediate succession pairs.

\begin{cor}
Given an unbounded from both sides ordered set $S$, let $\oS$ be the naturally ordered set obtained by adding to $S$ the set of cuts of $S$ and the set $S^0$ of intermediate points for each immediate successive pair $x<y$. Then $(\oS,S^0)$ corresponds to a quasi-line and $S\into\oS$ is the universal embedding of $S$ into a quasi-line which induces the canonical topology on $S$ (for any other such embedding $S\into\oS'$ there exists a unique order preserving retraction $\oS'\onto\oS$).
\end{cor}

\subsubsection{Ranger completion}
Rangers (see \S\ref{rangersubsec}) provide another, in a sense dual, canonical way to complete an unbounded ordered set $S$ to a quasi-line. This time we only add closed points -- cut points and infinitesimal points, immediate succession pairs $x<y$ are saturated with a single infinitesimal element $x^+=y^-$ (of course, this never happens for ordered groups). The induced topology on $S$ is discrete.

\begin{cor}\label{rangercomplete}
Let $S$ be an unbounded ordered set. Then the set of bounded rangers $\cR^b(S)$ is a quasi-line and the set of all rangers $\cR(S)$ is a quasi-segment. In addition, $S\into\cR^b(S)$ is the universal embedding of $S$ into a quasi-line which induces the discrete topology on $S$.
\end{cor}

\bibliographystyle{amsalpha}
\bibliography{wild_multiplicity}

\end{document}